\documentclass{article}
\usepackage{amsmath,amssymb,amsthm}
\usepackage[colorlinks=true]{hyperref}
\usepackage{graphicx}  
\usepackage[english]{babel}   
\usepackage{amsmath,amssymb}  
\usepackage{amsthm}           
\usepackage{float,xcolor}

\newtheorem{theorem}{Theorem}[section]      
\newtheorem{lemma}[theorem]{Lemma}          

\newtheorem{remark}[theorem]{Remark}          

\theoremstyle{definition}                   

\nonstopmode    

\newcommand{\Ltwo}{L^{2}}

\begin{document}

\title{Viscosity in  Error Upper Bound for a Consistent Splitting Scheme of the Navier–Stokes Equations}
\author{M Nader Alhomsi\footnote{Department of Mathematics,  Central Michigan University, 
Mount Pleasant, MI 48858. Email: Alhom1n@cmich.edu}\,, Jiahong Wu\footnote{Department of Mathematics,  University of Notre Dame, 
Notre Dame, IN 46556. Email: jwu29@nd.edu}\,
and Xiaoming Zheng\footnote{Department of Mathematics, Central Michigan University, Mount Pleasant, MI 48858. Email: zheng1x@cmich.edu}
}
\date{}
\maketitle

\begin{abstract}
This paper investigates the role of viscosity in the error upper bounds of a consistent splitting scheme for the Navier--Stokes equations proposed by Huang and Shen \cite{HuangShen2023}. In their original analysis the viscosity is fixed to unity. By following and extending their proof methodology while keeping the viscosity $\nu$ symbolic, we obtain an $H^{1}$ velocity error bound that contains negative powers of $\nu$, indicating that the scheme is not robust as $\nu\to 0$. To establish this bound we refine a theorem in \cite{LiuLiuPego2007} on the constant in the Stokes pressure estimate, which is crucial to the error analysis. A targeted numerical experiment based on a perturbation of the Kovasznay flow corroborates this analytical prediction: the scheme of \cite{HuangShen2023} blows up at high Reynolds number, and a comparison with a fully implicit Newton solver and with the time-dependent Stokes counterpart of the same scheme localizes the failure to the explicit treatment of the convection term.
\end{abstract}

\section{Introduction}
\label{sec_intro}
To solve the Navier-Stokes equations
\begin{align}
    \partial_t \mathbf{u} + \mathbf{u} \cdot \nabla \mathbf{u} &= -\nabla p + \nu \Delta \mathbf{u} + \mathbf{f}, \label{ns1} \\
    \nabla \cdot \mathbf{u} &= 0, \label{ns2}
\end{align}
with velocity ${\bf u}$, pressure $p$, viscosity $\nu$, and external force ${\bf f}$ in a rectangular domain $\Omega \subset \mathbb{R}^d$ ($d=2,3$) with the no-slip boundary condition ${\bf u}|_{\partial\Omega}=0$, 
Huang and Shen introduce the following second-order splitting scheme in \cite{HuangShen2023}
\begin{subequations}\label{HSscheme}
\begin{align}
\frac{(2k+1)\,\mathbf{\bar{u}}^{n+1} - 4k\,\mathbf{\bar{u}}^{n} + (2k-1)\,\mathbf{\bar{u}}^{n-1}}{2\delta t} &=
\nu \Delta \!\left( k\,\mathbf{\bar{u}}^{n+1} - (k-1)\,\mathbf{\bar{u}}^{n} \right)
- \hat{\mathbf{u}}^{\,n} \cdot \nabla \hat{\mathbf{u}}^{\,n}
- \nabla \hat{p}^{\,n} + \mathbf{f}^{\,n+k}, 
\label{HSscheme-a}
\\[6pt]
\bigl(\nabla p^{n+1}, \nabla q \bigr) 
&= \bigl( \mathbf{f}^{\,n+1} 
- \bar{\mathbf{u}}^{\,n+1} \cdot \nabla \bar{\mathbf{u}}^{\,n+1} 
- \nu \nabla \times \nabla \times \bar{\mathbf{u}}^{\,n+1}, \nabla q \bigr),
\label{HSscheme-b}
 \\
\frac{r^{\,n+1} - r^{\,n}}{\delta t} 
&= \frac{r^{\,n+1}}{E\!\bigl(\bar{\mathbf{u}}^{\,n+1}\bigr) + \bar{C}}  \Bigl( - \nu \,\|\nabla \bar{\mathbf{u}}^{\,n+1}\|^2 
+ \bigl( \mathbf{f}^{\,n+1}, \bar{\mathbf{u}}^{\,n+1} \bigr) \Bigr), 
\label{HSscheme-c}
\\
\xi^{\,n+1} &= \frac{r^{\,n+1}} {E(\bar{\mathbf{u}}^{\,n+1}) + \bar{C}}, 
\quad \eta^{\,n+1} = 1 - \bigl(1 - \xi^{\,n+1}\bigr)^2, 
\label{HSscheme-d}
 \\
\mathbf{u}^{n+1} &= \eta^{\,n+1}\,\bar{\mathbf{u}}^{\,n+1}, 
\label{HSscheme-e}
\end{align}
\end{subequations}
where $\delta t$ is the time step size, $n\in\mathbb{N}\cup\{0\}$ is the time step, 
$t^n=n\delta t$, 
$k$ is a positive integer,  ${\bf u}^n$ and $p^n$ are approximations of ${\bf u}(t^n)$ and $p(t^n)$ respectively, 
$\hat{\mathbf{u}}^{\,n} = (k+1)\,\mathbf{u}^{n} - k\,\mathbf{u}^{n-1}$, 
$\hat{p}^{\,n} = (k+1)\,p^{n} - k\,p^{n-1}$. 
The quantity $r^n$ is an approximation of $r(t^n)$, called the generalized scalar auxiliary variable (GSAV) in \cite{HuangShen2023}, whose definition is
\setcounter{equation}{4}
\begin{align}
\label{def_SAV}
r(t) \triangleq E({\bf u}(t))+\bar{C},
\end{align}
where $E({\bf u}(t))=\frac12\|{\bf u}(t)\|^2$ is the kinetic energy and $\bar{C}$ is a constant.  
It is well-known that the solution ${\bf u}$ of \eqref{ns1} and \eqref{ns2} with the zero boundary condition satisfies
\begin{equation}
\frac{dr}{dt} = -\nu \| \nabla {\bf u} \|^2 + ({\bf f}, {\bf u}).
\end{equation}
As shown in Theorem\,6 of \cite{HuangShen2023}( Theorem\,\ref{thm:HS2023_Thm6} in this work), $\bar{C}\ge \max\{2\delta t^2 C^2_f, 2C^2_f ,1\}$ is chosen to guarantee a weak stability, where $C_f$ is a constant only dependent on ${\bf f}$.

As in \cite{HuangShen2023}, we denote the Lebesgue space of squared integrable functions as $L^2(\Omega)$ and the Soboblev spaces $H^n(\Omega)$ and their norms as $\|\cdot \|$ and $\|\cdot\|_n$, respectively. In the above,  $(\cdot, \cdot)$ is the inner product in the $L^2(\Omega)$ space. The vectorized spaces are denoted as  
${\bf L}^2(\Omega)=(L^2(\Omega))^d$, 
${\bf H}^n(\Omega)=(H^n(\Omega))^d$,
and 
${\bf H}^n_0(\Omega)=(H^n_0(\Omega))^d$,
where $d$ is the space dimension and taken as $2$ in this work.

The scheme \eqref{HSscheme-a}-\eqref{HSscheme-e} combines two key ingredients. First, it employs a general second-order backward difference formula (BDF) temporal discretization (implicit on diffusion but explicit on convection), where a suitable choice of 
$k$ enhances the stability of the scheme compared with the classic BDF2 method ($k=1$). For example, Theorem 5 in \cite{HuangShen2023} proves that the schemes \eqref{3.5} and \eqref{3.6} for the time dependent Stokes equations are unconditionally stable when $k\ge 5$ and $\nu=1$.
Second, it introduces the GSAV $r$, representing the kinetic energy. This formulation allows the nonlinear term in the full Navier-Stokes equations to be treated explicitly while still achieving unconditional stability in a weak sense (Theorem 6 in \cite{HuangShen2023}, see a revised version Theorem\,\ref{thm:HS2023_Thm6} in this work). More importantly, \cite{HuangShen2023} is the first rigorous stability and error analysis for
any second-order consistent splitting scheme for the Navier-Stokes equations.
Note that Huang and Shen further improve the scheme in \cite{HuangShen2025} by eliminating the GSAV formulation and reducing the parameter value to $k=3$.

The purpose of this work is to study the robustness of this scheme.  Following \cite{John2021}, a numerical scheme for the Navier--Stokes equations is called \emph{robust} if the upper bound on the $L^{2}$ error of the velocity contains no negative power of the viscosity, and \emph{non-robust} otherwise.
The theoretical results of \cite{HuangShen2023} are stated under the assumption $\nu=1$, so they cannot decide whether the scheme is robust in this sense.  We therefore revisit Theorems~5 and~7 of \cite{HuangShen2023} and re-derive them for arbitrary $\nu>0$, retaining $\nu$ symbolically throughout in order to expose the precise viscosity-dependence of the error bound.  This requires a careful tracking of every constant generated by Young's inequalities, Sobolev embeddings, and the Stokes-pressure splitting.  In particular, we refine a theorem from \cite{LiuLiuPego2007} concerning the constant in the Stokes-pressure estimate; the refined statement and its proof are given in the Appendix.

The remainder of this paper is organized as follows.  Section~\ref{sec_preliminaries} introduces the notation, several classical estimates, and two auxiliary lemmas used in the proofs of the main results.  Section~\ref{sec_stability} presents a refined stability analysis for the time-dependent Stokes equations at arbitrary viscosity.  Section~\ref{sec_error_analysis} extends the error analysis to the Navier--Stokes equations at arbitrary viscosity.  Section~\ref{sec_numerical} presents a numerical experiment that exposes the blow-up behavior of the scheme as the viscosity tends to zero, in addition to a temporal convergence test.  Finally, Section~\ref{sec-discussion} concludes the non-robustness of the scheme in light of these results.

For readability and ease of direct comparison, this work adopts the notation and proof structures established in \cite{HuangShen2023}.  Specifically, the proofs of Theorem~\ref{Theorem 5-New} and Theorem~\ref{Theorem 7-new} follow the same methodology as Theorems~5 and 7 of \cite{HuangShen2023}, respectively, but are written out in a self-contained manner so that the reader does not need to consult \cite{HuangShen2023} alongside the present paper.

\section{Preliminaries}
\label{sec_preliminaries}
As in \cite{HuangShen2023},  the following estimates are often used in this work:
$\forall {\bf u}, {\bf v}, {\bf w}\in {\bf H}^1_0(\Omega)\cap {\bf H}^2(\Omega)$, 
\begin{align}
|({\bf u}\cdot\nabla {\bf v}, {\bf w})|
& \le C\,
\|{\bf u}\|^{1/2}\,\|\nabla {\bf u}\|^{1/2} \,
\|\nabla {\bf v}\|^{1/2}\,\| {\bf v}\|_2^{1/2} \|{\bf w}\|,
\quad d=2, 
\label{convection_estimate}
\end{align}
\begin{align}
|({\bf u}\cdot\nabla {\bf v}, {\bf w})|
&\le
\left\{
\begin{aligned}
&C \|\mathbf u\|_{1}\|\mathbf v\|_{1}\|\mathbf w\|_{1},\\
&C \|\mathbf u\|_{2}\|\mathbf v\|_{0}\|\mathbf w\|_{1},\\
&C \|\mathbf u\|_{2}\|\mathbf v\|_{1}\|\mathbf w\|_{0},\\
&C \|\mathbf u\|_{1}\|\mathbf v\|_{2}\|\mathbf w\|_{0},\\
&C \|\mathbf u\|_{0}\|\mathbf v\|_{2}\|\mathbf w\|_{1},
\end{aligned}
\right.
\qquad d\le 4,
\label{Teman_estimate}
\end{align}
\begin{align}
\|{\bf u}\|_2
&\le C \|\Delta {\bf u} \|
\quad \text{(elliptic regularity)},
\label{elliptic_regularity}
\end{align}
\begin{align}
\|{\bf u} \|
&\le C \| \nabla {\bf u} \|
\quad \text{(Poincar\'{e}/Sobolev inequality)},
\label{Sobole-inequality}
\end{align}
\begin{align}
\|\nabla p_s({\bf u})\|^2 
&\le \left(\frac{1}{2}+\varepsilon\right) \|\Delta {\bf u}\|^2
+ C_S(\varepsilon) \| \nabla {\bf u} \|^2,
\quad \text{ where } C_S(\varepsilon)=\frac{C}{\varepsilon^3} \text{ when } \varepsilon\to 0+.
\label{lemma4-Liu}
\end{align}
The inequalities \eqref{convection_estimate} and  \eqref{Teman_estimate} are 
directly taken from (2.2) and (2.4) of \cite{HuangShen2023}, respectively.
The upper bound in  \eqref{lemma4-Liu} is proved in Theorem\,\ref{refinedTheorem1LiuLiuPego} in Appendix, a refinement of Theorem 1 in \cite{LiuLiuPego2007}. The symbol    
$p_s({\bf u})$ is the  Stokes pressure for any ${\bf u} \in {\bf H}^{2}(\Omega)$ defined by
$\nabla p_s({\bf u})
= \bigl( \Delta \mathcal{P} - \mathcal{P}\Delta \bigr){\bf u}$, 
where $\mathcal{P}$ is the Leray--Helmholtz projection onto divergence-free vector fields with zero normal component on
$\partial\Omega$. In the Helmholtz decomposition
${\bf u} = \mathcal{P}{\bf u} + \nabla \phi$, the quantity $\phi \in H^{1}(\Omega)$ satisfies 
$\bigl( {\bf u} - \nabla \phi, \nabla q \bigr) = 0$, 
$\forall q \in H^{1}(\Omega)$. 
The form $C_S(\varepsilon)=C/\varepsilon^3$ holds only when $\varepsilon\to 0+$. 
This work, however, focuses primarily on the case of small $\varepsilon$. 
\begin{remark}
All the constants $C$ from \eqref{convection_estimate} to \eqref{lemma4-Liu} 
are only dependent on $\Omega$, thus taken as a single value (their maximum). Throughout this work, the symbol $C$ denotes this fixed constant, rather than a generic constant that may vary from step to step.
\end{remark}

Denote the exact solution of \eqref{ns1} and \eqref{ns2} at time $t^i=i \delta t$ as $({\bf u}(t^i), p(t^i))$, the numerical solution of \eqref{HSscheme-a} to \eqref{HSscheme-e} as $({\bf u}^i, p^i)$, 
and the errors at $t^i$  as 
\begin{align}
{\bf e}^i = {\bf u}^i - {\bf u}(t^i),
\quad
{\bf \bar e}^i = {\bf \bar u}^i - {\bf u}(t^i),
\quad
{e}^i_p = p^i - p(t^i),
\quad
\text{where} \quad {\bf u}^i = \eta^i \bar{\bf u}^i.
\end{align}
The following symbols are used in the proofs below to simplify the notations, $\forall i\in\mathbb{N}$,
\begin{equation}
\begin{array}{lll}
\tilde{\bf u}^i \triangleq (k+1) \bar{\bf u}^i - k \bar{\bf u}^{i-1}, 
&
\hat{\bf u}^i \triangleq (k+1) {\bf u}^i - k {\bf u}^{i-1},
&
\hat{\bf u}(t^i) \triangleq (k+1) {\bf u}(t^i) - k {\bf u}(t^{i-1}),
\\
\tilde{\bf e}^i \triangleq (k+1) \bar{\bf e}^i - k \bar{\bf e}^{i-1}, 
&
\hat{\bf e}^i \triangleq (k+1) {\bf e}^i - k {\bf e}^{i-1}, 
&
\hat{e}^i_p \triangleq (k+1)e^i_p - k e^{i-1}_p,
\\
\hat{p}^i \triangleq (k+1) p^i - k p^{i-1}.
& &
\end{array}
\label{notations}
\end{equation}
A simple way to remember these symbols is that the tilde quantities are associated with barred variables, whereas the hat quantities correspond to unbarred variables.

The following two lemmas  are used in this work. 
\begin{lemma}
\label{lemma:algebraic_identity}
For any $x,y,z\in\mathbb{R}$ and $k\ge\tfrac12$, the following identity holds:
\begin{equation}\label{eq:alg-identity}
\begin{aligned}
&[ (2k+1)x - 4ky + (2k-1)z ] \cdot 
 [ (k+1)x - k y ]\\
= &  A\,(|x|^2 - |y|^2)
  + \lvert Bx - D y\rvert^2 
  - \lvert B y - D z\rvert^2
  + E\,\lvert x - y\rvert^2
  - F\,\lvert y - z\rvert^2
  + G\,\lvert x - 2y + z\rvert^2,
\end{aligned}
\end{equation}
where
\[
A = \tfrac1{2k},\quad
B = \tfrac{(k+1)\sqrt{2k-1}}{\sqrt{2k}},\quad
D = \sqrt{\tfrac{k(2k-1)}2},\quad
E = \tfrac{2k+3}2,\quad
F = \tfrac{2k-1}2,\quad
G = \tfrac{(k+1)(2k-1)}2.
\]
In particular, for $k\ge\tfrac12$ one has $E>F\ge 0$.
\end{lemma}
\begin{lemma}
\label{lemma-k-values}
Let  $\theta>0$ and suppose the positive real number $k$ satisfies 
\begin{align}
\frac{k-1}{k}\ge \frac{\varepsilon_2}{2}
+\frac{1}{4\varepsilon_2} + \epsilon \,\theta
\label{k-condition}
\end{align}
for some positive values of $\varepsilon_2$ and $\varepsilon$. 
The infimum of $k$ over $\varepsilon_2$ and $\varepsilon$ is $k_{\inf}=\frac{1}{1 - \frac{1}{\sqrt{2}}}\approx 3.4$ when $\varepsilon_2=\frac{1}{\sqrt{2}}$ and $\varepsilon\to 0+$. 
Thus, the minimum integer of $k$ is $4$ and it can be attained when $\varepsilon_2=\frac{1}{\sqrt{2}}$ and $0<\varepsilon \le \frac{1}{\theta} \left( \frac{3}{4} - \frac{1}{\sqrt{2}} \right)$. 
\end{lemma}

\begin{proof}
\,\newline
Denote $g(\varepsilon_2, \varepsilon)\triangleq \frac{\varepsilon_2}{2}+\frac{1}{4\varepsilon_2} + \epsilon \theta$.  
Then \eqref{k-condition} becomes 
$1-\frac{1}{k} \ge g$.
Because $1-\frac{1}{k}$ is increasing on $k$, the infimum of $k$ occurs when $g$ takes the infimum over positive values of $\epsilon_2$ and $\epsilon$. This happens when $\varepsilon_2=\frac{1}{\sqrt{2}}$ and $\varepsilon\to 0+$. In this case, $k_{\inf}=\frac{1}{1 - \frac{1}{\sqrt{2}}}$. 
To obtain $k=4$, we let $k=4$ and $\varepsilon_2=\frac{1}{\sqrt{2}}$ in $1-\frac{1}{k}\ge g$,  which yields $0<\varepsilon \le \frac{1}{\theta}\left( \frac{3}{4} - \frac{1}{\sqrt{2}} \right)$.

\end{proof}

\section{Refined Stability for the time dependent Stokes Problem}
\label{sec_stability}
When the scheme \eqref{HSscheme-a} and \eqref{HSscheme-b} are applied to the time dependent Stokes equations (the equations \eqref{ns1} and \eqref{ns2} without the convection term ${\bf u}\cdot \nabla {\bf u}$), we get the following scheme 
\begin{align}
\frac{(2k+1)\mathbf{u}^{\,n+1} 
- 4k\mathbf{u}^n+(2k-1)\mathbf{u}^{\,n-1} }{2\delta t}
&= \nu\Delta\bigl(k\mathbf{u}^{\,n+1}-(k-1)\mathbf{u}^n\bigr) -\nabla\bigl((k+1)p^n-kp^{n-1}\bigr),
\label{3.5}
\\
\bigl(\nabla p^{\,n+1},\nabla q\bigr)
&= -\nu\bigl(\nabla\times\nabla\times\mathbf{u}^{\,n+1},\nabla q\bigr),
\qquad\forall q\in H^1(\Omega).
\label{3.6}
\end{align}

\begin{theorem}[Updated Theorem 5 of \cite{HuangShen2023}; Stability for general $\nu$ and $k\ge 4$ for time dependent Stokes equations; spatial dimension=2 or 3]
\label{Theorem 5-New}
Assume $\mathbf{u}^1$ is computed using a first order consistent splitting scheme \cite{guermond2003new}. 
Let $\nu_{\max}>0$. 
Then for any integer $k\ge4$ and viscosity $0<\nu<\nu_{\max}$, the scheme $(3.5)-(3.6)$ is unconditionally stable. In particular, for all $n$ with $n\delta t\le T$, there holds
\begin{equation}
\begin{aligned}
  \|\nabla\mathbf{u}^{n+1}\|^2
  \;+\;\nu\,\delta t\sum_{i=0}^n\|\Delta\mathbf{u}^{i+1}\|^2
  \;+\;\frac{\delta t}{\nu}
    \sum_{i=0}^n
    \|\nabla p^{i+1} \|^2
  \;\le\; C_{3.7},
\end{aligned}
\label{3.7-new}  
\end{equation}
where $C_{3.7}$ depends on the initial condition, the final time $T$, and the value of $k$, and is monotonically increasing with respect to the maximum viscosity $\nu_{\max}$, but is independent of $\delta t$ and $n$. The explicit form of $C_{3.7}$ is given in \eqref{def_C3.7}.

\end{theorem}

\begin{remark}
\label{remark1}
The original stability result (Theorem 5 in \cite{HuangShen2023}) requires $k\ge 5$ and is presented only for $\nu=1$.
Reducing $k$ from 5 to 4 improves the truncation errors, as illustrated in (3.1) and (3.2) of \cite{HuangShen2023}: 
\begin{align}
\frac{(2k+1)\,\phi(t^{n+1}) - 4k\,\phi(t^{n}) + (2k-1)\,\phi(t^{n-1})}{2\,\delta t}
&= \phi'(t^{n+k})
+ \frac{1 - 3k^{2}}{6}\,\phi'''(t^{n+k})\,\delta t^{2}
+ O(\delta t^{3}),
\\
k\,\phi(t^{n+1}) - (k-1)\,\phi(t^n)
&= \phi(t^{n+k}) - \frac{k(k-1)}{2}\,\phi''(t^{n+k})\,\delta t^2 + O(\delta t^3).
\end{align}
The updated proof follows the methodology in \cite{HuangShen2023}, but incorporates an additional tuning parameter, $\varepsilon_2$, within the pressure term of \eqref{3.13-new}. Furthermore, it utilizes Lemma~\ref{lemma:algebraic_identity} to accommodate general values of $k$. 
\end{remark}

\begin{proof}
\,\newline
The proof follows the methodology of Theorem~5 in \cite{HuangShen2023}. Recall 
$\hat{\bf u}^j\triangleq(k+1){\bf u}^j-k{\bf u}^{j-1}$ and
$\hat p^j\triangleq(k+1)p^j-kp^{j-1}$ from \eqref{notations}.

Taking the $L^2$ inner product of \eqref{3.5} with
$-\Delta\hat{\bf u}^{n+1}$ and multiplying both sides by $2\delta t$ yields
\begin{align}
&\bigl((2k+1){\bf u}^{n+1}-4k{\bf u}^{n}+(2k-1){\bf u}^{n-1},\,-\Delta\hat{\bf u}^{n+1}\bigr) \notag\\
&\quad+2\delta t\,\bigl(-\nu\Delta(k{\bf u}^{n+1}-(k-1){\bf u}^n),\,-\Delta\hat{\bf u}^{n+1}\bigr)
=2\delta t\,\bigl(-\nabla\hat p^n,\,-\Delta\hat{\bf u}^{n+1}\bigr).
\label{3.7a-new}
\end{align}
The three inner products are analyzed in turn.

As for the viscous term, using
$(-\nu\Delta{\bf v},-\Delta{\bf w})=\nu(\Delta{\bf v},\Delta{\bf w})$
together with the elementary identity
\begin{equation}
(ka-(k-1)b)\bigl((k+1)a-kb\bigr)
=\frac{k-1}{k}\bigl((k+1)a-kb\bigr)^2+\frac{1}{k}a^2+\frac{1}{2}\bigl(a^2-b^2+(a-b)^2\bigr),
\label{3.7b-new}
\end{equation}
applied componentwise with $a=\Delta{\bf u}^{n+1}$ and $b=\Delta{\bf u}^n$
(noting $ka-(k-1)b=\Delta(k{\bf u}^{n+1}-(k-1){\bf u}^n)$ and
$(k+1)a-kb=\Delta\hat{\bf u}^{n+1}$), one obtains
\begin{align}
&\bigl(-\nu\Delta(k{\bf u}^{n+1}-(k-1){\bf u}^n),\,-\Delta\hat{\bf u}^{n+1}\bigr) \notag\\
=\;& \nu\frac{k-1}{k}\|\Delta\hat{\bf u}^{n+1}\|^2
+\frac{\nu}{k}\|\Delta{\bf u}^{n+1}\|^2
+\frac{\nu}{2}\Bigl(\|\Delta{\bf u}^{n+1}\|^2-\|\Delta{\bf u}^n\|^2+\|\Delta{\bf u}^{n+1}-\Delta{\bf u}^n\|^2\Bigr).
\label{3.8}
\end{align}

For the pressure term, Cauchy--Schwarz and Young's inequality with parameter
$\varepsilon_2>0$ give
\begin{align}
\bigl|(-\nabla \hat p^n, -\Delta \hat{\bf u}^{n+1})\bigr|
\le \|\nabla\hat p^n\|\cdot\|\Delta\hat{\bf u}^{n+1}\|
\le \frac{1}{2\nu\varepsilon_2}\|\nabla\hat p^n\|^2
+\frac{\nu\varepsilon_2}{2}\|\Delta\hat{\bf u}^{n+1}\|^2.
\label{3.9-new}
\end{align}
We next bound $\|\nabla\hat p^n\|^2$ by combining the scheme \eqref{3.6} with
the Stokes pressure identity (Theorem~1 of
\cite{LiuLiuPego2007}):
\begin{align}
(\nabla p_s({\bf u}), \nabla q) = - (\nabla\times\nabla\times{\bf u}, \nabla q),
\qquad\forall q\in H^1(\Omega),
\label{3.10}
\end{align}
Note \eqref{3.6} yields
\begin{equation}
(\nabla \hat{p}^{n},\nabla q)=
-\nu (\nabla\times\nabla\times \hat{\bf u}^{n},\nabla q), \qquad \forall q\in H^1(\Omega).
\label{3.11}
\end{equation}
Thus, 
$(\nabla \hat{p}^{n},\nabla q)=\nu(\nabla p_s(\hat{\bf u}^n),\nabla q)$ for every
$q\in H^1(\Omega)$. Taking $q=\hat{p}^{n}$ leads to
\begin{equation}
\|\nabla\hat p^n\|\le\nu\,\|\nabla p_s(\hat{\bf u}^n)\|.
\label{3.12-new}
\end{equation}
The following estimate can be obtained by applying the refined Stokes pressure estimate \eqref{lemma4-Liu} on \eqref{3.12-new} with the parameter
$\varepsilon$ replaced by $2\varepsilon$:
\begin{align}
\frac{1}{2\nu\varepsilon_2}\|\nabla\hat p^n\|^2
\le\frac{\nu}{\varepsilon_2}\Bigl(\frac14+\varepsilon\Bigr)\|\Delta\hat{\bf u}^n\|^2
+\frac{C\nu}{16\varepsilon^3\varepsilon_2}\|\nabla\hat{\bf u}^n\|^2,
\label{3.13-new}
\end{align}
where $\varepsilon>0$ is arbitrary.

For the temporal-difference term, integration by parts (zero Dirichlet
boundary conditions) gives$({\bf v},-\Delta{\bf w})$
$= (\nabla{\bf v},\nabla{\bf w})$.  ;
applying Lemma~\ref{lemma:algebraic_identity} componentwise with
$x=\nabla{\bf u}^{n+1}$, $y=\nabla{\bf u}^n$, $z=\nabla{\bf u}^{n-1}$ yields,
for arbitrary $k\ge\tfrac12$,
\begin{align}
&((2k+1){\bf u}^{n+1} - 4k{\bf u}^n + (2k-1){\bf u}^{n-1},
-\Delta ((k+1){\bf u}^{n+1} - k{\bf u}^n)) \notag\\
=&A(\|\nabla{\bf u}^{n+1}\|^2 - \|\nabla{\bf u}^{n}\|^2)
+\|B\nabla{\bf u}^{n+1} - D\nabla{\bf u}^n\|^2
- \|B\nabla{\bf u}^{n} - D\nabla{\bf u}^{n-1}\|^2 \notag\\
& + E\|\nabla{\bf u}^{n+1} - \nabla{\bf u}^n\|^2
- F\|\nabla{\bf u}^{n} - \nabla{\bf u}^{n-1}\|^2
+G \|\nabla{\bf u}^{n+1} - 2\nabla{\bf u}^n + \nabla{\bf u}^{n-1} \|^2,
 \label{3.16_New}
\end{align}
where the constants $A=\frac{1}{2k}$, $B$, $D$, $E$, $F$, $G$ are those
specified in Lemma~\ref{lemma:algebraic_identity}; in particular $A>0$ and
$E>F\ge 0$ for $k\ge\tfrac12$.

Substituting \eqref{3.16_New} and \eqref{3.8}  into the
left-hand side of \eqref{3.7a-new}, and bounding the right-hand side via
\eqref{3.9-new} combined with \eqref{3.13-new} (note the 
telescoping term $\nu\delta t(\|\Delta{\bf u}^{n+1}\|^2-\|\Delta{\bf u}^n\|^2)$
is dropped in \cite{HuangShen2023}), we obtain
\begin{align}
&A(\|\nabla{\bf u}^{n+1}\|^2 - \|\nabla{\bf u}^{n}\|^2)
+\|B\nabla{\bf u}^{n+1} - D\nabla{\bf u}^n\|^2
- \|B\nabla{\bf u}^{n} - D\nabla{\bf u}^{n-1}\|^2
\notag\\
& + E\|\nabla{\bf u}^{n+1} - \nabla{\bf u}^n\|^2
- F\|\nabla{\bf u}^{n} - \nabla{\bf u}^{n-1}\|^2
+\nu\delta t (\|\Delta{\bf u}^{n+1}\|^2
-\|\Delta{\bf u}^{n}\|^2 )
\notag\\
&+ 2\nu\delta t \left(\frac{k-1}{k}\right)
\|\Delta \hat{\bf u}^{n+1}\|^2
+ \frac{2\nu\delta t}{k} \|\Delta{\bf u}^{n+1}\|^2
\notag\\
 \le\quad& \frac{2\nu\delta t}{\varepsilon_2}\left(\frac{1}{4} + \varepsilon  \right) \|\Delta \hat{\bf u}^n \|^2
+ \frac{C\nu\delta t}{8 \varepsilon^3 \varepsilon_2}
\|\nabla \hat{\bf u}^n \|^2
+\nu \varepsilon_2\delta t \|\Delta \hat{\bf u}^{n+1}\|^2.
\label{3.17-new}
\end{align}
The nonnegative term
$G\|\nabla{\bf u}^{n+1}-2\nabla{\bf u}^n+\nabla{\bf u}^{n-1}\|^2$ from
\eqref{3.16_New} and the nonnegative term
$\frac{\nu}{2}\|\Delta{\bf u}^{n+1}-\Delta{\bf u}^n\|^2$ from \eqref{3.8} have
been dropped on the LHS.

Summing \eqref{3.17-new} over $n=1,\dots,m$ for any $m\le T/\delta t-1$ dropping unnecessary terms yields
\begin{align}
&A\|\nabla {\bf u}^{m+1}\|^2
+ \|B\nabla {\bf u}^{m+1} - D\nabla {\bf u}^{m}\|^2
+ E\|\nabla {\bf u}^{m+1} - \nabla {\bf u}^{m}\|^2 + \frac{2\nu \delta t}{k}\sum_{n=1}^{m}\|\Delta {\bf u}^{n+1}\|^2
+ 2\nu \delta t \|\Delta {\bf u}^{m+1}\|^2 \notag \\
& + 2\nu \delta t
\Bigg[
\frac{k-1}{k} - \frac{\varepsilon_2}{2}
- \frac{1}{\varepsilon_2}
\left(\frac{1}{4}+\varepsilon\right)
\Bigg]
\sum_{n=1}^{m}\|\Delta\hat{\bf u}^{n+1}\|^2
\le  \frac{C\nu \delta t}{8\varepsilon^3 \varepsilon_2}
\sum_{n=1}^{m}\|\nabla \hat{\bf u}^n\|^2
+ M_0,
\label{3.17-summed}
\end{align}
where $M_0$ contains all initial terms. This yields the stability condition:
    \begin{align}
    \frac{k-1}{k} - \frac{\varepsilon_2}{2}
    \ge \frac{1}{\varepsilon_2} \left(\frac{1}{4}+\varepsilon\right).
    \label{eq:stability-condition}
    \end{align}
By Lemma\,\ref{lemma-k-values}, \eqref{eq:stability-condition} is satisfied
with the smallest integer $k=4$ when
\begin{align}
\varepsilon_2 =\frac{1}{\sqrt{2}}, \quad
\varepsilon   = \frac{1}{\sqrt{2}} \left( \frac{3}{4} - \frac{1}{\sqrt{2}}\right) \approx 0.03.
\label{def_epsandeps2}
\end{align}
Under these conditions, we discard some positive terms on the left of \eqref{3.17-summed} to get
\begin{align}
A\|\nabla {\bf u}^{m+1}\|^2
+ \frac{2\nu \delta t}{k}\sum_{n=1}^{m}\|\Delta {\bf u}^{n+1}\|^2
\le   \frac{C\nu \delta t}{ 8\varepsilon^3 \varepsilon_2}
\sum_{n=1}^{m}\|\nabla\hat{\bf u}^n \|^2
+ M_0.
\label{3.17-reduced}
\end{align}
With $A=\frac{1}{2k}$ from Lemma\,\ref{lemma:algebraic_identity}, multiplying
\eqref{3.17-reduced} by~$2k$ and using
$\|\nabla\hat{\bf u}^n\|^2\le 2(k+1)^2\bigl(\|\nabla{\bf u}^n\|^2+\|\nabla{\bf u}^{n-1}\|^2\bigr)$
gives
\begin{align}
\|\nabla{\bf u}^{m+1}\|^2 + 4\nu\delta t \sum_{n=0}^m
\|\Delta {\bf u}^{n+1}\|^2
&\le
\frac{Ck\nu \delta t}{4\varepsilon^3 \varepsilon_2 }
 \sum_{n=1}^m\|\nabla \hat{\bf u}^n\|^2 +C_I\notag\\
&\le
\frac{Ck(k+1)^2\nu \delta t}{2\varepsilon^3 \varepsilon_2 }
 \sum_{n=0}^m\|\nabla{\bf u}^n\|^2 +C_{II},
\label{3.17-after-new}
\end{align}
where $C_I$ and $C_{II}$ depend only on $k$, $\varepsilon$, $\varepsilon_2$,
and the initial data ${\bf u}^0,{\bf u}^1$ (the latter being computed by a
first-order consistent splitting scheme \cite{guermond2003new}, so that
$\|\nabla{\bf u}^1\|$ and $\|\Delta{\bf u}^1\|$ are bounded uniformly).

Applying the discrete Gronwall inequality (Lemma~2 of \cite{HuangShen2023},
which states: if nonnegative sequences satisfy
$a^{m+1}\le \alpha\,\delta t\sum_{n=0}^{m}a^n+\beta$ with
$\alpha\,\delta t<1$, then $a^{m+1}\le\beta\exp(\alpha T/(1-\alpha\delta t))$
for $m\delta t\le T$) to \eqref{3.17-after-new} with
$a^n=\|\nabla{\bf u}^n\|^2$,
$\alpha=\frac{Ck(k+1)^2\nu_{\max}}{2\varepsilon^3\varepsilon_2}$, and
$\beta=C_{II}$ yields, for all $n$ with $n\delta t\le T$,
\begin{align}
  \|\nabla\mathbf{u}^{n+1}\|^2
  + 4 \nu\,\delta t\sum_{i=0}^n\|\Delta\mathbf{u}^{i+1}\|^2
  \le  \tilde{C}_{3.7}
  \triangleq C_{II} \exp\left(
\frac{Ck(k+1)^2\nu_{\max} T}{2\varepsilon^3\varepsilon_2 }
 \right).
\label{def_tildeC3.7}
\end{align}

It remains to obtain the pressure bound. 
The identity \eqref{3.10} is applied to $p_s({\bf u}^{n+1})$ to earn  
$(\nabla p_s({\bf u}^{n+1}), \nabla q) = -\nu (\nabla\times \nabla\times {\bf u}^{n+1}, \nabla q)$.
Along with \eqref{3.6}, one obtains $(\nabla p^{n+1},\nabla q)=\nu (\nabla p_s({\bf u}^{n+1}), \nabla q)$. Taking $q=p^{n+1}$ results in 
$\| \nabla p^{n+1} \| \le \nu \| \nabla p_s({\bf u}^{n+1})\|$. 
Applying \eqref{lemma4-Liu}, we obtain 
$$
\frac{1}{\nu^2} \| \nabla p^{n+1}\|^2 \le \left(\frac{1}{2}+\epsilon\right) \| \Delta {\bf u}^{n+1} \|^2 
+ \frac{C}{\varepsilon^3} \|\nabla {\bf u}^{n+1} \|^2.
$$
Multiplying by $\nu\delta t$ and summing over $i=0,\dots,n$, then invoking
\eqref{def_tildeC3.7},
\begin{align}
\frac{\delta t}{\nu} \sum_{i=0}^n \|\nabla p^{i+1} \|^2
& \le \left(\frac{1}{2}+\varepsilon\right) \nu \delta t \sum_{i=0}^n
\| \Delta {\bf u}^{i+1} \|^2
+ \frac{C \nu \delta t}{\varepsilon^3} \sum_{i=0}^n \| \nabla {\bf u}^{i+1} \|^2 \notag\\
&\le  \left(\frac{1}{2}+\varepsilon\right) \tilde{C}_{3.7} + \frac{C\nu_{\max}\, T}{\varepsilon^3} \tilde{C}_{3.7}
\le \left(1 +\varepsilon  + \frac{C\nu_{\max} T}{\varepsilon^3} \right) \tilde{C}_{3.7}
\triangleq \frac{1}{2} C_{3.7}.
\label{def_C3.7}
\end{align}

Finally, adding \eqref{def_tildeC3.7} and \eqref{def_C3.7} produces
\eqref{3.7-new} with
\[
C_{3.7}=\left( 3+ 2\varepsilon+ \frac{2C\nu_{\max} T}{\varepsilon^3}\right) \tilde{C}_{3.7}.
\]
The constants $\tilde C_{3.7}$ and $C_{3.7}$ defined in
\eqref{def_tildeC3.7} and \eqref{def_C3.7} are monotonically increasing in
$\nu_{\max}$ and depend only on the initial condition, the final time~$T$,
and the parameter~$k$ (through $C_{II}$ and the explicit exponential factor).
The auxiliary constants $C_I$ and $C_{II}$ are independent of $\nu_{\max}$
altogether. All four are independent of the discretization parameters
$\delta t$ and~$n$. This completes the proof.
\end{proof}

\section{Generalized Error Analysis for Navier-Stokes Equations}
\label{sec_error_analysis}
The original error analysis for the scheme \eqref{HSscheme-a}-\eqref{HSscheme-e}, Theorem 7 in \cite{HuangShen2023}, is only for $\nu=1$ and $k\ge 5$.
Our analysis generalize their result to arbitrary viscosity $\nu > 0$ and reduce the stabilization parameter from $k=5$ to $k=4$. While we note that \cite{HuangShen2025} has reduced this requirement to $k=3$ for the second order scheme without GSAV, our work provides the necessary framework for tracking viscosity dependence.

The following lemma refines the ranges of $\xi$ and $\eta$ in \eqref{5.8-new} and \eqref{5.9-new}.
\begin{lemma}
\label{lemma_xi_eta}
Assume $|1-\xi| \le C_0 \delta t$ holds for $C_0>1$ and 
$0<\delta t \le \frac{1}{2C_0^2}$. 
Let $\eta=1-(1-\xi)^2$.
Then 
\begin{align}
\tfrac{1}{2} < \xi < \tfrac{3}{2},\quad 
\tfrac34 < \eta \le 1.
\label{estimate_xi_eta}
\end{align}
\end{lemma}
\begin{proof}
$ $\newline
Using $|1-\xi| \le C_0\,\delta t$, $C_0>1$ and 
$\delta t \le \frac{1}{2C_0^2}$,  
we obtain
\[
\frac{1}{2}
< 1 - \frac{1}{2C_0} \le 1 - C_0\delta t
\le \xi \le 1 + C_0\delta t
\le 1 + \frac{1}{2C_0} < \frac{3}{2}.
\]
Hence, $\frac{1}{2} < \xi < \tfrac{3}{2}$. 
Note $\eta= 1-(1-\xi)^{2}=2\,\xi- \xi^{2}$.
It is easy to derive that $\frac{3}{4}< \eta \le 1$ from the range of $\xi$.

\end{proof}

The weak stability result of Huang and Shen
\cite{HuangShen2023} (Theorem~6) is slightly revised below. 
\begin{theorem}[Updated Theorem 6 of \cite{HuangShen2023}]
\label{thm:HS2023_Thm6}
Let $\|{\bf f}(\cdot,t)\| \le C_f$ for all $t \in [0,T]$, and let
$\bar{\bf u}^{n+1}$, ${\bf u}^{n+1}$ be the solution of the scheme
\eqref{HSscheme}. For all $\delta t > 0$, choose
\begin{equation}
\bar C \;\ge\; \max\bigl\{ 2\,\delta t^{2}\, C_f^{2},\; 2\,C_f^{2},\; 1\bigr\}.
\label{eq:HS2023_Cbar_condition}
\end{equation}
Then, given $r^{n} \ge 0$, we have $r^{n+1} \ge 0$,
$\xi^{n+1} \ge 0$, and there exists a constant $M_T > 0$ depending only on $T$, ${\bf u}(0)$, and $C_f$ (see \eqref{M_T_form}) such that
\begin{equation}
\nu\, \delta t\, \sum_{j=0}^{n} \xi^{j+1}\,\|\nabla \bar{\bf u}^{j+1}\|^{2},
\quad
\|{\bf u}^{n+1}\|,
\quad
r^{n+1}
\;\le\; M_T,
\qquad \forall\, n+1 \le \frac{T}{\delta t}.
\label{eq:HS2023_Thm6_bound}
\end{equation}
\end{theorem}
The only change is that Theorem~6 in \cite{HuangShen2023} claims that $M_T$ depends only on $T$. However, from the last line of their proof,
\[
M_T \ge r^0 \bigl(1+ 2(1+T)e^T \bigr),
\]
and by using the definition of $r^0$ in \eqref{def_SAV} together with the condition of $\bar{C}$ in \eqref{eq:HS2023_Cbar_condition}, we further obtain
\begin{align}
M_T 
&\ge 
\left(\frac{1}{2}\|{\bf u}(0)\|^2 + \bar{C}\right)
\bigl(1+ 2(1+T)e^T \bigr)\notag \\
&\ge 
\left(
\frac{1}{2}\|{\bf u}(0)\|^2 
+ \max\bigl\{2\delta t^{2} C_f^{2},\; 2C_f^{2},\; 1\bigr\}
\right)
\bigl(1+ 2(1+T)e^T \bigr).
\label{M_T_form}
\end{align}
Therefore, $M_T$ depends not only on $T$, but also on the initial data ${\bf u}(0)$ and the magnitude of the external force $C_f$.

The inequality (5.13) in \cite{HuangShen2023} appears to contain a gap and is therefore replaced by the following lemma.
\begin{lemma}
\label{lemma5.13-new}
Suppose $\bar{\bf u}^i\in {\bf H}^1_0(\Omega)\cap {\bf H}^2(\Omega)$,  
$\|\bar{\bf u}^{i}\| \le 2 M_T$ for some $M_T>0$, and $|\eta^i|\le 1$ for all $i$. 
Denote $\widehat {\bf u}^{\,i}=(k+1)\eta^{i}\,\bar {\bf u}^{\,i}-k\,\eta^{\,i-1}\,\bar {\bf u}^{\,i-1}$
and
${\bf \tilde{u}}^{i+1}=(k+1) {\bar{\bf u}}^{i+1} - k {\bar{\bf u}}^i$. 
Then for any $\varepsilon>0$, 
\begin{align}
|\bigl(\widehat {\bf u}^{\,i}\!\cdot\!\nabla \widehat {\bf u}^{i},\;\Delta\tilde {\bf u}^{\,i+1}\bigr)|
\le 
\varepsilon \big( \|\Delta\bar {\bf u}^i\|^2
+ \|\Delta\bar {\bf u}^{\,i-1}\|^2
+ \|\Delta\tilde {\bf u}^{\,i+1}\|^2 \big)
+  \frac{4C^3(k+1)^6 M_T^2}{\varepsilon^3} \big(\|\nabla\bar {\bf u}^i\|^4+\|\nabla\bar {\bf u}^{\,i-1}\|^4\big).
\label{5.13-new}
\end{align}
where $C$ is constant shown in the elliptic regularity estimate \eqref{elliptic_regularity} and depends only on the domain.

\end{lemma}
\begin{proof}
$ $\newline
Denote $
I:=\bigl(\widehat {\bf u}^{\,i}\!\cdot\!\nabla \widehat {\bf u}^{i},\;\Delta\tilde {\bf u}^{\,i+1}\bigr)$.
Substituting and using the triangle inequality 
gives
\begin{align*}
|I|
&= \Bigl|\Bigl(\big[(k+1)\eta^{i}\bar {\bf u}^{\,i}-k\eta^{\,i-1}\bar {\bf u}^{\,i-1}\big]\!\cdot\!\nabla \big[(k+1)\eta^{i}\bar {\bf u}^{\,i}-k\eta^{\,i-1}\bar {\bf u}^{\,i-1}\big],
\Delta\tilde {\bf u}^{\,i+1}\Bigr)\Bigr|\\[4pt]
&\le (k+1)^2\,|\eta^{i}|^2\,
    \bigl|(\bar {\bf u}^{\,i}\!\cdot\!\nabla \bar {\bf u}^{i},\,\Delta\tilde {\bf u}^{\,i+1})\bigr|
  +(k+1)k\,|\eta^{i}\eta^{\,i-1}|\,
    \bigl|(\bar {\bf u}^{\,i}\!\cdot\!\nabla \bar {\bf u}^{\,i-1},\,\Delta\tilde {\bf u}^{\,i+1})\bigr|\\[4pt]
&\quad +k(k+1)\,|\eta^{\,i-1}\eta^{i}|\,
    \bigl|(\bar {\bf u}^{\,i-1}\!\cdot\!\nabla \bar {\bf u}^{i},\,\Delta\tilde {\bf u}^{\,i+1})\bigr|
  +k^2\,|\eta^{\,i-1}|^2\,
    \bigl|(\bar {\bf u}^{\,i-1}\!\cdot\!\nabla \bar{\bf u}^{\,i-1},\,\Delta\tilde {\bf u}^{\,i+1})\bigr|.
\end{align*}
By using \eqref{convection_estimate} and \eqref{elliptic_regularity} for $\tilde{\bf u} \in {\bf H}^1_0(\Omega)\cap {\bf H}^2(\Omega)$, we get
\[
\begin{aligned}
\bigl|(\bar {\bf u}^{i}\!\cdot\!\nabla \bar {\bf u}^{i},\;\Delta\tilde {\bf u}^{\,i+1})\bigr|
&\le C\,\|\bar {\bf u}^{i}\|^{1/2}\,\|\nabla\bar {\bf u}^{i}\|\,\|\Delta\bar {\bf u}^{i}\|^{1/2}\,\|\Delta\tilde {\bf u}^{\,i+1}\|,\\[6pt]
\bigl|(\bar {\bf u}^{i}\!\cdot\!\nabla \bar {\bf u}^{\,i-1},\;\Delta\tilde {\bf u}^{\,i+1})\bigr|
&\le C\,\|\bar {\bf u}^{i}\|^{1/2}\,\|\nabla\bar {\bf u}^{i}\|^{1/2}\,
     \|\nabla\bar {\bf u}^{\,i-1}\|^{1/2}\,\|\Delta\bar {\bf u}^{\,i-1}\|^{1/2}\,\|\Delta\tilde {\bf u}^{\,i+1}\|,\\[6pt]
\bigl|(\bar {\bf u}^{\,i-1}\!\cdot\!\nabla \bar {\bf u}^{i},\;\Delta\tilde {\bf u}^{\,i+1})\bigr|
&\le C\,\|\bar {\bf u}^{\,i-1}\|^{1/2}\,\|\nabla\bar {\bf u}^{\,i-1}\|^{1/2}\,
     \|\nabla\bar {\bf u}^{i}\|^{1/2}\,\|\Delta\bar {\bf u}^{i}\|^{1/2}\,\|\Delta\tilde {\bf u}^{\,i+1}\|,\\[6pt]
\bigl|(\bar {\bf u}^{\,i-1}\!\cdot\!\nabla \bar {\bf u}^{\,i-1},\;\Delta\tilde {\bf u}^{\,i+1})\bigr|
&\le C\,\|\bar {\bf u}^{\,i-1}\|^{1/2}\,\|\nabla\bar {\bf u}^{\,i-1}\|\,\|\Delta\bar {\bf u}^{\,i-1}\|^{1/2}\,\|\Delta\tilde {\bf u}^{\,i+1}\|.
\end{aligned}
\]

\medskip

\noindent Using $|\eta^i|\le 1$ from \eqref{5.9-new} and  Young's inequality with $\tilde{\varepsilon}>0$ to the above terms, one obtains
\begin{align}
|I| & \le C (k+1)^2 \left[ \tilde\varepsilon\|\Delta\tilde {\bf u}^{\,i+1}\|^{2} 
 + \frac{1}{\tilde\varepsilon}\Big(
\|\bar {\bf u}^{i}\|\,\|\nabla\bar {\bf u}^{i}\|^{2}\,\|\Delta\bar {\bf u}^{i}\|
+ \|\bar {\bf u}^{i}\|\,\|\nabla\bar {\bf u}^{i}\|\,\|\nabla\bar {\bf u}^{\,i-1}\|\,\|\Delta\bar {\bf u}^{\,i-1}\|
\right.
\notag
\\
& \left. \qquad\qquad
+ \|\bar {\bf u}^{\,i-1}\|\,\|\nabla\bar {\bf u}^{\,i-1}\|\,\|\nabla\bar {\bf u}^{i}\|\,\|\Delta\bar {\bf u}^{i}\|
+ \|\bar {\bf u}^{\,i-1}\|\,\|\nabla\bar {\bf u}^{\,i-1}\|^{2}\,\|\Delta\bar {\bf u}^{\,i-1}\|
\Big) \right].
\end{align}
Applying Young's inequality with $\varepsilon'>0$ on terms involving 
$\|\Delta\bar {\bf u}^{i}\|$ and $\|\Delta\bar {\bf u}^{i-1}\|$ 
and $\|\bar {\bf u}^i\|\le 2M_T$, $\forall i$ yields
\begin{align}
|I| \le 
C(k+1)^2 \left[ 
\tilde\varepsilon \|\Delta\tilde {\bf u}^{\,i+1}\|^2
+ \frac{2\varepsilon'}{\tilde{\varepsilon}}
 ( \|\Delta\bar {\bf u}^{i}\|^2 + 
   \|\Delta\bar {\bf u}^{\,i-1}\|^2 )
+ \frac{2M_T^2}{\tilde{\varepsilon} \varepsilon'}
  ( \|\nabla\bar {\bf u}^i\|^4 +
   \|\nabla\bar {\bf u}^{\,i-1}\|^4)
   \right]
\end{align}
The final result \eqref{5.13-new} is obtained by choosing $\varepsilon'=\tilde{\varepsilon}^2/2$ and
$\tilde{\varepsilon}=\frac{\varepsilon}{C(k+1)^2}$.
\end{proof}

\begin{theorem}[Updated Theorem 7 of \cite{HuangShen2023}; Error estimates for general $\nu$ and $k \geq 4$; spatial dimension=2]
\label{Theorem 7-new}
Consider the 2D Navier-Stokes equations (1.1) with initial data $\mathbf{u}_0 \in \mathbf{V} \cap \mathbf{H}^2_0$, where ${\bf V}=\{{\bf v}\in {\bf H}^1_0(\Omega): \nabla\cdot {\bf v}=0 \}$, 
and viscosity $\nu$ satisfies $0<\nu<\nu_{max}$ for some $\nu_{max}>0$.
Suppose the external force satisfies $\|{\bf f}(\cdot,t)\| \le C_f$ for all $t\in[0,T]$,
and the GSAV shift $\bar C$ satisfies
\begin{align}
\bar C \;\ge\; \max\bigl\{ 2\,\delta t^{2}\, C_f^{2},\; 2\,C_f^{2},\; 1\bigr\}.
\label{eq:Cbar_condition_Thm44}
\end{align}
Under the temporal regularity conditions
\begin{align*}
\frac{\partial\mathbf{u}}{\partial t} \in L^2(0,T;{\bf H}^1), 
\quad
\frac{\partial^2\mathbf{u}}{\partial t^2} \in L^2(0,T;{\bf H}^2),
\quad
\frac{\partial^3\mathbf{u}}{\partial t^3} \in L^2(0,T;{\bf L}^2), \quad
\frac{\partial^2 p}{\partial t^2} \in L^2(0,T;H^1), 
\end{align*}
there exists $C_0>1$ such that  when $\delta t \leq \frac{1}{1+2C_0^2}$, the numerical scheme \eqref{HSscheme-a}-\eqref{HSscheme-e} with $k \geq 4$ satisfies
\begin{align}
\|\nabla\bar{\mathbf{e}}^{n+1}\|^2 &+ \|\nabla \mathbf{e}^{n+1}\|^2 + \nu \delta t \sum_{i=0}^{n+1} \left( \|\Delta\bar{\mathbf{e}}^i\|^2 + \|\Delta \mathbf{e}^i\|^2 \right) + \frac{\delta t}{\nu} \sum_{i=0}^{n+1} \|\nabla e^i_p\|^2  \leq C_{5.3}\delta t^4
\label{5.3-new}
\end{align}
where $C_0, C_{5.3}> 0$ depend on $T$, $\Omega$, $\frac{1}{\nu}$, and the exact solution and the external force, but are independent of $\delta t$.
\end{theorem}

\begin{proof}
\mbox{}\\
The proof proceeds by induction on $n$, establishing the bound
$|1-\xi^i| \le C_0 \delta t$ for all $i\le \frac{T}{\delta t}$, where the
values of $C_0$ and $\delta t$ are determined in \eqref{def_C0} and
\eqref{5.71-new}, respectively. We assume $C_0>1$ and
$|1 - \xi^i| \le C_0 \delta t$ for all $i \le n$ (the base case $n=1$
follows from the first-order initialization \cite{guermond2003new}, and
the values of $C_0$ and $\delta t$ are chosen below to make this
assumption hold for $i=1$); the goal is to prove
$|1-\xi^{n+1}|\le C_0 \delta t$, which is achieved in \eqref{final_goal}
at the end of Step~3. The estimate \eqref{5.3-new} is a byproduct of
this induction.

Using $|1-\xi^{i}| \le C_0\,\delta t$, 
$\forall i\le n$, $C_0>1$, 
$\eta^{i} \;=\; 1-(1-\xi^{i})^{2}$, 
and the assumption $\delta t \le \frac{1}{1+2C_0^2}< \frac{1}{2C_0^2}$, 
we obtain from Lemma\,\ref{lemma_xi_eta} that 
\begin{align}
\tfrac{1}{2} < \xi^{i} < \tfrac{3}{2},  \quad \forall i\le n, 
\label{5.8-new}
\end{align}
and 
\begin{align}
\tfrac34 < \eta^{i} \le 1, \quad \forall i\le n.
\label{5.9-new}
\end{align}

\vskip 0.2cm

\noindent{\bf Step 1: finding upper bounds of $\|\nabla \bar{\bf u}^i\|$ and $\| \nabla{\bf u}^i\|$ for all $i\le n$}.
To establish \eqref{5.24-new} and \eqref{5.25-new}, we take the inner product of
\eqref{HSscheme-a} at $n=i+1$ with $-\Delta \tilde{\bf u}^{\,i+1}$, yielding
\begin{align}
 &\left( \frac{ (2k+1)\,\mathbf{\bar{u}}^{i+1} - 4k\,\mathbf{\bar{u}}^{i} + (2k-1)\,\mathbf{\bar{u}}^{i-1} } {2\delta t}, -\Delta  \tilde{\bf u}^{i+1} \right)
+ 2\delta t \nu \left((\Delta(k\bar{\mathbf{u}}^{i+1} - (k-1)\bar{\mathbf{u}}^i), \Delta \tilde{\mathbf{u}}^{i+1}\right)
\notag\\
 &=
 ( \hat{\mathbf{u}}^{\,i} \cdot \nabla \hat{\mathbf{u}}^{\,i}, \Delta  \tilde{\bf u}^{i+1}) 
+ ( \nabla \hat{p}^{\,i}, \Delta  \tilde{\bf u}^{i+1})
+ ( \mathbf{f}^{\,i+k}, -\Delta  \tilde{\bf u}^{i+1}).
\label{cat0090}
\end{align}
The analysis of the terms appearing in the resulting expression is presented below.
The time difference term becomes, according to Lemma\,\ref{lemma:algebraic_identity}, 
\begin{align}
&\left((2k+1)\bar{\mathbf{u}}^{i+1} - 4k\bar{\mathbf{u}}^i + (2k-1)\bar{\mathbf{u}}^{i-1}, -\Delta \tilde{\mathbf{u}}^{i+1}\right)\notag\\
=&\; A(\|\nabla\bar{\mathbf{u}}^{i+1}\|^2 - \|\nabla\bar{\mathbf{u}}^i\|^2) + \|B\nabla\bar{\mathbf{u}}^{i+1} - D\nabla\bar{\mathbf{u}}^i\|^2
- \|B\nabla\bar{\mathbf{u}}^i - D\nabla\bar{\mathbf{u}}^{i-1}\|^2 \notag\\
& + E\|\nabla\bar{\mathbf{u}}^{i+1} - \nabla\bar{\mathbf{u}}^i\|^2 
- F\|\nabla\bar{\mathbf{u}}^i - \nabla\bar{\mathbf{u}}^{i-1}\|^2 + G\|\nabla\bar{\mathbf{u}}^{i+1} - 2\nabla\bar{\mathbf{u}}^i + \nabla\bar{\mathbf{u}}^{i-1}\|^2. 
\label{5.11-new}
\end{align}
Due to the identity 
\begin{equation}
(ka-(k-1)b)\cdot ((k+1)a-kb)
= \frac{k-1}{k} ((k+1)a-kb)^2 + \frac{1}{k}a^2
+ \frac{1}{2} ( a^2 - b^2 + (a-b)^2),
\end{equation}
the viscous term with the general $\nu$ becomes
\begin{align}
\nu \left((\Delta(k\bar{\mathbf{u}}^{i+1} - (k-1)\bar{\mathbf{u}}^i), \Delta \tilde{\mathbf{u}}^{i+1}\right)
&= \nu \frac{k-1}{k} \|\Delta \tilde{\mathbf{u}}^{i+1}\|^2 + \frac{\nu}{k} \|\Delta\bar{\mathbf{u}}^{i+1}\|^2 \notag\\ 
&+ \frac{\nu}{2} \left(\|\Delta\bar{\mathbf{u}}^{i+1}\|^2 - \|\Delta\bar{\mathbf{u}}^i\|^2 + \|\Delta\bar{\mathbf{u}}^{i+1} - \Delta\bar{\mathbf{u}}^i\|^2\right).
\label{5.12-new}
\end{align}

Note that the conditions of Theorem~\ref{thm:HS2023_Thm6} are satisfied. Therefore,
$\|{\bf u}^{j}\|\le M_T$ and $r^{j}\ge 0$ for all
$0\le j \le T/\delta t$. Moreover, since $3/4 \le \eta^i < 1$ by
\eqref{5.9-new} and $\bar{\bf u}^i={\bf u}^i/\eta^i$, we obtain
\begin{align}
\|\bar{\bf u}^i\|
\le \frac{\|{\bf u}^i\|}{\eta^i}
\le \frac{4}{3}\|{\bf u}^i\|
\le 2M_T.
\label{ubar_upperbound}
\end{align}
Then, by Lemma~\ref{lemma5.13-new}, the following inequality for  the convection term follows where $\varepsilon$ is replaced by $4\varepsilon_1$:
\begin{align}
|\bigl(\widehat {\bf u}^{\,i}\!\cdot\!\nabla \widehat {\bf u}^{i},\;\Delta\tilde {\bf u}^{\,i+1}\bigr)|
\le 
4\varepsilon_1 \big( \|\Delta\bar {\bf u}^i\|^2
+ \|\Delta\bar {\bf u}^{\,i-1}\|^2
+ \|\Delta\tilde {\bf u}^{\,i+1}\|^2 \big)
+  \frac{4C^3(k+1)^6 M_T^2}{\varepsilon_1^3} \big(\|\nabla\bar {\bf u}^i\|^4+\|\nabla\bar {\bf u}^{\,i-1}\|^4\big).
\label{5.13-new-use}
\end{align}

For the pressure equation, taking $q\in H^1(\Omega)$ in
\eqref{HSscheme-b}, using the identity \eqref{3.10} together with
$\nabla\times\nabla\times\bar{\bf u}^i = -\Delta\bar{\bf u}^i + \nabla(\nabla\cdot\bar{\bf u}^i)$
and following the derivation of \eqref{3.11}--\eqref{3.12-new} for the
general viscosity~$\nu$, we obtain
\begin{align}
(\nabla p^i, \nabla q) = (f^i - \bar{\mathbf{u}}^i \cdot \nabla\bar{\mathbf{u}}^i, \nabla q) + \nu (\nabla p_s(\bar{\mathbf{u}}^i), \nabla q).
\label{5.15-new}
\end{align}
With $\hat{f}^i = (k+1)f^i - kf^{i-1}$, we have:
\begin{align}
(\nabla \hat{p}^i, \nabla q) = \left(\hat{f}^i - (k+1)\bar{\mathbf{u}}^i \cdot \nabla\bar{\mathbf{u}}^i + k\bar{\mathbf{u}}^{i-1} \cdot \nabla\bar{\mathbf{u}}^{i-1}, \nabla q\right) + \nu (\nabla p_s(\tilde{\mathbf{u}}^i), \nabla q).
\label{5.16-new}
\end{align}
Letting $q=\hat{p}^i$ in \eqref{5.16-new} leads to an upper bound for $\nabla \hat{p}^i$:
\begin{align}
\|\nabla \hat{p}^i\| \leq \left\| \hat{f}^i - (k+1)\bar{\mathbf{u}}^i \cdot \nabla\bar{\mathbf{u}}^i + k\bar{\mathbf{u}}^{i-1} \cdot \nabla\bar{\mathbf{u}}^{i-1} \right\| + \nu \|\nabla p_s(\tilde{\mathbf{u}}^i)\|.
\label{5.17-new}
\end{align}
Note the first term on the right of \eqref{5.17-new} can be bounded as:
\begin{align}
&\left\| \hat{f}^i - (k+1)\bar{\mathbf{u}}^i \cdot \nabla\bar{\mathbf{u}}^i + k\bar{\mathbf{u}}^{i-1} \cdot \nabla\bar{\mathbf{u}}^{i-1} \right\|^2\notag\\
\leq &\, 3\|\hat{f}^i\|^2 + 3(k+1)^2 \|\bar{\mathbf{u}}^i \cdot \nabla\bar{\mathbf{u}}^i\|^2 + 3k^2 \|\bar{\mathbf{u}}^{i-1} \cdot \nabla\bar{\mathbf{u}}^{i-1}\|^2 \notag \\
\leq &\, 3\|\hat{f}^i\|^2 + 3(k+1)^2 \|\bar{\mathbf{u}}^i\| \|\nabla\bar{\mathbf{u}}^i\|^2 \|\nabla\bar{\mathbf{u}}^i\|_1 + 3k^2 \|\bar{\mathbf{u}}^{i-1}\| \|\nabla\bar{\mathbf{u}}^{i-1}\|^2 \|\nabla\bar{\mathbf{u}}^{i-1}\|_1 \notag \\
\leq &\, 3\|\hat{f}^i\|^2 
+ \frac{9(k+1)^4}{\varepsilon_1'} M_T^2 \left(\|\nabla\bar{\mathbf{u}}^i\|^4 + \|\nabla\bar{\mathbf{u}}^{i-1}\|^4\right)
+ \varepsilon_1' \left(\|\Delta\bar{\mathbf{u}}^i\|^2 + \|\Delta\bar{\mathbf{u}}^{i-1}\|^2\right), 
\label{5.18-new}
\end{align}
where 
$\varepsilon_1'>0$ is a control parameter given by  Young's inequality. 
In the last step, the estimates $ \|\bar{\mathbf{u}}^i\|,  \|\bar{\mathbf{u}}^{i-1}\|\le 2 M_T$ are used, which are derived in \eqref{ubar_upperbound}.
Combining \eqref{5.17-new} and \eqref{5.18-new} via Cauchy--Schwarz and
Young's inequality (with parameters $\varepsilon_1,\varepsilon_2>0$) yields
\begin{align}
\big( \nabla \hat{p}^i, \Delta \tilde{\bf u}^{i+1}\big) 
&\le \| \Delta \tilde{\bf u}^{i+1} \| \cdot
\Big( \| \hat{f}^i
- (k+1)\bar{\bf u}^{i}\cdot\nabla\bar{\bf u}^{i}
+ k \bar{\bf u}^{i-1}\cdot\nabla\bar{\bf u}^{i-1} \|
+ \nu \|\nabla p_s(\tilde{\bf u}^i) \| \Big) \notag \\
&\le \frac{1}{4\varepsilon_1} \| \hat{f}^i
- (k+1)\bar{\bf u}^{i}\cdot\nabla\bar{\bf u}^{i}
+ k \bar{\bf u}^{i-1}\cdot\nabla\bar{\bf u}^{i-1} \|^2
+ \varepsilon_1 \| \Delta \tilde{\bf u}^{i+1} \| ^2
\notag\\
& \quad + \frac{\nu}{2\varepsilon_2} \|\nabla p_s(\tilde{\bf u}^i) \|^2
+ \frac{\nu\varepsilon_2}{2}  \| \Delta \tilde{\bf u}^{i+1} \| ^2 \notag\\
& \le \frac{3}{4\varepsilon_1} \| \hat{f}^i\|^2 
+ \frac{9(k+1)^4}{4\varepsilon_1\varepsilon_1'} 
M^2_T ( \| \nabla\bar{\bf u}^{i} \|^4 
  + \| \nabla\bar{\bf u}^{i-1} \|^4 ) 
  + \frac{\varepsilon_1'}{4\varepsilon_1} 
  (\| \Delta\bar{\bf u}^{i} \|^2
   + \| \Delta\bar{\bf u}^{i-1} \|^2  ) 
+ \varepsilon_1 \| \Delta \tilde{\bf u}^{i+1} \| ^2
\notag\\
& \quad + \frac{\nu}{2\varepsilon_2} \|\nabla p_s(\tilde{\bf u}^i) \|^2
+ \frac{\nu\varepsilon_2}{2}  \| \Delta \tilde{\bf u}^{i+1} \| ^2.
\end{align}
Here, one choose $\varepsilon_1'=4\varepsilon_1^2$. Using  the estimate 
$\|\nabla p_s(\tilde{\bf u}^i) \|^2
\le (\frac{1}{2}+\varepsilon_1) \|\Delta\tilde {\bf u}^i\|^2 + C_S(\varepsilon_1) \|\nabla \tilde{\bf  u}^i\|^2$ based on \eqref{lemma4-Liu}, one attains  
\begin{align}
\big( \nabla \hat{p}^i, \Delta \tilde{\bf u}^{i+1}\big) 
&\le 
C_{100}(\varepsilon_1) \cdot \Big(  \|\hat{f}^i\|^2 + M_T^2 ( \|\nabla\bar{\bf u}^i\|^4 + \|\nabla\bar{\bf u}^{i-1}\|^4 ) \Big)
+ \varepsilon_1 (\|\Delta\bar{\bf u}^i\|^2 + \|\Delta\bar{\bf u}^{i-1}\|^2 + \| \Delta \tilde{\bf u}^{i+1} \| ^2)\notag\\
& \quad + \frac{\nu}{2\varepsilon_2}\left( \frac{1}{2} + \varepsilon_1\right) \|\Delta \tilde{\bf u}^i \|^2
+ \frac{\nu}{\varepsilon_2} C_{100}(\varepsilon_1) \|\nabla\tilde{\bf u}^i \|^2
+ \frac{\nu\varepsilon_2}{2}  \| \Delta \tilde{\bf u}^{i+1} \| ^2,
\label{5.19-new}
\end{align}
where   
\begin{align}
C_{100}(\varepsilon_1)=\max\left\{\frac{3}{4\varepsilon_1}, \frac{9 (k+1)^4}{16 \varepsilon_1^3}, \frac{C}{2\varepsilon_1^3} \right\}.
\label{cat0026}
\end{align}
Note that the value of $C_{100}(\varepsilon_1)$ will continue to change until it is finalized in \eqref{def_Ceps}.
The inner product between the external force and $- \Delta \widetilde{\mathbf{u}}^{\,i+1}$ is changed to 
\begin{align}
\bigl( \mathbf{f}^{i+k}, - \Delta \widetilde{\mathbf{u}}^{\,i+1} \bigr)
\le 
\frac{1}{4\varepsilon_1} \|\mathbf{f}^{\,i+k}\|^{2}
+ 
\varepsilon_1\,\|\Delta \widetilde{\mathbf{u}}^{\,i+1}\|^{2}.
\label{5.20-new}
\end{align}
Inserting \eqref{5.11-new}, \eqref{5.12-new}, \eqref{5.13-new-use}, \eqref{5.19-new}, and \eqref{5.20-new} to \eqref{cat0090},  multiplying $2\delta t$, dropping $\delta t\nu \|\Delta \bar{\bf u}^{i+1} - \Delta \bar{\bf u}^i \|^2$ and the G-term in \eqref{5.11-new} yields
\begin{align}
&A(\|\nabla\bar{\mathbf{u}}^{i+1}\|^2 - \|\nabla\bar{\mathbf{u}}^i\|^2) + \|B\nabla\bar{\mathbf{u}}^{i+1} - D\nabla\bar{\mathbf{u}}^i\|^2 
- \|B\nabla\bar{\mathbf{u}}^i - D\nabla\bar{\mathbf{u}}^{i-1}\|^2+ 
E\|\nabla\bar{\mathbf{u}}^{i+1} - \nabla\bar{\mathbf{u}}^i\|^2 \notag\\
&- F\|\nabla\bar{\mathbf{u}}^i - \nabla\bar{\mathbf{u}}^{i-1}\|^2 
+ 2\nu\delta t \frac{k-1}{k} \|\Delta\tilde{\mathbf{u}}^{i+1}\|^2 
+ \frac{2\nu\delta t}{k} 
\|\Delta\bar{\mathbf{u}}^{i+1}\|^2 
+ \nu\delta t \left(\|\Delta\bar{\mathbf{u}}^{i+1}\|^2 - \|\Delta\bar{\mathbf{u}}^i\|^2\right) \notag \\
\leq & C_{100}(\varepsilon_1)M_T^2 \delta t 
\left(\|\nabla\bar{\mathbf{u}}^i\|^4 + \|\nabla\bar{\mathbf{u}}^{i-1}\|^4\right) + \frac{\nu}{\varepsilon_2}\left(\frac{1}{2} + \varepsilon_1\right) \delta t \|\Delta\tilde{\mathbf{u}}^i\|^2 + \varepsilon_2\nu\delta t \|\Delta\tilde{\mathbf{u}}^{i+1}\|^2 \notag \\
&+ 4\varepsilon_1\delta t 
\left(1.5 \|\Delta\tilde{\mathbf{u}}^{i+1}\|^2 
 + \|\Delta\bar{\mathbf{u}}^i\|^2 + \|\Delta\bar{\mathbf{u}}^{i-1}\|^2\right)+ C_{100}(\varepsilon_1) \frac{\nu}{\varepsilon_2} \delta t \|\nabla\tilde{\mathbf{u}}^i\|^2 + C_{100}(\varepsilon_1)\delta t \left(\|f^{i+k}\|^2 + \|\hat{f}^i\|^2\right).
\label{5.21-new}
\end{align}
Here 
\begin{align}
C_{100}(\varepsilon_1)=\max\left\{ \frac{3}{4\varepsilon_1}, \frac{9 (k+1)^4}{16 \varepsilon_1^3}, \frac{C}{2\varepsilon_1^3} \right\} +  \frac{4C^3(k+1)^6 M_T^2}{\varepsilon_1^3}.
\label{cat0027}
\end{align}
Summation of \eqref{5.21-new} for $i=1,\cdots,  m-1$, for any $m\le n$,  yields
\begin{align}
& \quad A \|\nabla \bar{\bf u}^m\|^2 
+ \| B \nabla \bar{\bf u}^m - D \nabla\bar{\bf u}^{m-1}\|^2 
+ E \|\nabla \bar{\bf u}^m - \nabla \bar{\bf u}^{m-1}\|^2 
 + \sum^{m-2}_{i=2} (E-F) \|\nabla \bar{\bf u}^{i+1} - \nabla \bar{\bf u}^{i}\|^2 \notag \\
&  + \sum_{i=2}^{m} \|\Delta \tilde{\bf u}^i\|^2 \delta t \Big( \frac{2\nu (k-1)}{k} - \nu \varepsilon_2 -  6\varepsilon_1  - \frac{\nu}{\varepsilon_2} (\frac{1}{2}+\varepsilon_1) \Big) 
+ \sum_{i=2}^{m} \|\Delta \bar{\bf u}^i\|^2 \delta t \Big( \tfrac{2\nu}{k} - 8\varepsilon_1 \Big) 
 \notag \\
\leq & \quad \sum_{i=1}^{m-1} C_{100}(\varepsilon_1) M_T^2 \,\delta t \|\nabla \bar{\bf u}^i\|^4 
+ \sum_{i=1}^{m-1} C_{100}(\varepsilon_1) \tfrac{\nu}{\varepsilon_2}\,\delta t \|\nabla \tilde{\bf u}^{i}\|^2 + \sum_{i=1}^{m-1} C_{100}(\varepsilon_1) \,\delta t \Big(\|f^{i+k}\|^2 + \|\hat{f}^{i}\|^2\Big) 
+ M_0,
\label{cat0023} \\
\leq & \quad C_{100}(\varepsilon_1) M_T^2 \,\delta t \sum_{i=1}^{m-1}  \|\nabla \bar{\bf u}^i\|^4 
+ C_{100}(\varepsilon_1) \frac{\nu}{\varepsilon_2}\,\delta t \sum_{i=1}^{m-1}  \|\nabla \bar{\bf u}^{i}\|^2 + \sum_{i=1}^{m-1} C_{100}(\varepsilon_1) \,\delta t \Big(\|f^{i+k}\|^2 + \|\hat{f}^{i}\|^2\Big) 
+ M_0,
\label{cat0024} 
\end{align}
where $M_0$ contains all other terms at the initial time steps $i=0,1$ and  $C_{100}(\varepsilon_1)$ is multiplied by $4(k+1)^2$ from \eqref{cat0023} to \eqref{cat0024} according to the relation 
$\tilde{\bf u}^{i+1}=(k+1) {\bar{\bf u}}^{i+1} - k {\bar{\bf u}}^i$.
For stability, the following conditions are forced to hold:
\begin{align}
2\nu\frac{k-1}{k} - \nu\varepsilon_2 - 6 \varepsilon_1 - \frac{\nu}{\varepsilon_2} \left(\frac{1}{2} + \varepsilon_1\right) \ge 0
\quad \text{ and }
\frac{2\nu}{k} -8\varepsilon_1\ge \frac{\nu}{k},
\label{5.22-new}
\end{align}
where the first is from the coefficient of 
$\| \Delta\tilde{\bf u}^{i}\|^2$, and the second from  the coefficient of $\|\Delta\bar{\mathbf{u}}^{i}\|^2$, both on the left of \eqref{cat0023}. 
The first inequality in \eqref{5.22-new} is equivalent to
$\frac{k-1}{k} \ge \frac{\varepsilon_2}{2} + 
\frac{1}{4\varepsilon_2} + \varepsilon_1
 \left( \frac{3}{\nu} + \frac{1}{2\varepsilon_2} \right)$. Applying Lemma\,\ref{lemma-k-values}, we find that the condition $k\ge 4$ holds for this inequality when 
$\varepsilon_2=\frac{1}{\sqrt{2}}$ and 
 $0<\varepsilon_1\le \frac{\frac{3}{4} - \frac{1}{\sqrt{2}}}{\frac{3}{\nu} + \frac{1}{\sqrt{2}}}$. Combining this with the constraint imposed by the second inequality in \eqref{5.22-new}, we further require $0<\varepsilon_1  \le \min\left( 
 \frac{\frac{3}{4} - \frac{1}{\sqrt{2}}}{\frac{3}{\nu} + \frac{1}{\sqrt{2}}}, \frac{\nu}{32} \right)$ in order to ensure that $k\ge 4$. Because $0<\nu<\nu_{\max}$, we choose 
\begin{align}
\varepsilon_1=\nu \cdot D_{\nu_{\max}}
\quad
\text{ where }
D_{\nu_{\max}}= \min\left( 
 \frac{\frac{3}{4} - \frac{1}{\sqrt{2}}}{3 + \frac{\nu_{\max}}{\sqrt{2}}}, \frac{1}{32} \right).
\label{eps1_def}
\end{align}

With these choices of $k$ and $\varepsilon_1$, and taking
$A = \frac{1}{2k}$ together with the condition $E > F$ as stated in
Lemma\,\ref{lemma:algebraic_identity}, we discard certain nonnegative
terms on the left-hand side of \eqref{cat0024} and arrive at the
following estimate, in which the constant $C_{100}(\varepsilon_1)$ is
multiplied by $2k$,
\begin{align}
\|\nabla\bar{\mathbf{u}}^m\|^2 + \nu\delta t \sum_{i=1}^m \|\Delta\bar{\mathbf{u}}^i\|^2 \leq & C_{100}(\varepsilon_1) M_T^2\delta t \sum_{i=1}^{m-1} \|\nabla\bar{\mathbf{u}}^i\|^4 
+ C_{100}(\varepsilon_1) \nu\delta t \sum_{i=1}^{m-1} \|\nabla\bar{\mathbf{u}}^i\|^2 + C_{100}(\varepsilon_1) T C_f^2 + M_0,
\label{5.23-new}
\end{align}
where $C_f=\sup_{t\in [0,T]} \|{\bf f}(\cdot, t)\|$.
Using the estimate
$\nu \delta t \sum_{i=0}^{m-1} \|\nabla \bar{\bf u}^i \|^2 \le 2 M_T$ (derived from the inequality $\nu \delta t \sum_{i=0}^{m-1} \xi^i \|\nabla \bar{\bf u}^i \|^2 \le M_T$ from Theorem\,6 in \cite{HuangShen2023} and $0.5<\xi^i<2$ from \eqref{5.8-new}) reduces it to, with $C_{100}(\varepsilon_1)$ multiplied by 2 from its previous value,
\begin{align}
\|\nabla\bar{\mathbf{u}}^m\|^2 + \nu\delta t \sum_{i=1}^m \|\Delta\bar{\mathbf{u}}^i\|^2 \leq & C_{100}(\varepsilon_1) M_T^2\delta t \sum_{i=1}^{m-1} \|\nabla\bar{\mathbf{u}}^i\|^4 
+ C_{100}(\varepsilon_1) M_T^2 + C_{100}(\varepsilon_1) T C_f^2 + M_0.
\label{cat0025}
\end{align}
Applying the discrete Gronwall inequality (Lemma~2 of \cite{HuangShen2023}; see also the statement and proof of Theorem~\ref{Theorem 5-New}) gives rise to
\begin{align}
\|\nabla\bar{\mathbf{u}}^m\|^2 + \nu\delta t \sum_{i=1}^m \|\Delta\bar{\mathbf{u}}^i\|^2 \leq 
\exp\left( \frac{C_{100}(\varepsilon_1)M_T^3}{\nu} \right) 
\cdot ( C_{100}(\varepsilon_1)M_T + C_{100}(\varepsilon_1) TC_f^2 + M_0)
\triangleq \tilde{C}_1,
\quad \forall m\le n, 
\label{5.24-new}
\end{align}
where $\nu \delta t \sum_{i=0}^{m-1} \|\nabla \bar{\bf u}^i \|^2 \le 2 M_T$ is used again and the value of $C_{100}(\varepsilon_1)$ is doubled one more time.
Since ${\bf u}^i=\eta^i \bar{\bf u}^i$ and $|\eta^i| \le 1$ from Lemma\,\ref{lemma_xi_eta}, the above inequality yields the same upper bound for $\|\nabla\mathbf{u}^m\|^2$:
\begin{align}
\|\nabla\mathbf{u}^m\|^2 + \nu\delta t \sum_{i=1}^m \|\Delta\mathbf{u}^i\|^2 \le \tilde{C}_1,
\quad \forall m\le n.
\label{5.25-new}
\end{align}
From the above process, one obtains the final value of $C_{100}(\varepsilon_1)$ in \eqref{5.24-new} as 
\begin{align}
C_{100}(\varepsilon_1) = 16k(k+1)^2 \cdot \left( 
  \frac{64C^3(k+1)^6 M_T^2}{\varepsilon_1^3} 
+ \max\left\{ \frac{12}{\varepsilon_1}, 
\frac{9(k+1)^4}{\varepsilon_1^3}, \frac{8C}{\varepsilon_1^3}  \right\}
   \right),
   \label{def_Ceps}
\end{align}
where $C$ is the one used from \eqref{convection_estimate} to \eqref{lemma4-Liu}.
We then define a key constant:
\begin{align}
C_1 = 4(k+1)^2 \cdot  \max \{\tilde{C}_1,
 \max\limits_{t \in (0,T)} \|\mathbf{u}(t)\|_2^2 \},
\label{def_C1}
\end{align}
which ensures the boundedness in \eqref{5.50-new} is true. The appearance of $(k+1)^2$ is because of the definition of $\hat{\bf u}=(k+1){\bf u}^i -k {\bf u}^{i-1}$.

Due to the choice of $\varepsilon_1=\nu\cdot D_{\nu_{\max}}$ in \eqref{eps1_def}, we have $C_{100}(\varepsilon_1) = O(\frac{1}{\nu^3})$ when $\nu\to 0$.
Combined with the $\exp(C_{100}(\varepsilon_1) M_T^3/\nu)$ factor in \eqref{5.24-new}, this yields $C_1\sim \exp(\mathcal{O}(\nu^{-4}))$ as $\nu\to 0^+$, the first of the three nested exponentials traced in Remark~\ref{remark4.5}.

\vskip 0.2cm

{\bf Step 2: finding upper bound of $\|\nabla\bar{\bf e}^{n+1} \|$}. 
We begin with the error equation, obtained by subtracting the
exact-solution form of \eqref{HSscheme-a} (i.e., the same scheme applied
to the exact solution ${\bf u}(t^i)$) from \eqref{HSscheme-a} and
collecting the truncation residuals on the right-hand side:
\begin{align}
&(2k+1)\bar{\mathbf{e}}^{i+1}  -4k\bar{\mathbf{e}}^i
+ (2k-1) \bar{\mathbf{e}}^{i-1}
-2\nu\delta t \Delta (k\bar{\mathbf{e}}^{i+1} - (k-1)\bar{\mathbf{e}}^i) \notag\\
&+2\delta t \left( \hat{\mathbf{u}}^i\cdot\nabla\hat{\mathbf{u}}^i 
-\hat{\mathbf{u}}(t^i)\cdot\nabla\hat{\mathbf{u}}(t^i)\right) 
+ 2\delta t \nabla\hat{e}^i_p = 2\delta t P^i + 2\delta t Q^i + R^i + 2\delta t S^i,
\label{5.26-new}
\end{align}
where $\hat{e}^i_p\triangleq (k+1)e^i_p - k e^{i-1}_p$ and 
$\hat{\mathbf{u}}(t^i) = (k+1){\mathbf{u}}(t^i) - k{\mathbf{u}}(t^{i-1})$.
The truncation error terms are given by:
\begin{align}
P^i &= \nabla p(t^{i+k}) - \nabla ((k+1)p(t^i) - k p(t^{i-1})) \notag\\
&=(k+1) \int_{t^i}^{t^{i+k}}  (t^i-s) \nabla \frac{\partial^2 p}{\partial t^2} (s) ds
- k\int_{t^{i-1}}^{t^{i+k}} (t^{i-1}-s) \nabla\frac{\partial^2 p}{\partial t^2}(s) ds,
\label{5.27-new}
\\
Q^i &= -\Delta {\mathbf{u}}(t^{i+k}) + \Delta (k {\mathbf{u}}(t^{i+1}) - (k-1){\mathbf{u}}(t^{i})) \notag\\
&=-k \int_{t^{i+1}}^{t^{i+k}}  (t^{i+1}-s) \Delta \frac{\partial^2 {\mathbf{u}}}{\partial t^2} (s) ds
- (k-1) \int_{t^i}^{t^{i+k}} (t^{i}-s) \Delta \frac{\partial^2 {\mathbf{u}}}{\partial t^2}(s) ds,
\label{5.28-new}
\\
R^i &= -(2k+1) {\mathbf{u}}(t^{i+1}) + 4k {\mathbf{u}}(t^{i}) 
- (2k-1){\mathbf{u}}(t^{i-1})
+2\delta t {\mathbf{u}}_t(t^{i+k}) \notag\\
&=\frac{2k+1}{2} \int_{t^{i+1}}^{t^{i+k}}  (t^{i+1}-s)^2  \frac{\partial^3 {\mathbf{u}}}{\partial t^3} (s) ds
- (2k) \int_{t^i}^{t^{i+k}} (t^{i}-s)^2 
\frac{\partial^3 {\mathbf{u}}}{\partial t^3}(s) ds
\notag \\
&\quad + \frac{2k-1}{2} \int_{t^{i-1}}^{t^{i+k}} (t^{i-1}-s)^2 
\frac{\partial^3 {\mathbf{u}}}{\partial t^3}(s) ds,
\label{5.29-new}
\\
S^i &= {\mathbf{u}}(t^{i+k}) \cdot \nabla{\mathbf{u}}(t^{i+k})
-\hat{\mathbf{u}}(t^i)\cdot\nabla\hat{\mathbf{u}}(t^i)
\notag\\
&= {\mathbf{u}}(t^{i+k})\cdot\nabla\left({\mathbf{u}}(t^{i+k}) 
- \hat{\mathbf{u}}(t^i)\right)
- \left(\hat{\mathbf{u}}(t^i)-{\mathbf{u}}(t^{i+k})\right)\cdot\nabla\hat{\mathbf{u}}(t^i).
\label{5.30-new}
\end{align}
After taking the inner product of \eqref{5.26-new} with $-\Delta\tilde{\mathbf{e}}^{i+1}$ where $\tilde{\mathbf{e}}^{i+1} = (k+1)\bar{\mathbf{e}}^{i+1} - k\bar{\mathbf{e}}^i$, we analyze each term of the resulting expression below.

\vskip 0.2cm
{\bf Step 2.1}.  
The temporal difference term satisfies
\begin{align}
&\left((2k+1)\bar{\mathbf{e}}^{i+1} - 4k\bar{\mathbf{e}}^i + (2k-1)\bar{\mathbf{e}}^{i-1}, -\Delta \tilde{\mathbf{e}}^{i+1} \right) \notag\\
=&A\left(\|\nabla\bar{\mathbf{e}}^{i+1}\|^2 - \|\nabla\bar{\mathbf{e}}^{i}\|^2\right)
+\|B\nabla\bar{\mathbf{e}}^{i+1} - D\nabla\bar{\mathbf{e}}^i\|^2
- \|B\nabla\bar{\mathbf{e}}^{i} - D\nabla\bar{\mathbf{e}}^{i-1}\|^2
\notag\\
& + E\|\nabla\bar{\mathbf{e}}^{i+1} - \nabla\bar{\mathbf{e}}^i\|^2
- F\|\nabla\bar{\mathbf{e}}^{i} - \nabla\bar{\mathbf{e}}^{i-1}\|^2
+G \|\nabla\bar{\mathbf{e}}^{i+1} - 2\nabla\bar{\mathbf{e}}^i + \nabla\bar{\mathbf{e}}^{i-1}\|^2.
\label{5.31-new}
\end{align}

\vskip 0.2cm
{\bf Step 2.2}. 
The viscous term satisfies 
\begin{align}
&\left( -2\nu\delta t \Delta (k\bar{\mathbf{e}}^{i+1}-(k-1)\bar{\mathbf{e}}^i), -\Delta \tilde{\mathbf{e}}^{i+1}\right) \notag\\
=&\frac{2(k-1)\nu\delta t}{k} \|\Delta\hat{\mathbf{e}}^{i+1}\|^2
+ \frac{2\nu\delta t}{k} \|\Delta\bar{\mathbf{e}}^{i+1}\|^2
+ \nu\delta t \left(\|\Delta \bar{\mathbf{e}}^{i+1}\|^2
 - \|\Delta\bar{\mathbf{e}}^i\|^2
 + \|\Delta\bar{\mathbf{e}}^{i+1} - \Delta\bar{\mathbf{e}}^i\|^2\right).
\label{5.32-new}
\end{align}

\vskip 0.2cm
{\bf Step 2.3}. Next, we study the convection terms.
Decomposing
$\hat{\bf u}^i\cdot\nabla\hat{\bf u}^i - \hat{\bf u}(t^i)\cdot\nabla\hat{\bf u}(t^i)
= \hat{\bf e}^i\cdot\nabla\hat{\bf u}^i + \hat{\bf u}(t^i)\cdot\nabla\hat{\bf e}^i$
and applying the convection estimate \eqref{Teman_estimate} together with
elliptic regularity \eqref{elliptic_regularity}, one obtains
\begin{align}
|\bigl(
&\hat{\bf u}^{i} \cdot \nabla \hat{\bf u}^{i}
- \hat{\bf u}(t^{i})\cdot \nabla \hat{\bf u}(t^{i}),
- \Delta \tilde{\bf e}^{\,i+1}\bigr)| =
| \bigl(
\hat{\mathbf e}^{i} \cdot \nabla \hat{\mathbf u}^{i},
- \Delta \tilde{\mathbf e}^{\,i+1}
\bigr)
+
\bigl(
\hat{\mathbf u}(t^{i})\cdot \nabla \hat{\mathbf e}^{i},
- \Delta \tilde{\mathbf e}^{\,i+1}
\bigr) | \notag
\\
&\le
C \|\nabla \hat{\bf e}^{i}\|
\|\hat{\bf u}^{i}\|_{2}
\|\Delta \tilde{\bf e}^{\,i+1}\|
+
C \|\hat{\bf u}(t^{i})\|_{2}
\|\nabla \hat{\bf e}^{i}\|
\|\Delta \tilde{\bf e}^{\,i+1}\|
\label{cat0102}
\\
&\le
C^2 \|\nabla \hat{\bf e}^{i}\|
\|\Delta \hat{\bf u}^{i}\|
\|\Delta \tilde{\mathbf e}^{\,i+1}\|
+
C \|\hat{\bf u}(t^{i})\|_{2}
\|\nabla \hat{\bf e}^{i}\|
\|\Delta \tilde{\bf e}^{\,i+1}\|
\label{cat0103}
\\
&\le
C_{201} \|\nabla \hat{\bf e}^{i}\|
\|\Delta \hat{\bf u}^{i}\|
\|\Delta \tilde{\mathbf e}^{\,i+1}\|
+
C_{201} \|\hat{\bf u}(t^{i})\|_{2}
\|\nabla \hat{\bf e}^{i}\|
\|\Delta \tilde{\bf e}^{\,i+1}\|
\label{cat0104}
\\
&\le
\frac{C_{201}^2}{2\varepsilon_3} 
\|\nabla \hat{\bf e}^{i}\|^{2}
\cdot \left( \|\Delta \hat{\bf u}^{i}\|^{2}
+
\|\hat{\bf u}(t^{i})\|_2^{2}\right)
+
\varepsilon_3 \|\Delta \tilde{\mathbf e}^{\,i+1}\|_{2}^{2},
\label{5.34-new}
\end{align}
where the constant $C$ in \eqref{cat0102} arises from the basic inequalities
for the convection terms \eqref{Teman_estimate}, and 
the constant $C^2$ in \eqref{cat0103} appears due to an additional factor $C$
coming from the elliptic regularity estimate
$\|\hat{\bf u}^{i}\|_{2} \le C \|\Delta \hat{\bf u}^{i}\|$. 
In \eqref{cat0104}, 
\begin{align}
C_{201} = \max\{C^2,\, C\}.
\label{def_C201}
\end{align}
The upper bound in \eqref{5.34-new} follows from Young’s inequality with an arbitrary
$\varepsilon_3 > 0$.

Using the symbols listed in \eqref{notations}, 
we can derive that $\hat{\bf u}^i - \hat{\bf e}^i = \tilde{\bf u}^i - \tilde{\bf e}^i = \hat{\bf u}(t^i)$.
Thus, 
\begin{align}
\|\nabla \hat{\bf e}^{i}\|^2
&=
\|\nabla \hat{\bf u}^{i}
- \nabla \tilde{\bf u}^{i}
+ \nabla \tilde{\bf e}^{i}\|^2 \notag\\
&\le 
2 \| \nabla (\hat{\bf u}^i - \tilde{\bf u}^i) \|^2
+ 2 \|\nabla \tilde{\bf e}^i \|^2
\notag\\
&= 2 \| \nabla 
      [(k+1) ({\bf u}^i     - \bar{\bf u}^i) 
         -k  ({\bf u}^{i-1} - \bar{\bf u}^{i-1}) 
      ]   \|^2
+ 2 \|\nabla \tilde{\bf e}^i \|^2
\notag\\
&\le 
4 (k+1)^2 
\| \nabla  ({\bf u}^i - \bar{\bf u}^i)\|^2 
+4k^2 \| \nabla ({\bf u}^{i-1} - \bar{\bf u}^{i-1}) 
   \|^2
+ 2 \|\nabla \tilde{\bf e}^i \|^2.
\label{cat0105}
\end{align}
Applying the relations ${\bf u}^i - \bar{\bf u}^i = (\eta^i-1) \bar{\bf u}^i$ and $(1-\eta^i) = (1-\xi^i)^2$, together with the induction hypothesis $|1-\xi^i| \le C_0 \delta t$ for all $i\le n$, we simplify \eqref{cat0105} to
\begin{align}
\|\nabla \hat{\bf e}^{i}\|^2
\le 
4 (k+1)^2 C_0^4 \delta t^4 
\left( \| \nabla \bar{\bf u}^i\|^2 
+      \| \nabla \bar{\bf u}^{i-1} \|^2 \right) 
+ 2 \|\nabla \tilde{\bf e}^i \|^2.
\label{cat0106}
\end{align}
Using \eqref{5.24-new} and the definition of $C_1$ in \eqref{def_C1}, one attains 
$\|\nabla \bar{\bf u}^i\|^2,  \| \nabla \bar{\bf u}^{i-1} \|^2 \le C_1$ and thus
\begin{align}
\|\nabla \hat{\bf e}^{i}\|^2
\le 
8 (k+1)^2 C_1 C_0^4 \delta t^4 
+ 2 \|\nabla \tilde{\bf e}^i \|^2.
\label{5.35-new}
\end{align}
Inserting \eqref{5.35-new} to \eqref{5.34-new} leads to
\begin{align}
&|\bigl(
\hat{\bf u}^{i} \cdot \nabla \hat{\bf u}^{i}
- \hat{\bf u}(t^{i})\cdot \nabla \hat{\bf u}(t^{i}),
- \Delta \tilde{\bf e}^{\,i+1}
\bigr)| \notag \\
&\le \frac{C_{201}^2}{\varepsilon_3} 
\|\nabla \tilde{\bf e}^{i}\|^{2}
\cdot \left( \|\Delta \hat{\bf u}^{i}\|^{2}
+
\|\hat{\bf u}(t^{i})\|_2^{2}\right)
+
\varepsilon_3 \|\Delta \tilde{\bf e}^{\,i+1}\|_{2}^{2}
+ \frac{C_{201}^2}{\varepsilon_3} 4 (k+1)^2C_1 C_0^4 \delta t^4 
\cdot \left( \|\Delta \hat{\bf u}^{i}\|^{2}
+
\|\hat{\bf u}(t^{i})\|_2^{2}\right).
\label{5.36-new}
\end{align}

\vskip 0.2cm
{\bf Step 2.4}. 
For the pressure-error term, with
$\hat e_p^{i}\triangleq (k+1)e^i_p - k e^{i-1}_p$ and
$e^i_p\triangleq p^i-p(t^i)$, Cauchy--Schwarz directly yields:
\begin{align}
|\bigl(\nabla \hat e_p^{i},\,
- \Delta \tilde{\bf e}^{\,i+1}
\bigr)|
\le
\|\nabla \hat e_p^{i}\|
\,
\|\Delta \tilde{\bf e}^{\,i+1}\|.
\label{5.36}
\end{align}
To bound $\|\nabla\hat e^i_p\|$, we proceed as follows. Subtracting the
exact-solution pressure equation from the scheme \eqref{HSscheme-b} (with
indices shifted appropriately to account for the GSAV correction) gives,
for the unbarred and the BDF-combined error pressure,
\begin{align}
\left(\nabla e^i_p, \nabla q\right) &=
\left({\mathbf{u}}(t^i)\cdot\nabla{\mathbf{u}}(t^i)-\bar{\mathbf{u}}^i\cdot\nabla\bar{\mathbf{u}}^i,\nabla q\right)
+ \nu \left(\nabla p_s(\bar{\mathbf{e}}^i), \nabla q\right),
\label{5.37-new}
\\
\left(\nabla \hat{e}^i_p, \nabla q\right)
&= \left(
\begin{aligned}
&(k+1) {\mathbf{u}}(t^i)\cdot\nabla{\mathbf{u}}(t^i)
-(k+1)\bar{\mathbf{u}}^i\cdot\nabla\bar{\mathbf{u}}^i \\
&- k{\mathbf{u}}(t^{i-1})\cdot\nabla{\mathbf{u}}(t^{i-1})
+ k \bar{\mathbf{u}}^{i-1}\cdot\nabla\bar{\mathbf{u}}^{i-1}
\end{aligned},
\nabla q \right)
+ \nu \left(\nabla p_s(\tilde{\mathbf{e}}^i), \nabla q\right),
\label{5.38-new}
\\
\|\nabla\hat{e}^i_p\|
&\le (k+1) \|{\mathbf{u}}(t^i)\cdot\nabla{\mathbf{u}}(t^i) - \bar{\mathbf{u}}^i \cdot\nabla\bar{\mathbf{u}}^i\|
\notag\\
& \quad + k \|{\mathbf{u}}(t^{i-1})\cdot\nabla{\mathbf{u}}(t^{i-1})
- \bar{\mathbf{u}}^{i-1} \cdot\nabla\bar{\mathbf{u}}^{i-1}\|
+ \nu \| \nabla p_s(\tilde{\mathbf{e}}^i)\|,
\label{5.39-new}
\end{align}
and
\begin{align}
(k+1) \left[ {\bf{u}}(t^i)\cdot\nabla{\bf{u}}(t^i) - \bar{\bf{u}}^i \cdot\nabla\bar{\bf{u}}^i\right]
&= -(k+1) \left[ \bar{\bf{e}}^i\cdot\nabla\bar{\bf{u}}^i
  + {\bf{u}}(t^i) \cdot\nabla\bar{\bf{e}}^i \right],
\label{5.40-new}
  \\
(k+1)^2 \| {\bf{u}}(t^i)\cdot\nabla{\bf{u}}(t^i) - \bar{\bf{u}}^i \cdot\nabla\bar{\bf{u}}^i \|^2
&\le  2(k+1)^2 C^4 \| \nabla\bar{\bf{e}}^i \|^2 \|\Delta \bar{\bf{u}}^i\|^2 
  + 2(k+1)^2 \| {\bf{u}}(t^i)\|^2_2  
      \| \nabla\bar{\bf{e}}^i \|^2,
  \label{5.41-new}  
  \\
k^2 \| {\bf{u}}(t^{i-1})\cdot\nabla{\bf{u}}(t^{i-1}) 
- \bar{\bf{u}}^{i-1} \cdot\nabla\bar{\bf{u}}^{i-1} \|^2
&\le 2k^2C^4 \| \nabla\bar{\bf{e}}^{i-1} \|^2 
    \|\Delta \bar{\bf{u}}^{i-1} \|^2 
  + 2k^2 \| {\bf{u}}(t^{i-1})\|^2_2  
      \| \nabla\bar{\bf{e}}^{i-1} \|^2.
  \label{5.42-new}  
\end{align}
where the constant $C$ in \eqref{5.41-new} and \eqref{5.42-new} is from the Poincar\'{e} inequality and elliptic regularity mentioned at the beginning of Section\,\ref{sec_error_analysis}.
Combining \eqref{5.36} with \eqref{5.39-new}--\eqref{5.42-new}, applying
Cauchy--Schwarz, and using Young's inequality with parameters
$\varepsilon_3,\varepsilon_4>0$, we obtain
\begin{align}
&\big|\left( \nabla \hat{e}^i_p, -\Delta \tilde{\bf{e}}^{i+1} \right)\big|  \\
\le & \|\Delta\tilde{\bf{e}}^{i+1}\|
\bigg( (k+1) \|{\bf{u}}(t^i)\cdot\nabla{\bf{u}}(t^i)
           -\bar{\bf{u}}^i\cdot\nabla\bar{\bf{u}}^i \|
  + k \|{\bf{u}}(t^{i-1})\cdot\nabla{\bf{u}}(t^{i-1})
           -\bar{\bf{u}}^{i-1}\cdot\nabla\bar{\bf{u}}^{i-1} \| \bigg)    
\notag\\     
& + \nu \|\Delta\tilde{\bf{e}}^{i+1} \| 
     \cdot \|\nabla p_s(\tilde{\bf{e}}^i \|
\notag\\
\le &  \frac{1}{2\varepsilon_3} \bigg( 
   (k+1)^2 \|{\bf{u}}(t^i)\cdot\nabla{\bf{u}}(t^i)
    -\bar{\bf{u}}^i\cdot\nabla\bar{\bf{u}}^i \|^2
  + k^2 \|{\bf{u}}(t^{i-1})\cdot\nabla{\bf{u}}(t^{i-1})
     -\bar{\bf{u}}^{i-1}\cdot\nabla\bar{\bf{u}}^{i-1} \|^2 \bigg)       
\notag\\
& + \varepsilon_3 \|\Delta\tilde{\bf{e}}^{i+1} \|^2 
+ \frac{\nu}{2\varepsilon_4} \|\nabla p_s(\tilde{\bf{e}}^i)\|^2
+ \frac{\nu \varepsilon_4}{2} \|\Delta\tilde{\bf{e}}^{i+1} \|^2
\notag \\
\le & \frac{C_{202}}{\varepsilon_3} \|\nabla\bar{\bf{e}}^i\|^2 
\bigg( \|\Delta\bar{\bf{u}}^i \|^2 
      + \|{\bf{u}}(t^i)\|^2_2 \bigg) 
+ \frac{C_{202}}{\varepsilon_3}
  \|\nabla\bar{\bf{e}}^{i-1}\|^2 
\bigg( \|\Delta\bar{\bf{u}}^{i-1} \|^2 
      + \|{\bf{u}}(t^{i-1})\|^2_2 \bigg) 
\notag\\
& + \bigg( \varepsilon_3 + \frac{\nu\varepsilon_4}{2} \bigg)
    \|\Delta\tilde{\bf{e}}^{i+1}\|^2
    + \frac{\nu}{2\varepsilon_4} \bigg( \frac{1}{2} + \varepsilon_3 \bigg) \|\Delta\tilde{\bf{e}}^i \|^2
    +   \frac{\nu C}{2\varepsilon_4\varepsilon_3^3}  \|\nabla\tilde{\bf{e}}^i\|^2. 
\label{5.43-new}         
\end{align}
where 
\begin{align}
C_{202}=(k+1)^2 \max\{1, C^4 \},
\label{def_C202}
\end{align}
and  \eqref{lemma4-Liu} is used in the last step  but with $\varepsilon$ replaced by $\varepsilon_3$.

\vskip 0.2cm
{\bf Step 2.5}. 
Using \eqref{5.27-new}, 
\begin{align}
& \bigl| (P^i, -\Delta \tilde{\bf{e}}^{\,i+1}) \bigr|
\le \frac{1}{4\varepsilon_3} \|P^i\|^2
   + \varepsilon_3 \|\Delta \tilde{\bf{e}}^{\,i+1}\|^2
\notag\\
&\le \frac{2 (k+1)^2}{4\varepsilon_3}
\left\|
\int_{t^i}^{t^{i+k}} (t^i - s)\,
\nabla \frac{\partial^2 p}{\partial t^2}(s)\, ds
\right\|^2 
+ \frac{2 k^2}{4\varepsilon_3}
\left\|
\int_{t^{i-1}}^{t^{i+k}} (t^{i-1} - s)\,
\nabla \frac{\partial^2 p}{\partial t^2}(s)\, ds
\right\|^2 
+ \varepsilon_3 \|\Delta \tilde{\bf{e}}^{\,i+1}\|^2
\notag\\
&\le \frac{(k+1)^2}{2\varepsilon_3}
\frac{(k\delta t)^3}{3} 
\int_{t^i}^{t^{i+k}}
\| \nabla \frac{\partial^2 p}{\partial t^2}(s)\|^2 ds
+ \frac{k^2}{2\varepsilon_3}
\frac{((k+1)\delta t)^3}{3}
\int_{t^{i-1}}^{t^{i+k}} 
\| \nabla \frac{\partial^2 p}{\partial t^2}(s)\|^2 ds
+ \varepsilon_3 \|\Delta \tilde{\bf{e}}^{\,i+1}\|^2
\notag\\
&\le \frac{(k+1)^5 \delta t^3}{3\varepsilon_3} 
\int_{t^{i-1}}^{t^{i+k}} 
\| \nabla \frac{\partial^2 p}{\partial t^2}(s)\|^2 ds
+ \varepsilon_3 \|\Delta \tilde{\bf{e}}^{\,i+1}\|^2.
\label{5.44-new}
\end{align}
The relation  
$| \int_{t^i}^{t^{i+k}} (t^i - s)\,
\nabla \frac{\partial^2 p}{\partial t^2}(s)ds|^2
\le \int_{t^i}^{t^{i+k}} |t^i - s|^2 ds
\cdot\int_{t^i}^{t^{i+k}} 
|\nabla \frac{\partial^2 p}{\partial t^2}(s)|^2ds
=\frac{(k\delta t)^3}{3}\cdot \int_{t^i}^{t^{i+k}} 
|\nabla \frac{\partial^2 p}{\partial t^2}(s)|^2ds
$ is used to get the third inequality above. 
Similarly, we can get from \eqref{5.28-new}, \eqref{5.29-new}, and \eqref{5.30-new} that 
\begin{align}
|(Q^i, -\Delta \tilde{\bf e}^{i+1})|
&\le \frac{k^5 \delta t^3 }{3\varepsilon_3} 
\int_{t^{i-1}}^{t^{i+k}} 
\| \Delta \frac{\partial^2 {\bf u}}{\partial t^2}(s)\|^2 ds
+ \varepsilon_3 \|\Delta \tilde{\bf{e}}^{\,i+1}\|^2.
\label{5.45-new}
\\
|(R^i, -\Delta \tilde{\bf e}^{i+1})|
&\le \frac{1}{4\varepsilon_3 \delta t} \|R^i\|^2
+ \varepsilon_3\delta t \|\Delta \tilde{\bf e}^{i+1} \|^2
\le 
\frac{6 (k+1)^5 \delta t^4 }{5\varepsilon_3} 
\int_{t^{i-1}}^{t^{i+k}} 
\| \frac{\partial^3 {\bf u}}{\partial t^3}(s)\|^2 ds
+ \varepsilon_3 \delta t\|\Delta \tilde{\bf{e}}^{\,i+1}\|^2.
\label{5.46-new}
\\
|(S^i, -\Delta \tilde{\bf e}^{i+1})|
&\le 
C \|{\bf u}(t^{i+k})\|_2 \cdot 
\| \nabla[ {\bf u}(t^{i+k}) - \hat{\bf{u}}(t^i)] \|
\cdot \| \Delta \tilde{\bf e}^{i+1} \|
+
C \|\hat{\bf u}(t^{i})\|_2 \cdot 
\| \nabla[ {\bf u}(t^{i+k}) - \hat{\bf{u}}(t^i)] \|
\cdot \| \Delta \tilde{\bf e}^{i+1} \|
\label{cat0123}
\\
&\le 
\frac{C^2}{2\varepsilon_3} 
\cdot  \left( \|{\bf u}(t^{i+k})\|_2^2
+ \|\hat{\bf u}(t^{i})\|_2^2 \right) \cdot
\| \nabla[ {\bf u}(t^{i+k}) - \hat{\bf{u}}(t^i)] \|^2
+ \varepsilon_3 \| \Delta \tilde{\bf e}^{i+1} \|^2
\label{cat0124}
\\
&\le 
\frac{4C^2 (k+1)^5 \delta t^3 \max_{t\in [0,T]} \|{\bf u}(t)\|_2^2} {3\varepsilon_3} 
\int_{t^i}^{t^{i+k}} \|\nabla \frac{\partial^2 {\bf u}}{\partial t^2}(s)\|^2 ds 
+ \varepsilon_3 \| \Delta \tilde{\bf e}^{i+1} \|^2
.
\label{5.47-new}
\end{align}
The estimate \eqref{Teman_estimate} is used to get \eqref{cat0123}, and the analysis similar to $\|P^i\|$ in \eqref{5.44-new} is applied on 
$\| \nabla[ {\bf u}(t^{i+k}) - \hat{\bf{u}}(t^i)] \|$ in \eqref{cat0124} to earn \eqref{5.47-new}.

\vskip 0.2cm
{\bf Step 2.6}.
Inserting \eqref{5.31-new}, \eqref{5.32-new},
\eqref{5.36-new}, \eqref{5.43-new}, \eqref{5.44-new}, 
\eqref{5.45-new}, \eqref{5.46-new},
\eqref{5.47-new} to \eqref{5.26-new} yields
\begin{align}
&A\left(\|\nabla\bar{\bf{e}}^{i+1}\|^2 - \|\nabla\bar{\bf{e}}^i\|^2\right)
+ \|B\nabla\bar{\bf{e}}^{i+1} - D\nabla\bar{\bf{e}}^i\|^2
- \|B\nabla\bar{\bf{e}}^{i} - D\nabla\bar{\bf{e}}^{i-1}\|^2
\notag\\
&+E\|\nabla\bar{\bf{e}}^{i+1} - \nabla\bar{\bf{e}}^i \|^2
-F\|\nabla\bar{\bf{e}}^i - \nabla\bar{\bf{e}}^{i-1}\|^2
\notag\\
&
+2\nu\delta t \frac{k-1}{k} \|\Delta\tilde{\bf{e}}^{i+1}\|^2
+\frac{2\nu\delta t}{k} \|\Delta\bar{\bf{e}}^{i+1}\|^2
+\nu\delta t \left(\|\Delta\bar{\bf{e}}^{i+1}\|^2
 - \|\Delta\bar{\bf{e}}^i\|^2 \right) 
\notag\\
\le &
\delta t ( 11 \varepsilon_3 + \nu\varepsilon_4)
\|\Delta\tilde{\mathbf{e}}^{i+1} \|^2 
+ \frac{\nu\delta t}{\varepsilon_4} 
\left( \frac{1}{2} + \varepsilon_3 \right) 
 \|\Delta\tilde{\bf{e}}^i \|^2 
\notag\\
&+ 
C_{260}(\varepsilon_3) \delta t \|\nabla\tilde{\bf{e}}^i\|^2 
\left(\|\Delta\hat{\bf{u}}^i\|^2 + \|\hat{\bf{u}}(t^i)\|^2_2 + 1 \right)
+  
C_{260}(\varepsilon_3)  \delta t \|\nabla\bar{\bf{e}}^i\|^2 
\left(\|\Delta\bar{\bf{u}}^i\|^2 + \|{\bf{u}}(t^i)\|^2_2 \right)
\notag\\
&  + 
C_{260}(\varepsilon_3)  \delta t \|\nabla\bar{\bf{e}}^{i-1}\|^2 
\left(\|\Delta\bar{\bf{u}}^{i-1}\|^2 + \|{\bf{u}}(t^{i-1})\|^2_2 \right)
+
C_{260}(\varepsilon_3)  C^4_0\delta t^5 \left( \|\Delta\hat{\bf{u}}^i\|^2 
        + \|\hat{\bf{u}}(t^i)\|^2_2\right)
\notag\\
&+ C_{260}(\varepsilon_3)  \delta t^4 
\int_{t^{i-1}}^{t^{i+k}} 
\bigg(
\|\nabla \frac{\partial^2 p}{\partial t^2}(s)\|^2
+ \|\Delta \frac{\partial^2 {\bf{u}}}{\partial t^2}(s)\|^2
+ \| \frac{\partial^3 {\bf{u}}}{\partial t^3}(s)\|^2
+ \|\nabla \frac{\partial^2 {\bf{u}}}{\partial t^2}(s)\|^2
\bigg) ds.
\end{align}
Here,
\begin{align}
C_{260}(\varepsilon_3)
=\max\Big\{
& \frac{2C_{201}^2}{\varepsilon_3}, \quad
\frac{C \nu_{\max} }{\varepsilon_4 \varepsilon_3^3},\quad
\frac{2C_{202}^2}{\varepsilon_3},\quad
\frac{2(k+1)^5}{3\varepsilon_3},\quad
\frac{2k^5}{3\varepsilon_3},\quad
\frac{6(k+1)^5}{5\varepsilon_3},\quad
\notag\\
&\frac{8C^2 (k+1)^5 \max_{t\in[0,T]}\|{\bf u}(t)\|_2^2}{3\varepsilon_3},\quad
\frac{8 (k+1)^2 C_{201}^2 C_1}{\varepsilon_3}
\Big\}
\end{align}
where 
$C$ is the uniform constant given at the beginning of Section\,\ref{sec_error_analysis}, 
$C_{201}$ is defined in \eqref{def_C201},
$C_{202}$ in \eqref{def_C202}, 
$C_1$ in \eqref{def_C1}, and $\varepsilon_3$ and $\varepsilon_4$ are arbitrary positive real numbers.

Proceeding as in Step 1 and summing over 
$i=1,\cdots, m$, for any $m\le n$,  we discard the unnecessary terms to obtain
\begin{align}
& A\|\nabla\bar{\bf{e}}^{m+1}\|^2 
+ \|B\nabla\bar{\bf{e}}^{m+1} - D\nabla\bar{\bf{e}}^m\|^2
+(E-F) \sum_{i=1}^m \|\nabla\bar{\bf{e}}^{i+1} - \nabla\bar{\bf{e}}^i \|^2
\notag\\
&+ \Biggl[
  \frac{2\nu (k-1)}{k}
  - \Bigl( \nu\varepsilon_4 + 11 \varepsilon_3\Bigr)
  - \frac{\nu }{\varepsilon_4}\Bigl(\tfrac{1}{2} +      \varepsilon_3\Bigr)
\Biggr]\delta t \sum_{i=1}^m \|\Delta \tilde{\bf e}^{i+1}\|^2 
+\frac{2\nu\delta t}{k} \sum_{i=1}^m \|\Delta\bar{\bf{e}}^{i+1}\|^2
+\nu\delta t \|\Delta\bar{\bf{e}}^{m+1}\|^2
\notag\\
\le \quad  &\;
C_{260}(\varepsilon_3) \delta t 
\sum_{i=1}^m \|\nabla\tilde{\bf{e}}^i\|^2 
\Big(\|\Delta\hat{\bf{u}}^i\|^2 + \|\hat{\bf{u}}(t^i)\|^2_2 + 1 \Big)
\notag\\
& +2C_{260}(\varepsilon_3) \delta t \sum_{i=1}^m \|\nabla\bar{\bf{e}}^i\|^2 
\Big(\|\Delta\bar{\bf{u}}^i\|^2 + \|{\bf{u}}(t^i)\|^2_2 \Big)
+ C_{260} C_0^4\delta t^5 \sum_{i=1}^m \Big( \|\Delta\hat{\bf{u}}^i\|^2 
+ \|\hat{\bf{u}}(t^i)\|^2_2\Big)
\notag\\
&+ C_{260} (\varepsilon_3) \delta t^4 
\int_{t^{0}}^{t^{m+k}}\!\!
\Biggl(
\|\nabla \tfrac{\partial^2 p}{\partial t^2}(s)\|^2
+ \|\Delta \tfrac{\partial^2 \bf{u}}{\partial t^2}(s)\|^2
+ \|\tfrac{\partial^3 \bf{u}}{\partial t^3}(s)\|^2
+ \|\nabla \tfrac{\partial^2 \bf{u}}{\partial t^2}(s)\|^2
\Biggr) ds + M_1,
\label{cat0125}
\end{align}
where $M_1=
\| B\nabla\bar{\bf e}^1 - D \nabla\bar{\bf e}^0\|^2
+ F \|\nabla\bar{\bf e}^1 - \nabla\bar{\bf e}^0  \|^2
+\nu\delta t \|\Delta\bar{\bf e}^1 \|^2
+ \frac{\nu \delta t}{\varepsilon_4} (\frac{1}{2}+\varepsilon_3) \| \Delta\tilde{\bf e}^1 \|^2$. By using an explicit scheme such as the second order Runge-Kutta method, it can be controlled that $M_1<\hat{C}_{260} \delta t^6$ where $\hat{C}_{260}$ only depends on the exact solution. Because $M_1$ is far smaller than the two proceeding terms involving $\delta t^5$ and $\delta t^4$, we neglect this term in the subsequent analysis. 
Stability requires the coefficient of
$\sum_{i=1}^m \|\Delta \tilde{\bf e}^{i+1}\|^2$ to be positive, which yields 
$
\frac{k-1}{k} \ge \frac{\varepsilon_4}{2} + \frac{1}{4\varepsilon_4} + \varepsilon_3 \left(\frac{11}{\nu} + \frac{1}{2\varepsilon_4}\right)$. 
According to Lemma\,\ref{lemma-k-values}, this is true for $k\ge 4$ when  $\varepsilon_4 = 1/\sqrt{2}$,  and $\varepsilon_3 =E_\varepsilon \nu$ with 
$E_\varepsilon= \frac{\frac{3}{4}-\frac{1}{\sqrt{2}}}{11+\frac{\nu_{\max}}{\sqrt{2}}}$.
Using these special values of $\varepsilon_3$ and $\varepsilon_4$, \eqref{cat0125} reduces to (after dividing $A=1/(2k)$ and dropping some positive terms on the left) 
\begin{align}
& \|\nabla\bar{\bf{e}}^{m+1}\|^2 
+ \nu\delta t \sum_{i=1}^m \|\Delta\bar{\bf{e}}^{i+1}\|^2
\notag\\
\le & \quad 2kC_{260}(\varepsilon_3) \delta t 
\sum_{i=1}^m \|\nabla\bar{\bf{e}}^i\|^2 
\left(\|\Delta\bar{\bf{u}}^i\|^2 + \|{\bf{u}}(t^i)\|^2_2 \right) 
+ 2k C_{260}(\varepsilon_3) \delta t 
\sum_{i=1}^m \|\nabla\tilde{\bf{e}}^i\|^2 
\left(\|\Delta\hat{\bf{u}}^i\|^2 + \|\hat{\bf{u}}(t^i)\|^2_2 + 1\right)
\notag\\
&+ 2k C_{260}(\varepsilon_3) C^4_0 \delta t^5 
\sum_{i=1}^m \left(\|\Delta\hat{\bf{u}}^i\|^2 + \|\hat{\bf{u}}(t^i)\|_2^2 \right) 
+ 8 k C_{260}(\varepsilon_3)\, \delta t^4 \,
(T+1)\,  C_{261},
\label{cat0126}
\end{align}
where $T+1$ is used to accommodate $[t^0, t^{m+k}]$ by assuming $k\delta t<1$, and 
\begin{align}
C_{261}\triangleq \max_{s\in [0,T+1]}\{
\|\nabla \tfrac{\partial^2 p}{\partial t^2}(s)\|^2,
\|\Delta \tfrac{\partial^2 \bf{u}}{\partial t^2}(s)\|^2, 
\|\tfrac{\partial^3 \bf{u}}{\partial t^3}(s)\|^2, 
\|\nabla \tfrac{\partial^2 \bf{u}}{\partial t^2}(s)\|^2\}.
\label{def_C261}
\end{align}
Using 
$\tilde{\bf e}^i=(k+1)\bar{\bf e}^i - k \bar{\bf e}^{i-1}$ and thus 
$\|\nabla\tilde{\bf{e}}^i\|^2 \le 2(k+1)^2 ( 
\|\nabla\bar{\bf{e}}^i\|^2 + \|\nabla\bar{\bf{e}}^{i-1}\|^2)$, 
we get
\begin{align}
& \|\nabla\bar{\bf{e}}^{m+1}\|^2 
+ \nu\delta t \sum_{i=1}^m \|\Delta\bar{\bf{e}}^{i+1}\|^2
\notag\\
\le & \quad 4k(k+1)^2
C_{260}(\varepsilon_3) \delta t 
\sum_{i=0}^m \|\nabla\bar{\bf{e}}^i\|^2 
\Bigl(
\|\Delta\bar{\bf{u}}^i\|^2 + \|{\bf{u}}(t^i)\|^2_2 
+ \|\Delta\hat{\bf{u}}^i\|^2 + \|\hat{\bf{u}}(t^i)\|^2_2 
\notag\\
&\qquad\qquad\qquad\qquad\qquad\qquad\qquad
 +  (\|\Delta\hat{\bf{u}}^{i+1}\|^2 + \|\hat{\bf{u}}(t^{i+1})\|^2_2) \cdot \chi_{i<m}   +2
\Bigr) 
\notag\\
&+ 2k C_{260}(\varepsilon_3) C^4_0 \delta t^5 
\sum_{i=1}^m \left(\|\Delta\hat{\bf{u}}^i\|^2 + \|\hat{\bf{u}}(t^i)\|_2^2 \right) 
+ 8 k C_{260}(\varepsilon_3)\, \delta t^4 \,
(T+1)\,  C_{261},
\label{cat0127}
\end{align}
where $\chi_{i<m}=1$ for $i\le m-1$ and $\chi_{i<m}=0$ for $i=m$.
Define 
\begin{align}
C_{262}(\varepsilon_3) \triangleq \max\{4k(k+1)^2, 8k C_{261} \} C_{260}(\varepsilon_3),
\label{def_C262}
\end{align}
then \eqref{cat0127} turns to, for all $m\le n$, 
\begin{align}
& \|\nabla\bar{\bf e}^{m+1}\|^2
+ \nu\delta t \sum_{i=1}^m \|\Delta\bar{\bf e}^{i+1}\|^2
\notag\\
\le\;&
C_{262}(\varepsilon_3)\,\delta t
\sum_{i=0}^m \|\nabla\bar{\bf e}^i\|^2
\Bigl(
\|\Delta\bar{\bf u}^i\|^2 + \|{\bf u}(t^i)\|_2^2
+ \|\Delta\hat{\bf u}^i\|^2 + \|\hat{\bf u}(t^i)\|_2^2
 +  (\|\Delta\hat{\bf{u}}^{i+1}\|^2 + \|\hat{\bf{u}}(t^{i+1})\|^2_2) \cdot \chi_{i<m}   +2
\Bigr) 
\notag\\
&
+ C_{262}(\varepsilon_3) C_0^4 \delta t^5
\sum_{i=1}^m
\bigl(\|\Delta\hat{\bf u}^i\|^2
+ \|\hat{\bf u}(t^i)\|^2\bigr)
+ C_{262}(\varepsilon_3)(T+1)\delta t^4.
\label{5.49-new}
\end{align}
According to \eqref{5.24-new}, \eqref{5.25-new}, and the  definition of $C_1$ in \eqref{def_C1}, we obtain 
\begin{align}
\nu\delta t\sum_{i=0}^n\|\Delta\bar{\bf{u}}^i\|^2, \,
\nu\delta t\sum_{i=0}^n\|\Delta\hat{\bf{u}}^i\|^2, \,
\|\hat{\bf{u}}(t^i)\|^2_2, \,
\|{\bf{u}}(t^i)\|^2_2 \le C_1, 
\quad \forall i\le n,
\label{5.50-new}
\end{align}
where the inequality $\|\Delta\hat{\bf{u}}^i\|^2 
= \|(k+1)\Delta {\bf u}^i- k \Delta{\bf u}^{i-1} \|^2
\le 2(k+1)^2 (\|\Delta {\bf u}^i\|^2 + \|\Delta {\bf u}^{i-1} \|^2)$ and a similar one $\|\hat{\bf{u}}(t^i)\|^2_2
\le 2(k+1)^2 (\|{\bf u}(t^i)\|^2_2+\|{\bf u}(t^{i-1})\|^2_2)$
are used.

Applying the discrete Gronwall inequality (Lemma~2 of \cite{HuangShen2023}; see also Theorem~\ref{Theorem 5-New}) yields
\begin{align}
\|\nabla\bar{\bf{e}}^{n+1}\|^2 
+ \nu\delta t \sum_{i=0}^{n+1} \|\Delta\bar{\bf{e}}^i\|^2
\le& 
C_{262}\cdot \delta t^4 \cdot
\left( C^4_0 C_1 \Big(\frac{1}{\nu}+T\Big) +(1+T) \right)
\cdot
\exp\left( C_{262} \cdot
 \Big( \frac{3C_1}{\nu} + 3 C_1T + 2T\Big)  \right) 
\notag\\
\le & C_2\delta t^4 (C^4_0+1) ,
\label{5.51-new}
\end{align}
where 
\begin{align}
C_2\triangleq  C_{262}\cdot \exp\left( C_{262} \cdot
 \Big( \frac{3C_1}{\nu} + 3 C_1T + 2T\Big)  \right)
\cdot 
\max\left\{ C_1\left(\frac{1}{\nu} +T\right), 1+T \right\},
\end{align}
where $C_1$ is defined in \eqref{def_C1} and $C_{262}$ in \eqref{def_C262}.
Note that $C_{262}\sim \nu^{-3}$ (via \eqref{lemma4-Liu} with the choice $\varepsilon_3=E_\varepsilon\,\nu$ stated below \eqref{cat0126}), and $C_1\sim \exp(\mathcal{O}(\nu^{-4}))$ from Step~1; hence $C_2\sim \exp(\exp(\mathcal{O}(\nu^{-4})))$, the second of the three nested exponentials traced in Remark~\ref{remark4.5}.

The last thing in Step 2 is an upper bound for $\|\nabla \bar{\bf u}^{n+1} \|^2$:
\begin{align}
\|\nabla \bar{\bf u}^{n+1} \|^2
&=\|\nabla \bar{\bf e}^{n+1} + \nabla {\bf u}(t^{n+1}) \|^2 
\le 2 \|\nabla\bar{\bf e}^{n+1} \|^2 
+ 2 \| \nabla {\bf u}(t^{n+1}) \|^2 
\notag\\
&\le 2C_2 \delta t^4 (C_0^4+1) + 2 C_1
\le 4C_2 + 2 C_1 \triangleq C_3.
\label{5.52-new}
\end{align}
The relation $\delta t\le \frac{1}{2C_0^2}$ is used in the last step.

{\bf Step 3: Estimate for $|1-\xi^{n+1}|$}.
Define
\begin{equation}
s^i \triangleq r^i - r(t^i).
\end{equation}
The error equation for $s^i$ for the general viscosity $\nu$ is 
\begin{align}
s^{i+1} - s^i 
=& \nu \delta t
\bigg( \|\nabla{\bf u}(t^{i+1})\|^2
- \frac{r^{i+1}}{E(\bar{\bf u}^{i+1}) + \bar{C}}
\|\nabla(\bar{\bf u}^{i+1})\|^2 
\bigg) \notag \\ 
&+ 
\delta t \bigg( 
\frac{r^{i+1}}{E(\bar{\bf u}^{i+1}) + \bar{C}}
({\bf f}^{i+1}, \bar{\bf u}^{i+1})
- ({\bf f}(t^{i+1}), {\bf u}(t^{i+1}))
\bigg)
+ T^i,
\label{5.53-new}
\end{align}
where the remainder term $T_i$ is given as 
\begin{equation}
\label{5.54}
T_i = r(t_i) - r(t_{i+1}) + \delta t\cdot r_t(t_{i+1})
= \int_{t_i}^{t_{i+1}} (s - t_i)\, r_{tt}(s)\, ds .
\end{equation}
The summation of \eqref{5.53-new} over time steps $i=0, \cdots, n$ is
\begin{align}
s^{n+1} 
=& \nu \delta t
\sum_{i=0}^n \bigg( \|\nabla{\bf u}(t^{i+1})\|^2
- \frac{r^{i+1}}{E(\bar{\bf u}^{i+1}) + \bar{C}}
\|\nabla(\bar{\bf u}^{i+1})\|^2 
\bigg) \notag \\ 
&+ 
\delta t \sum_{i=0}^n \bigg( 
\frac{r^{i+1}}{E(\bar{\bf u}^{i+1}) + \bar{C}}
({\bf f}^{i+1}, \bar{\bf u}^{i+1})
- ({\bf f}(t^{i+1}), {\bf u}(t^{i+1}))
\bigg)
+ \sum_{i=0}^n T^i.
\label{5.55-new}
\end{align}
Based on \eqref{5.54} and $r_{tt}=\int_{\Omega} |{\bf u}_t|^2 + {\bf u}\cdot{\bf u}_{tt} dx$, we have 
\begin{align}
|T^i|\le \delta t \int_{t^i}^{t^{i+1}} |r_{tt}(s)| ds
\le\delta t \int_{t^i}^{t^{i+1}} 
( \|{\bf u}_t(s)\|^2 + \|{\bf u}(s)\| \|{\bf u}_{tt}(s)\| )ds.
\label{5.57-new}
\end{align}
The first pair of parentheses in \eqref{5.53-new} can be split as
\begin{align}
& \Bigl|
\|\nabla {\bf u}(t^{i+1})\|^2
-
\frac{r^{i+1}}{E(\bar{\bf u}^{i+1})+\bar{C}}
\|\nabla \bar{\bf u}^{\,i+1}\|^2
\Bigr| \notag\\
\le &
\|\nabla {\bf u}(t^{i+1})\|^2
\Bigl|
1-\frac{r^{i+1}}{E(\bar{\bf u}^{i+1})+\bar{C}}
\Bigr|
+\frac{r^{i+1}}{E(\bar{\bf u}^{i+1})+\bar{C}}
\Bigl|
\|\nabla {\bf u}(t^{i+1})\|^2
-
\|\nabla \bar{\bf u}^{\,i+1}\|^2
\Bigr|
=: W^i_1 + W^i_2.
\label{5.58}
\end{align}
As for $W^i_1$, using 
$r(t^{i+1})=E({\bf u}(t^{i+1}))+\bar{C}$ and 
$E({\bf u})+\bar{C}\ge 1$ for all ${\bf u}$  yields
\begin{align}
W^i_1
&\le \max_{t\in[0,T]} \|\nabla{\bf u}(t)\|^2 \cdot
\left|
1-\frac{r^{\,i+1}}{E(\bar{\bf u}^{\,i+1})+\bar{C}}
\right|
\notag\\
&\le \max_{t\in[0,T]} \|\nabla{\bf u}(t)\|^2 \cdot 
\left( 
\left|
\frac{r(t^{\,i+1})}{E({\bf u}(t^{\,i+1}))+\bar{C}}
-
\frac{r^{\,i+1}}{E({\bf u}(t^{\,i+1}))+\bar{C}}
\right|
+ 
\left|
\frac{r^{\,i+1}}{E({\bf u}(t^{\,i+1}))+\bar{C}}
-
\frac{r^{\,i+1}}{E(\bar{\bf u}^{\,i+1})+\bar{C}}
\right| \right)
\notag\\
&\le \max_{t\in[0,T]} \|\nabla{\bf u}(t)\|^2 \cdot
\Bigl(|s_{i+1}|+ 
 |r^{i+1}|\cdot \bigl|E({\bf u}(t^{\,i+1}))-E(\bar{\bf u}^{\,i+1})\bigr| 
\Bigr)
\notag\\
&\le C_{301} \cdot 
\Bigl(|s_{i+1}|+ 
 \bigl|E({\bf u}(t^{\,i+1}))-E(\bar{\bf u}^{\,i+1})\bigr| 
\Bigr),
\label{5.59-new}
\end{align}
where 
\begin{align}
C_{301}\triangleq \max_{t\in[0,T]} \|\nabla{\bf u}(t)\|^2 \cdot \max\{1, M_T\}.
\label{def_C301}
\end{align}
Note $|r^{i+1}|\le M_T$ from Theorem\,\ref{thm:HS2023_Thm6} is used to get \eqref{5.59-new}.
As for $W_2^i$, it follows from \eqref{5.52-new}, using $E(\bar{\mathbf u})+\bar C>1$ for all $\bar{\mathbf u}$ 
and $r^{i+1}\le M_T$ that
\begin{align}
W_2^i
&\le M_T \Bigl|  \| \nabla \bar{\bf u}^{\,i+1} \|^2
      - \| \nabla {\bf u}(t^{\,i+1}) \|^2 \Bigr| 
      \notag\\
&\le M_T \, \| \nabla \bar{\bf u}^{\,i+1}
           - \nabla {\bf u}(t^{\,i+1}) \|
      \cdot \bigl( \| \nabla \bar{\bf u}^{\,i+1} \|
            + \| \nabla {\bf u}(t^{\,i+1}) \| \bigr) 
            \notag\\
&\le C_{302} \cdot \| \nabla \bar{\bf e}^{\,i+1} \|,
\label{5.60-new}
\end{align}
where
\begin{equation}
C_{302} \triangleq  M_T \bigl(\sqrt{C_3} + \max_{t\in[0,T]} \|\nabla{\bf u}(t)\| \bigr).
\label{def_C302}
\end{equation}
The second pair of parentheses in \eqref{5.53-new} divides to
\begin{align}
&\Biggl|
\frac{r^{i+1}}{E(\bar{\bf u}^{i+1}) + \bar C}
\bigl( {\bf f}^{i+1}, \bar{\bf u}^{i+1} \bigr)
- \bigl( {\bf f}(t^{i+1}), {\bf u}(t^{i+1}) \bigr)
\Biggr|  \notag\\
\le
& \bigl| ( {\bf f}(t^{i+1}), {\bf u}(t^{i+1}) ) \bigr|
\cdot \Biggl|
1 - \frac{r^{i+1}}{E(\bar{\bf u}^{i+1}) + \bar C}
\Biggr|  
+ \frac{|r^{i+1}|}{E(\bar{\bf u}^{i+1}) + \bar C}
\Biggl|
( {\bf f}^{i+1}, \bar{\bf u}^{i+1} )
- ( {\bf f}(t^{\,i+1}), {\bf u}(t^{i+1}) )
\Biggr| \notag \\
=:& W_3^i + W_4^i.
\label{5.61}
\end{align}
The analysis of $W_3^i$ is similar to $W_1^i$ and is given by
\begin{align}
W_3^i \le \max_{t\in [0,T]} \|{\bf f} \|
\max_{t\in [0,T]} \|{\bf u} \| 
\left|
1-\frac{r^{\,i+1}}{E(\bar{\bf u}^{\,i+1})+\bar{C}}
\right|
\le C_{303} \cdot 
\Bigl(|s_{i+1}|+ 
 \bigl|E({\bf u}(t^{\,i+1}))-E(\bar{\bf u}^{\,i+1})\bigr| 
\Bigr),
\label{5.62-new}
\end{align}
where 
\begin{align}
C_{303}\triangleq \max_{t\in [0,T]} \|{\bf f} \|\cdot
\max_{t\in [0,T]} \|{\bf u} \| \cdot 
\max\{1, M_T\}.
\label{def_C303}
\end{align}
As for $W_4^i$, using ${\bf f}^{i+1}={\bf f}(t^{i+1})$ and $|r^{i+1}|\le M_T$, we obtain
\begin{align}
W_4^i \le  C_{304} \|\bar{\bf e}^{i+1}\|,
 \quad \text{ where } 
 C_{304}\triangleq M_T \max_{t\in[0,T]} \|{\bf f}\|.
 \label{5.63-new}
\end{align}
As for $\bigl|E({\bf u}(t^{\,i+1}))-E(\bar{\bf u}^{\,i+1})\bigr|$, one has
\begin{align}
\bigl|E({\bf u}(t^{\,i+1}))-E(\bar{\bf u}^{\,i+1})\bigr| 
&\le \frac{1}{2} \bigl(\|{\bf u}(t^{i+1})\| + 
\| \bar{\bf u}^{i+1} \| \bigr)
\cdot 
\| {\bf u}(t^{i+1}) - \bar{\bf u}^{i+1}\|
\notag\\
& \le C_{305} \|\bar{\bf e}^{i+1} \|,
\label{5.64-new}
\end{align}
where
\begin{align}
C_{305}\triangleq \frac{1}{2}
\left( 
\max_{t\in [0,T]} \|{\bf u} \| 
+ C\sqrt{C_3} \right).
\label{def_C305}
\end{align}
Here, the Poincar\'{e} inequality \eqref{Sobole-inequality}  
$\| \bar{\bf u}^{i+1} \|\le C \|\nabla \bar{\bf u}^{i+1} \|$
and $\|\nabla \bar{\bf u}^{i+1} \|\le \sqrt{C_3}$ 
from \eqref{5.52-new} are used.

Inserting \eqref{5.58} and \eqref{5.61} along with the estimates for $|T^i|$, $W^i_j$ for $j=1,2,3,4$, 
and  $\bigl|E({\bf u}(t^{\,i+1}))-E(\bar{\bf u}^{\,i+1})\bigr|$ to \eqref{5.55-new} yields
\begin{align}
|s^{n+1}|
\le & \nu \delta t
\sum_{i=0}^n (W_1^i + W_2^i)  
+ \delta t \sum_{i=0}^n (W_3^i+ W_4^i)
+ \sum_{i=0}^n |T^i| 
\notag \\
\le &
(\nu C_{301} + C_{303}) \delta t \sum_{i=0}^n |s^{i+1}| 
+ \nu C_{302} \delta t 
\sum_{i=0}^n \|\nabla\bar{\bf e}^{i+1} \| 
+  (\nu C_{301}C_{305} + C_{303} C_{305} + C_{304} ) \delta t \sum_{i=0}^n \|\bar{\bf e}^{i+1}\| 
\notag \\
& + \delta t \int_0^T (\|{\bf u}_t(s) \|^2 + 
  \|{\bf u}(s) \|  \|{\bf u}_{tt}(s) \| ) ds
\notag \\
\le &   
C_{311} (1+\nu) \delta t \sum_{i=0}^n |s^{i+1}| 
+ 
C_{312} (1+\nu) \delta t \sum_{i=0}^n 
\|\nabla\bar{\bf e}^{i+1} \| 
+C_{313} T \delta t
\notag \\
\le &   
(1+\nu) C_{320} \delta t \sum_{i=0}^n |s^{i+1}| 
+ (1+\nu)\, C_{320}\,  \delta t^2 
\sqrt{1+C^4_0} 
 +C_{320} \delta t.
\label{5.64-after-new}
\end{align}
where the Poincar\'{e} inequality $\|\bar{\bf e}^{i+1}\| \le C \| \nabla \bar{\bf e}^{i+1}\|$ is applied to get the third inequality, 
the estimate $\|\nabla\bar{\bf e}^{i+1} \| \le \sqrt{C_2 \delta t^4 (1+C_0^4)}$ from \eqref{5.51-new} is used in the last step, and 
\begin{align}
C_{311} &\triangleq \max\{C_{301},C_{303}  \},
\notag\\
C_{312} &\triangleq \max\{ C_{302} + C C_{301} C_{305}, 
  C C_{303} C_{305} + C C_{304}  \},
\notag\\
C_{313} &\triangleq \max_{s\in [0,T]} 
\bigl( \|{\bf u}_t(s) \|^2 + \|{\bf u}(s) \|  \|{\bf u}_{tt}(s) \| \bigr),
\notag\\
C_{320} &\triangleq \max\{C_{311}, C_{312} T \sqrt{C_2}, C_{313}T  \}
\end{align}

Applying the discrete Gronwall inequality (Lemma~1 of \cite{HuangShen2023}; the implicit-Gronwall variant: if $a^{m+1}\le \alpha\,\delta t\sum_{n=0}^{m}a^n + \gamma$ with $\alpha\,\delta t<1$, then $a^{m+1}\le \gamma\exp(\alpha T/(1-\alpha\delta t))$ for $m\delta t\le T$) yields
\begin{align}
|s^{n+1}| \le& 
C_{320} \delta t \cdot \exp\bigg(\frac{(1+\nu)C_{320}T}{1-\delta t (1+\nu) C_{320}} \bigg)
\cdot 
\bigg[ (1+\nu) \delta t \sqrt{1+C_0^4} +1 \bigg].
\label{cat0128}
\end{align}
Define
\begin{align}
C_4 := C_{320} \max\bigg\{ \exp\bigg( 2C_{320} T(1+\nu)\bigg), 
2(1+\nu)\bigg\}.
\label{def_C4}
\end{align}
and take the time step constraint 
\begin{align}
\delta t < \frac{1}{C_4}.
\label{5.67-new}
\end{align}
Since $C_{320}$ contains $\sqrt{C_2}$ with $C_2\sim \exp(\exp(\mathcal{O}(\nu^{-4})))$ from Step~2, the factor $\exp(2C_{320}T(1+\nu))$ in \eqref{def_C4} yields $C_4\sim \exp(\exp(\exp(\mathcal{O}(\nu^{-4}))))$, the third of the three nested exponentials traced in Remark~\ref{remark4.5}.

Then $\delta t< \frac{1}{2C_{320} (1+\nu)}$ and \eqref{cat0128} reduces to
\begin{align}
|s^{n+1}|
\le & C_{320}\delta t\cdot  \exp\big(2C_{320} T(1+\nu)\big)  
\cdot
\bigg[ (1+\nu)\delta t\sqrt{1+C_0^4} +1 \bigg]
\notag \\
\le &C_4 \delta t \bigg[ 1+ (1+\nu)\delta t\sqrt{1+C_0^4} \bigg]
\label{5.65-new}
\end{align}

Note $1-\xi^{n+1}=1-\frac{r^{n+1}}{E(\bar{\bf u}^{n+1})+\bar{C}}$, thus we use the analysis result from \eqref{5.59-new}, \eqref{5.65-new}, and \eqref{5.64-new}+Poincar\'{e} inequality+\eqref{5.51-new} obtain
\begin{align}
|1-\xi^{n+1}| \le &
\max\{1,M_T\}\cdot \left( |s^{n+1}| + 
|E({\bf u}(t^{n+1})) - E(\bar{\bf u}^{n+1})| \right)
\notag\\
\le& \max\{1,M_T\}\cdot 
\left[
C_4\delta t \left(
1+ (1+\nu)\delta t \sqrt{1+C_0^4}  \right)
+
C C_{305} \sqrt{C_2(1+C_0^4)} \delta t^2
\right]
\notag\\
\le & C_5 \delta t \bigg(
\delta t\sqrt{1+C^4_0} + 1
\bigg),
\label{5.68-new}
\end{align}
where 
\begin{align}
C_5 \triangleq \max\{1, M_T\}
\cdot (C_4(1+\nu) + C C_{305} \sqrt{C_2} ) > C_4.
\end{align}
Choosing 
\begin{equation}
C_0 = \max\{ 2C_5, 1.001\}
\label{def_C0}
\end{equation}
to ensure $C_0>1$ 
and the time step
\begin{align}
\delta t \le \frac{1}{1+2C_0^2},
\label{5.71-new}
\end{align}
then
\begin{align}
\delta t \le \frac{1}{1+8C_5^2}
\le \frac{1}{4C_5} <
\frac{1}{4C_4} < \frac{1}{C_4},
\end{align}
which is consistent with \eqref{5.67-new}
and on the other hand, \eqref{5.68-new} becomes
\begin{align}
|1-\xi^{n+1}|
\le \frac{C_0}{2} \delta t \bigg(
1 + \frac{\sqrt{1+C^4_0}}{1+C_0^2} \bigg)
\le C_0 \delta t.
\label{final_goal}
\end{align}
Then the induction process is completed.

It remains to convert the bounds on the barred error
$\bar{\bf e}^{i+1}$ into bounds on the unbarred error
${\bf e}^{i+1}$ and to estimate the pressure error~$e^i_p$.
From ${\bf u}^{n+1}=\eta^{n+1}\bar{\bf u}^{n+1}$ we have
\begin{equation}
{\bf e}^{n+1}={\bf u}^{n+1}-{\bf u}(t^{n+1})=\bar{\bf e}^{n+1}+(\eta^{n+1}-1)\bar{\bf u}^{n+1},
\label{cat2123}
\end{equation}
so
$\nabla{\bf e}^{n+1}=\nabla\bar{\bf e}^{n+1}+(\eta^{n+1}-1)\nabla\bar{\bf u}^{n+1}$.
Using $|\eta^{n+1}-1|=|1-\xi^{n+1}|^2\le C_0^2\delta t^2$ from
\eqref{final_goal} and $(a+b)^2\le 2a^2+2b^2$,
\[
\|\nabla{\bf e}^{n+1}\|^2
\le 2\|\nabla\bar{\bf e}^{n+1}\|^2 + 2C_0^4\delta t^4\,\|\nabla\bar{\bf u}^{n+1}\|^2.
\]
The first term is bounded by $2C_2(1+C_0^4)\delta t^4$ via
\eqref{5.51-new}, and the second by $2C_3C_0^4\delta t^4$ via
\eqref{5.52-new}. This yields
\begin{align}
\| \nabla {\bf e}^{n+1}\|^2 &\le
 [2C_2 (1+C_0^4) + 2C_3 C_0^4 ]\delta t^4.
\label{5.74-new}
\end{align}

The identity \eqref{cat2123} at index $i+1$ gives
$\Delta{\bf e}^{i+1}=\Delta\bar{\bf e}^{i+1}+(\eta^{i+1}-1)\Delta\bar{\bf u}^{i+1}$,
hence
\[
\|\Delta{\bf e}^{i+1}\|^2
\le 2\|\Delta\bar{\bf e}^{i+1}\|^2 + 2C_0^4\delta t^4\,\|\Delta\bar{\bf u}^{i+1}\|^2.
\]
Multiplying by $\nu\delta t$ and summing over $i=0,\dots,n$, the first
sum is bounded by $C_2(1+C_0^4)\delta t^4$ via \eqref{5.51-new} and the
second by $C_1$ via \eqref{5.50-new}, so
\begin{align}
\nu \delta t \sum_{i=0}^{n} \|\Delta{\bf e}^{i+1} \|^2
&\le (2C_2  (1+C_0^4) + 2C_1C_0^4) \delta t^4.
\label{5.76-new}
\end{align}

We now derive the bound on $\tfrac{\delta t}{\nu}\sum\|\nabla e^i_p\|^2$ in three
substeps. The weight $\delta t/\nu$, rather than $\nu\delta t$, is dictated
by physical dimensions: $\|\nabla p\|^2$ scales as $\nu^2\|\Delta u\|^2$
through the Stokes-pressure identity, so $\tfrac{\delta t}{\nu}\sum\|\nabla
e^i_p\|^2$ and $\nu\delta t\sum\|\Delta\bar{\bf e}^i\|^2$ share the same
physical dimensions and can be summed.
Setting $q=e^i_p$ in \eqref{5.37-new} and applying
Cauchy--Schwarz with Young's inequality (with weights $1/4,1/4$ on
$\|\nabla e^i_p\|$) gives
\begin{equation}
\|\nabla e^i_p\|^2
\le 2\|{\bf u}(t^i)\!\cdot\!\nabla{\bf u}(t^i)-\bar{\bf u}^i\!\cdot\!\nabla\bar{\bf u}^i\|^2
+ 2\nu^2\|\nabla p_s(\bar{\bf e}^i)\|^2.
\label{5.77a-new}
\end{equation}
The convection difference in \eqref{5.77a-new} is bounded by the
unbarred analog of \eqref{5.41-new} (i.e., the same identity \eqref{5.40-new}
applied without the $(k+1)$ factor, followed by Cauchy--Schwarz, elliptic
regularity, and Young's inequality):
\begin{equation}
\|{\bf u}(t^i)\!\cdot\!\nabla{\bf u}(t^i)-\bar{\bf u}^i\!\cdot\!\nabla\bar{\bf u}^i\|^2
\le 2C^4\|\nabla\bar{\bf e}^i\|^2\|\Delta\bar{\bf u}^i\|^2
+ 2\|{\bf u}(t^i)\|_2^2\|\nabla\bar{\bf e}^i\|^2.
\label{5.77b-new}
\end{equation}
The Stokes-pressure term in \eqref{5.77a-new} is bounded by
\eqref{lemma4-Liu} with $\varepsilon=1/2$:
\begin{equation}
\|\nabla p_s(\bar{\bf e}^i)\|^2\le \|\Delta\bar{\bf e}^i\|^2 + C_S(\tfrac12)\,\|\nabla\bar{\bf e}^i\|^2.
\label{5.77c-new}
\end{equation}
Substituting \eqref{5.77b-new} and \eqref{5.77c-new} into \eqref{5.77a-new},
and using $\|{\bf u}(t^i)\|_2^2 \le C_1$ from \eqref{def_C1}, gives the
pointwise bound
\begin{equation}
\|\nabla e^i_p\|^2
\le 4C^4\|\nabla\bar{\bf e}^i\|^2\|\Delta\bar{\bf u}^i\|^2
+ 4 C_1\,\|\nabla\bar{\bf e}^i\|^2
+ 2\nu^2\|\Delta\bar{\bf e}^i\|^2
+ 2\nu^2 C_S(\tfrac12)\,\|\nabla\bar{\bf e}^i\|^2.
\label{5.77d-new}
\end{equation}
Multiplying \eqref{5.77d-new} by $\tfrac{\delta t}{\nu}$ and summing over
$i=0,\dots,n+1$ produces four sums on the right-hand side, each of which is
bounded as follows.
First, applying the pointwise form of \eqref{5.51-new} (at each
$m\le n+1$) and using $\nu\delta t\sum\|\Delta\bar{\bf u}^i\|^2\le C_1$ from \eqref{5.50-new},
\[
\frac{\delta t}{\nu}\sum_{i=0}^{n+1}\|\nabla\bar{\bf e}^i\|^2\|\Delta\bar{\bf u}^i\|^2
\le \frac{1}{\nu}\max_{i}\|\nabla\bar{\bf e}^i\|^2 \cdot \delta t\sum_{i=0}^{n+1}\|\Delta\bar{\bf u}^i\|^2
\le \frac{C_1 C_2}{\nu^2}(1+C_0^4)\,\delta t^4.
\]
Second, using $(n+2)\delta t\le T+1$ together with the
pointwise bound from \eqref{5.51-new},
\[
\frac{\delta t}{\nu}\sum_{i=0}^{n+1}\|\nabla\bar{\bf e}^i\|^2
\le \frac{T+1}{\nu}\cdot \max_{i}\|\nabla\bar{\bf e}^i\|^2
\le \frac{(T+1) C_2}{\nu}(1+C_0^4)\,\delta t^4.
\]
Third, from \eqref{5.51-new}, $\frac{\delta t}{\nu}\sum\nu^2\|\Delta\bar{\bf e}^i\|^2 = \nu\delta t\sum\|\Delta\bar{\bf e}^i\|^2$,
\[
\frac{\delta t}{\nu}\sum_{i=0}^{n+1}\nu^2\|\Delta\bar{\bf e}^i\|^2
= \nu\delta t\sum_{i=0}^{n+1}\|\Delta\bar{\bf e}^i\|^2
\le C_2(1+C_0^4)\,\delta t^4.
\]
Fourth, $\frac{\delta t}{\nu}\sum\nu^2\|\nabla\bar{\bf e}^i\|^2 = \nu\,\delta t\sum\|\nabla\bar{\bf e}^i\|^2 \le \nu(T+1)\max_i\|\nabla\bar{\bf e}^i\|^2$, hence
\[
\frac{\delta t}{\nu}\sum_{i=0}^{n+1}\nu^2 C_S(\tfrac12)\,\|\nabla\bar{\bf e}^i\|^2
\le \nu(T+1) C_S(\tfrac12) C_2(1+C_0^4)\,\delta t^4.
\]
Inserting these into the multiplied-and-summed form of \eqref{5.77d-new}
yields
\begin{align}
\frac{\delta t}{\nu}\sum_{i=0}^{n+1}\|\nabla e^i_p\|^2
\le \delta t^4\,
\Big[
\frac{4 C^4 C_1}{\nu^2} \;+\; \frac{4 C_1 (T+1)}{\nu} \;+\; 2 \;+\; 2\nu C_S(\tfrac12)(T+1)
\Big]\,
C_2(1+C_0^4).
\label{5.78-new}
\end{align}
The dominant $1/\nu^2$ factor is absorbed into the negative-power $\nu$-dependence of $C_{5.3}$ documented in Remark~\ref{remark4.5}; the triple-exponential chain is unaffected.

\medskip
The conclusion \eqref{5.3-new} now follows from \eqref{5.51-new},
\eqref{5.74-new}, \eqref{5.76-new}, and \eqref{5.78-new}. The constant
$C_{5.3}$ in \eqref{5.3-new} is taken as the maximum of the coefficients
of $\delta t^4$ appearing in these four inequalities.

\end{proof}

\begin{remark}
\label{remark4.5}
The dependence of $C_{5.3}$ on $\nu$ as $\nu\to 0^+$ can be made explicit by
tracking the chain of constants in the proof. With the choices
$\varepsilon_1=\nu\,D_{\nu_{\max}}$ in \eqref{def_Ceps} and
$\varepsilon_3=E_\varepsilon\,\nu$ in the bound on $C_{260}$, the
leading-order behavior at each level is
\begin{align*}
&C_{100}=O(\nu^{-3}),\qquad
C_1=O\!\Big(\nu^{-3}\exp\!\big(\nu^{-4}\big)\Big),\\
&C_2,\,C_3
=O\!\Big(\exp\!\big(\exp\!\big(\nu^{-4}\big)\big)\Big),\qquad
C_4,\,C_5,\,C_0
=O\!\Big(\exp\!\big(\exp\!\big(\exp\!\big(\nu^{-4}\big)\big)\big)\Big),
\end{align*}
and consequently
\begin{align}
C_{5.3}
=O\!\Big(\exp\!\big(\exp\!\big(\exp\!\big(\nu^{-4}\big)\big)\big)\Big).
\label{C53_viscosity}
\end{align}
The triple-exponential growth originates from three nested applications of
the discrete Gronwall inequality: once for the velocity stability estimate
\eqref{5.24-new}, once for the velocity error estimate \eqref{5.51-new},
and once for the GSAV consistency $|s^{n+1}|$ in Step~3. Each application
multiplies the existing $\nu$-dependence into an exponent. As a result,
the time-step constraint $\delta t\le 1/(1+2C_0^2)$ inherits the same
triple-exponential dependence: $\delta t$ must shrink at least as fast as
\[
O\!\Big(\exp\!\big(-\exp\!\big(\exp\!\big(\nu^{-4}\big)\big)\big)\Big).
\]
to guarantee convergence as $\nu\to 0$.
\end{remark}

\section{Numerical Simulations}
\label{sec_numerical}
In this section, we use numerical simulations to investigate the temporal accuracy and robustness with respect to viscosity of the GSAV scheme \eqref{HSscheme} for the full Navier--Stokes equations with no-slip boundary conditions. The spatial discretization is based on the Legendre spectral Galerkin method as described in \cite{HuangShen2023, shen2011spectral}; hence, we refer to the resulting approach as the GSAV-Spectral method.

\subsection{Example 1: temporal accuracy when $k=4$}
\label{subsec:Eg1}
In Example 1, we examine the temporal convergence of the GSAV-Spectral method with time-step shift parameter $k=4$, using the following exact solution on $\Omega=[-1,1]^2$ (taken from \cite{HuangShen2023}):
\begin{align*}
u_{1}(x,y,t) &=  \sin(2\pi y)\,\sin^{2}(\pi x)\,\sin t, \\
u_{2}(x,y,t) &= -\sin(2\pi x)\,\sin^{2}(\pi y)\,\sin t, \\
p(x,y,t)     &=  \cos(\pi x)\,\sin(\pi y)\,\sin t.
\end{align*}
The spectral method uses 64 modes in each direction.
The constant $\bar{C}$ in the modified energy is taken as $1000$, which is consistent with the conditions in Theorem~\ref{Theorem 7-new}. We take  the viscosity $\nu=10^{-2}$ and final time $T=1$. The time step sizes $\delta t$ is decreased from $10^{-2}$ to $10^{-5}$. The $L^2$-errors of the velocity and pressure at $T=1$ are plotted in Fig.~\ref{fig:Eg1_dt}, which verifies the second-order convergence in time as proved in Theorem\,\ref{Theorem 7-new}. The figure also indicates that the error for $k=4$ is slightly smaller than that for $k=5$.
\begin{figure}[H]
\centering
\includegraphics[width=0.95\textwidth]{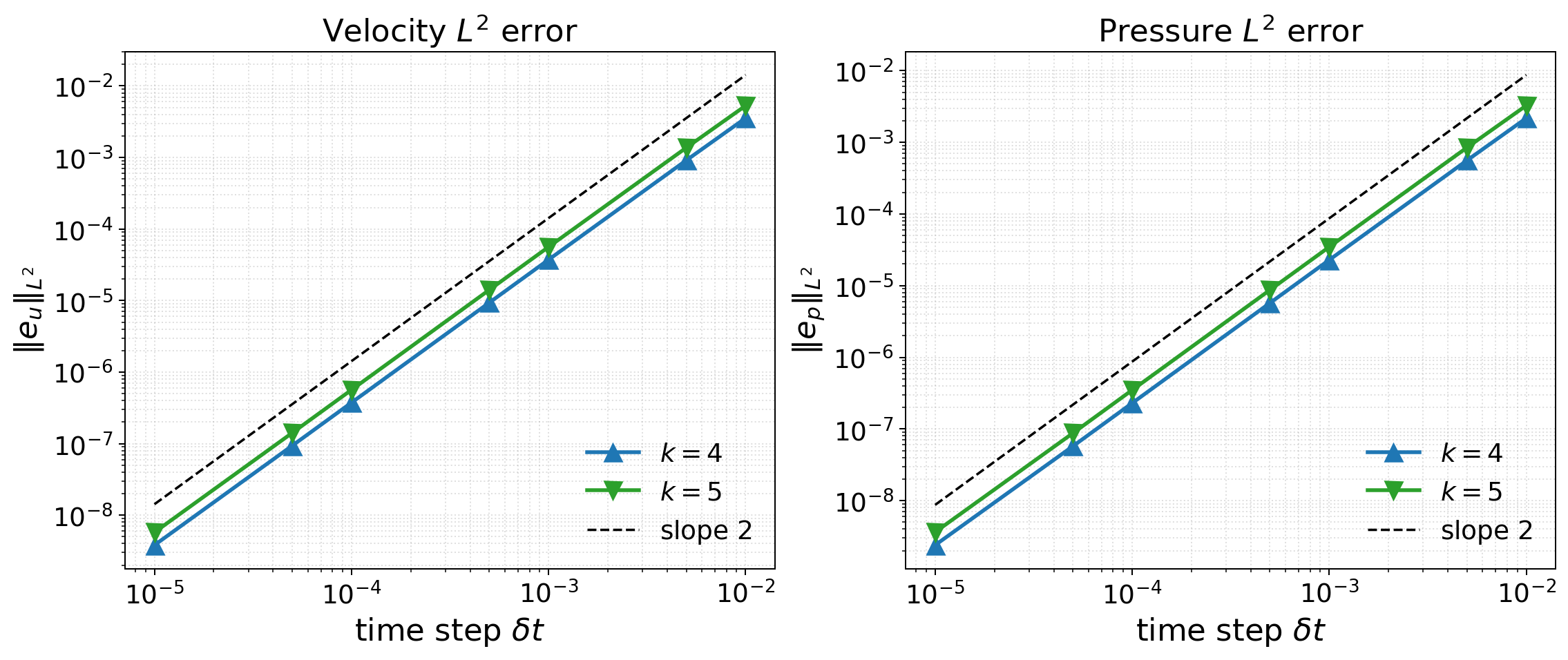}
\caption{Example~1, $\delta t$-convergence of GSAV-Spectral method at $\nu=10^{-2}$ and $T=1$ 
in $\|e_u\|_{\Ltwo}$ (left) and $\|e_p\|_{\Ltwo}$ (right).}
\label{fig:Eg1_dt}
\end{figure}


\subsection{Example 2: non-robustness on perturbed Kovasznay problem}
\label{subsec:Eg2}
In Example 2, we solve a perturbed Kovasznay problem \cite{Kovasznay1948} on 
$\Omega = (-1,1)^2$ with the steady state solution $({\bf u}_K, p_K)$ where 
\begin{align}
u_{K,1}
&=
1-e^{\lambda x}\cos(2\pi y),
\qquad
u_{K,2}
=
\frac{\lambda}{2\pi}e^{\lambda x}\sin(2\pi y),
\qquad
p_K(x,y)
=
\frac{1-e^{2\lambda x}}{2}.
\label{Kovasznay_solution}
\end{align}
where the Reynolds number $Re=\frac1{\nu}$ and 
$\lambda= \frac{Re}{2}
-
\sqrt{\frac{Re^2}{4}+4\pi^2}$.
The boundary condition is given by 
$u|_{\partial\Omega}=u_K$.
The initial condition is given by 
$u(x,y,0)=u_K(x,y)+u_{\rm pert}(x,y)$ 
where
\begin{align*}
u_{\rm pert}
=
\bigl(
\partial_y\psi,\,
-\partial_x\psi
\bigr), \quad
\psi=A(1-x^2)^2(1-y^2)^2
\end{align*}
Here, $A$ is the perturbation parameter and set as $10^{-2}$ in the simulations. In the full Navier-Stokes equations \eqref{ns1}, \eqref{ns2}, the external force is zero.  

To study the impact of viscosity on the GSAV scheme \eqref{HSscheme}, we consider Reynolds numbers ranging from $10^3$ to $10^5$, as reported in Table~\ref{tab:Kov_compare}. Several numerical methods are tested for this problem.
For the GSAV-Spectral method, we use the time-step shift parameter $k=4$, SAV constant $\bar C=1$, and $128$ spectral modes in each spatial direction. The time step sizes are chosen as $\delta t=10^{-4},10^{-5}$, with a smaller step size $\delta t=10^{-6}$ additionally used for the cases $Re=8\times 10^4$ and $10^5$.

The second method is a FEM-Newton method where the time discretization is 
\begin{align}
\frac{3{\bf u}^{n+1}-4{\bf u}^n+{\bf u}^{n-1}}{2\Delta t}
-\nu \Delta {\bf u}^{n+1}
+({\bf u}^{n+1}\cdot \nabla){\bf u}^{n+1}
+\frac12 (\nabla\cdot {\bf u}^{n+1}){\bf u}^{n+1}
&=
\gamma \nabla(\nabla \cdot {\bf u}^{n+1})
-\nabla p^{n+1}
+ {\bf f}^{n+1},
\label{eq:BDF2-NSE}
\\
\nabla \cdot {\bf u}^{n+1}&=0.
\label{eq:BDF2-div}
\end{align}
Here, $\gamma>0$ is the grad-div stabilization parameter and is taken as $0.25$ in this work. The spatial discretization uses the Taylor-Hood P2/P1 finite elements with 64 uniform subdivisions in each spatial direction. The time step size $\delta=10^{-3}$. As proved in \cite[Theorem~3]{GarciaArchillaJohnNovo2025}, the scheme is robust with $L^2$ error of the velocity is second order accurate in space and 2nd order accurate in time. 
The nonlinear problem is solved with a monolithic Newton's method with preconditioners described in \cite{CharnyiHeisterOlshanskiiRebholz2017}.

In addition, to investigate the effect of the convection term, we apply the schemes \eqref{3.5} and \eqref{3.6} to the time-dependent Stokes problem associated with the perturbed Kovasznay flow, using the same initial and boundary conditions but omitting the convection term. The spatial discretization again employs the spectral method with 128 modes in each direction, and the time step size is set to $\delta t=10^{-4}$.

Table~\ref{tab:Kov_compare} reports the velocity $L^{2}$ errors $\|{\bf u} - {\bf u}_K\|$ for these schemes over Reynolds numbers ranging from $10^4$ to $10^5$ at $T=1$. The GSAV-Spectral method for the full Navier--Stokes equations blows up when $Re\ge 2\times 10^4$ with $\delta t=10^{-4}$. Reducing the time step to $\delta t=10^{-5}$ improves the results, but blow-up still occurs for $Re\ge 8\times 10^4$. Even with a smaller time step $\delta t=10^{-6}$, the same blow-up behavior persists when $Re\ge 8\times 10^4$. 
We note that near the moment of blow-up, the  intermediate velocity $\bar{\mathbf u}^{n+1}$ and the pressure $p^{n+1}$ diverge to numerical infinity, while $\xi^{n+1}$ and $\eta^{n+1}$ collapse to nearly zero, which forces the corrected velocity $\mathbf u^{n}=\eta^{n}\bar{\mathbf u}^{n}$ to vanish; the weak-stability bound of Theorem~\ref{thm:HS2023_Thm6} is therefore preserved in a trivial sense.

In contrast, when the convection term is removed, the Spectral-Stokes solver produces smooth results for all tested Reynolds numbers. Furthermore, the FEM-Newton method with a fully implicit treatment of the convection term also eliminates all blow-ups. Figure~\ref{fig:Kov_surf_v_Re1e5} displays the plots of the second velocity component at $t=0.05$ for $Re=10^{5}$. At this time instant, the GSAV-Spectral method develops unphysical oscillations that eventually lead to blow-up, whereas the FEM-Newton method produces a smooth solution.

\begin{table}[H]
\centering
\footnotesize
\setlength{\tabcolsep}{6pt}
\renewcommand{\arraystretch}{1.2}
\caption{Comparison of three schemes for the perturbed Kovasznay problem.  This tables shows the velocity $L^{2}$ errors $\|{\bf u} - {\bf u}_K\|$ at $T=1$.}
\label{tab:Kov_compare}
\begin{tabular}{r|c|c|c|c}
\hline
\rule{0pt}{2.6ex}
$Re$
  & \shortstack{FEM-Newton-NSE \\ $\delta t=10^{-3}$}
  & \shortstack{GSAV-Spectral-NSE \\ $\delta t=10^{-4}$}
  & \shortstack{GSAV-Spectral-NSE \\ $\delta t=10^{-5}$}
  & \shortstack{Spectral-Stokes \\ $\delta t=10^{-4}$} \\[1.2ex]
\hline
$10^4$  & $3.24\!\times\!10^{-2}$ & $3.85\!\times\!10^{-2}$ & $3.84\!\times\!10^{-2}$ & $1.99\!\times\!10^{-2}$ \\
$2\times 10^4$  & $3.31\!\times\!10^{-2}$ & Blow-up        & $3.89\!\times\!10^{-2}$ & $1.99\!\times\!10^{-2}$ \\
$4\times 10^4$  & $3.37\!\times\!10^{-2}$ & Blow-up  & $3.97\!\times\!10^{-2}$ & $1.99\!\times\!10^{-2}$ \\
$5\times 10^4$  & $3.39\!\times\!10^{-2}$ & Blow-up       & $4.01\!\times\!10^{-2}$ & $1.99\!\times\!10^{-2}$ \\
$8\times 10^4$  & $3.42\!\times\!10^{-2}$ & Blow-up        & Blow-up        & $1.99\!\times\!10^{-2}$ \\
$10^{5}$ & $3.43\!\times\!10^{-2}$ & Blow-up & Blow-up        & $1.99\!\times\!10^{-2}$ \\
\hline
\end{tabular}
\end{table}

\begin{figure}[H]
\centering
\includegraphics[width=0.48\textwidth]{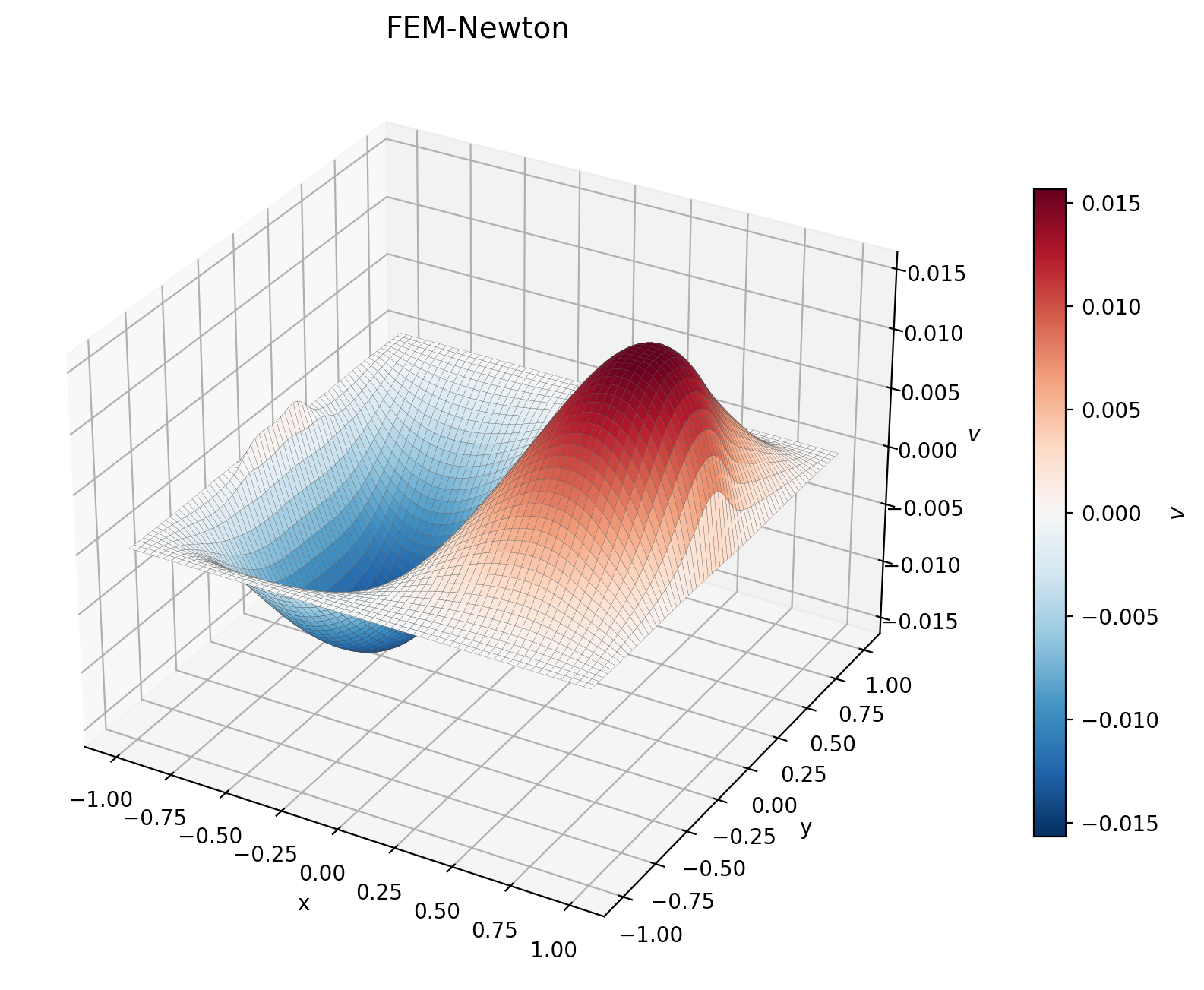}\hfill
\includegraphics[width=0.48\textwidth]{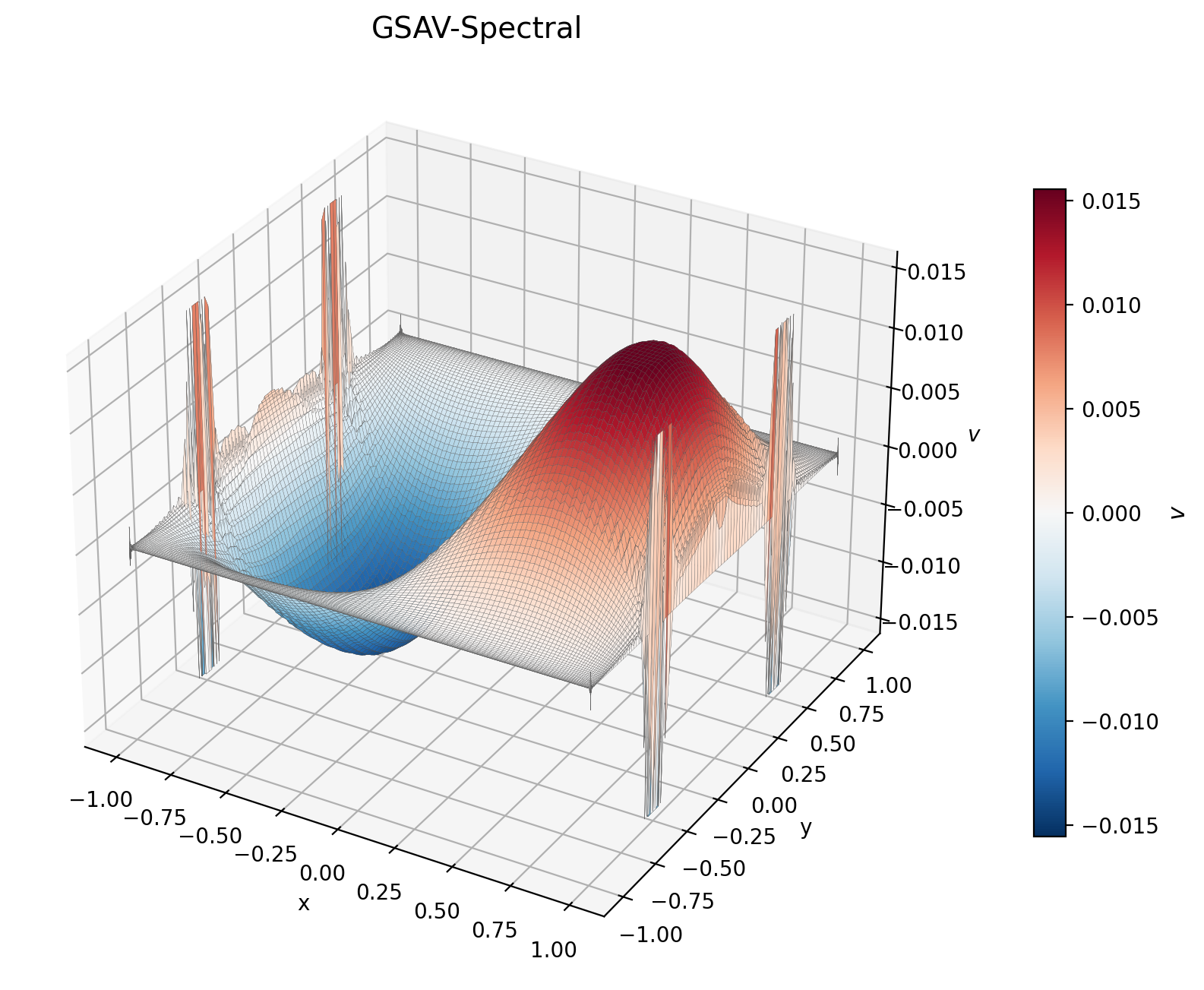}
\caption{Plots of the second velocity component $v$ for the Kovasznay-perturbed flow at $t=0.05$ when $Re=10^{5}$, $A=10^{-2}$.
Left: FEM-Newton method. Right: GSAV-Spectral method. }
\label{fig:Kov_surf_v_Re1e5}
\end{figure}

These results strongly indicate that the GSAV scheme \eqref{HSscheme} is not robust at high Reynolds number, and the limitation is intrinsic to its explicit treatment of the convection term.

\section{Conclusions}
\label{sec-discussion}
This work refines the analysis of the GSAV consistent-splitting scheme proposed in \cite{HuangShen2023} in two directions.  First, the lower bound on the key parameter $k$ is reduced from $5$ to $4$ in the stability analysis of the time-dependent Stokes equations (spatial dimension $d=2$ or $3$) and in the convergence analysis of the full Navier--Stokes equations ($d=2$).  One benefit of this reduction is a smaller error, as discussed in Remark~\ref{remark1} and observed numerically in Example~1 (Figure~\ref{fig:Eg1_dt}).
Second, the viscosity is extended from unity to arbitrary positive values in both problems.  For the Navier--Stokes equations, the resulting upper bound on the velocity error in the $H^{1}$ norm is found to depend on negative powers of the viscosity (see Theorem\,\ref{Theorem 7-new} and  Remark\,\ref{remark4.5}).

Following \cite{John2021}, a numerical scheme for the Navier--Stokes equations is called \emph{robust} if the $L^{2}$ error bound of the velocity contains no negative powers of the viscosity, and \emph{non-robust} otherwise.  A rough $L^{2}$ estimate can be obtained by combining the Poincar\'e inequality~\eqref{Sobole-inequality} with the $H^{1}$ bound~\eqref{5.3-new}, namely
\begin{align}
\|\mathbf{e}^{n}\|\;\le\; C\,C_{5.3}\,\delta t^{2}.
\end{align}
Although the constant $C_{5.3}$ involves negative powers of $\nu$ (see formula \eqref{C53_viscosity}), this estimate alone does not establish that the scheme is non-robust: $C_{5.3}$ is only an upper bound, and the inequality could be loose.
To resolve this, we design a carefully chosen perturbation of the Kovasznay flow with smooth and viscosity-independent data and compare the GSAV scheme of \cite{HuangShen2023} against a fully implicit Newton solver (taken as the robust reference) and against the time-dependent Stokes counterpart of the same GSAV discretization.  The experiment confirms that the GSAV scheme is not robust at high Reynolds number, and the comparison localizes the cause to its explicit treatment of the convection term.

\section*{Acknowledgments}
The authors would like to express their sincere gratitude to Dr.\,Fukeng Huang and Dr.\,Jie Shen of the Eastern Institute of Technology, Ningbo, China, for generously sharing their computer code and for their kind and invaluable support, without which the numerical part of this work would not have been possible.
J. Wu was partially supported by the National Science Foundation of the United States
(Grant No.DMS2104682 and DMS2309748). X.Zheng was partially supported by NSF grant DMS-2309747.

\appendix
-\section{Refined Theorem 1 of \cite{LiuLiuPego2007} or Lemma 4 of \cite{HuangShen2023}}
The purpose of this part is to refine the statement of Theorem~1 in \cite{LiuLiuPego2007}, which is cited as Lemma 4 in \cite{HuangShen2023}. In the original statement, the dependence of the
constant $C_S(\varepsilon)$ on the parameter $\varepsilon$ is not made explicit.
In the present work, we quantify this dependence. In particular, we show that
as $\varepsilon \to 0+$, the constant $C_S(\varepsilon)$ behaves like
$C/\varepsilon^3$, where $C$ depends only on the geometry of the domain.

\begin{theorem}[Refined Theorem 1 of \cite{LiuLiuPego2007}]
\label{refinedTheorem1LiuLiuPego}
Let $\Omega\subset\mathbb{R}^N$ ($N\ge 2$) be bounded, connected, with $C^3$ boundary.
For every $\varepsilon>0$ there exists $C_S(\varepsilon)>0$ such that for all
${\bf u}\in {\bf H}^2(\Omega)\cap {\bf H}^1_0(\Omega)$,
\begin{equation}\label{eq:thm}
\|\nabla p_s({\bf u})\|_{L^2(\Omega)}^2
\le \left(\frac12+\varepsilon\right)\|\Delta{\bf u}\|_{L^2(\Omega)}^2
+ C_S(\varepsilon)\|\nabla{\bf u}\|_{L^2(\Omega)}^2.
\end{equation}
When $\varepsilon\to 0+$, $C_S(\varepsilon)=C/\varepsilon^3$,  where the constant $C$ only relies on the domain $\Omega$. The concrete form of $C_S(\varepsilon)$ is given in \eqref{C_eps_arbi_eps} for arbitrary $\varepsilon>0$, and in \eqref{C_eps_small_eps} when $\varepsilon\to 0+$.
\end{theorem}

\medskip

\noindent
\textbf{Setup.}
Given ${\bf u}\in H^2\cap H^1_0$, define the Stokes pressure $p_S=p_S({\bf u})$ by
\begin{equation}\label{eq:ps-def}
\nabla p_S := (\Delta\mathcal{P}-\mathcal{P}\Delta){\bf u}.
\end{equation}
Let $\Phi(x)=dist(x,\partial\Omega)$, the distance function to the domain boundary. 
For any $s>0$,  set
\[
\Omega_s:=\{x\in\Omega:\Phi(x)<s\},\qquad \Omega_s^c:=\Omega\setminus\Omega_s.
\]
Note $\Omega_s$ is a tubular neighborhood of $\partial\Omega$ of thickness $s$. 
Since $\partial\Omega$ is $C^3$, there exists $s_0>0$ such that $\Phi(x)\in C^3(\Omega_{s_0})$. Next, we let $0<s<s_0$. 
Then ${\bf n}=-\nabla\Phi(x)$ is a  $C^2$ extension of the unit outward normal field into $\Omega_s$. 
Let $\rho\in C^\infty(\mathbb{R})$ satisfy $0\le \rho\le 1$,
$\rho\equiv 1$ on $[0,1/2]$, and $\rho\equiv 0$ on $[1,\infty)$, and set
$\xi(x):=\rho(\Phi(x)/s)$.
Then $\xi(x)$ is a smooth cutoff satisfying $0\le\xi\le 1$,
$\xi\equiv 1$ near $\partial\Omega$, and $\xi\equiv 0$ on $\Omega_s^c$.

\begin{lemma}[Cutoff decomposition and remainders]\label{lem:decomp-proof}
Define the tangential projector $T(x):=I-{\bf n}(x){\bf n}(x)^T$ on $\Omega_s$ and
set, for ${\bf u}\in {\bf H}^2(\Omega)\cap {\bf H}^1_0(\Omega)$,
\[
{\bf u}_\parallel := \xi\,T{\bf u},
\qquad
{\bf u}_\perp := {\bf u}-{\bf u}_\parallel.
\]
Then ${\bf u}={\bf u}_\perp+{\bf u}_\parallel$ and ${\bf u}_\parallel$ is supported
in $\Omega_s$. Moreover,
\begin{equation}\label{eq:Delta-u-par}
\Delta{\bf u}_\parallel
=
\xi\,T\,\Delta{\bf u} + R_2,
\end{equation}
\begin{equation}\label{eq:Delta-u-perp}
\Delta{\bf u}_\perp
=
(1-\xi)\Delta{\bf u} + \xi\,({\bf n}{\bf n}^T)\Delta{\bf u} + R_1,
\end{equation}
where the remainders are given explicitly by
\begin{align}
R_2 &:= 2\nabla \xi\cdot \nabla(T{\bf u})
      + (\Delta \xi)\,T{\bf u}
      + \xi\,\Delta(T{\bf u}) - \xi\,T\,\Delta{\bf u}, \label{eq:R2-def}\\
R_1 &:= -R_2. \label{eq:R1-def}
\end{align}
Finally, there exists a constant $C>0$, depending only on $\Omega$, such that
\begin{equation}\label{eq:Rbound-s}
\|R_1\|_{L^2(\Omega)}+\|R_2\|_{L^2(\Omega)}
\le
C_R\|\nabla{\bf u}\|_{L^2(\Omega)},
\quad \text{ where } C_R \triangleq \frac{2C}{\min\{1,s \}}.
\end{equation}
\end{lemma}

\begin{remark}
In \cite{LiuLiuPego2007}, the bound $\|R_1\| + \|R_2\| \le C_R \|\nabla{\bf u}\|$ is used in their proof of Claim 1 in Section 3.5, but is not proved. In addition,  the dependence of $C_R$ on $s$ is not given in \cite{LiuLiuPego2007}.
\end{remark}

\begin{proof}
\emph{Step 1: decomposition and support.}
By definition ${\bf u}={\bf u}_\perp+{\bf u}_\parallel$.
Since $\xi\equiv 0$ on $\Omega_s^c$, ${\bf u}_\parallel$ is supported in $\Omega_s$.

\emph{Step 2: formula for $\Delta{\bf u}_\parallel$.}
Compute using the product rule:
\[
\Delta(\xi\,T{\bf u})
=
(\Delta\xi)\,T{\bf u}
+2\nabla\xi\cdot\nabla(T{\bf u})
+\xi\,\Delta(T{\bf u}).
\]
Add and subtract $\xi\,T\Delta{\bf u}$ to obtain \eqref{eq:Delta-u-par} with
$R_2$ given by \eqref{eq:R2-def}.

\emph{Step 3: formula for $\Delta{\bf u}_\perp$.}
Since ${\bf u}_\perp={\bf u}-{\bf u}_\parallel$, we have
$\Delta{\bf u}_\perp=\Delta{\bf u}-\Delta{\bf u}_\parallel$.
Insert \eqref{eq:Delta-u-par} and note that $I-T={\bf n}{\bf n}^T$:
\[
\Delta{\bf u}_\perp
=
\Delta{\bf u}-\xi\,T\,\Delta{\bf u}-R_2
=
(1-\xi)\Delta{\bf u}+\xi(I-T)\Delta{\bf u}-R_2,
\]
which is \eqref{eq:Delta-u-perp} with $R_1=-R_2$, i.e. \eqref{eq:R1-def}.

\emph{Step 4: estimate of the remainders.}
All terms in $R_2$ are supported in $\Omega_s$. We estimate each term in
\eqref{eq:R2-def}.

Since $\xi(x)=\rho(\Phi(x)/s)$ with $\rho$ fixed once and for all and $\Phi$ chosen
as the signed distance function in a tubular neighborhood of $\partial\Omega$,
there exists a constant $C=C(\Omega)>0$ such that
$\|\nabla\xi\|_{L^\infty}\le C/s$ and $\|\Delta\xi\|_{L^\infty}\le C/s^2$.
Also $T$ and its first two derivatives are bounded on $\Omega_{s_0}$ since
${\bf n}\in C^2$ there.

We use the thin-layer Poincar\'e estimate (valid since ${\bf u}\in H^1_0$):
\begin{equation}\label{eq:poincare-layer}
\|{\bf u}\|_{L^2(\Omega_s)} \le C s \|\nabla{\bf u}\|_{L^2(\Omega)}.
\end{equation}
Here, the constant $C$ also only depends on $\Omega$.

Now estimate:
\[
\|2\nabla\xi\cdot\nabla(T{\bf u})\|_{L^2}
\le \frac{C}{s}\bigl(\|\nabla{\bf u}\|_{L^2(\Omega_s)}+\|{\bf u}\|_{L^2(\Omega_s)}\bigr)
\le \frac{C}{s}\|\nabla{\bf u}\|_{L^2(\Omega)} + C\|\nabla{\bf u}\|_{L^2(\Omega)},
\]
where we used \eqref{eq:poincare-layer} for the ${\bf u}$ term.

Next,
\[
\|(\Delta\xi)\,T{\bf u}\|_{L^2}
\le \frac{C}{s^2}\|{\bf u}\|_{L^2(\Omega_s)}
\le \frac{C}{s}\|\nabla{\bf u}\|_{L^2(\Omega)}.
\]
Finally, expand the commutator:
\[
\Delta(T{\bf u})-T\Delta{\bf u}
=
(\Delta T){\bf u}+2\sum_{j=1}^N (\partial_j T)(\partial_j{\bf u}),
\]
so (using boundedness of $\Delta T$ and $\nabla T$)
\[
\|\xi(\Delta(T{\bf u})-T\Delta{\bf u})\|_{L^2}
\le
C\|{\bf u}\|_{L^2(\Omega_s)}+C\|\nabla{\bf u}\|_{L^2(\Omega_s)}
\le
Cs\|\nabla{\bf u}\|_{L^2(\Omega)}+C\|\nabla{\bf u}\|_{L^2(\Omega)}.
\]
Collecting the three bounds yields
\[
\|R_2\|_{L^2(\Omega)}
\le
C\left(1+\frac{1}{s}\right)\|\nabla{\bf u}\|_{L^2(\Omega)}.
\]
Since 
$ \left(1+\frac{1}{s}\right)  
\le \frac{2}{\min\{1,s \}}$ and 
$R_1=-R_2$, \eqref{eq:Rbound-s} follows.
\end{proof}

\begin{lemma}[Orthogonality; part of Lemma 4 in \cite{LiuLiuPego2007}]\label{lem:orth}
With ${\bf a}:=\nabla p_S$ and ${\bf b}:=\Delta{\bf u}_\parallel$,
\begin{equation}\label{eq:orth}
\langle{\bf a},{\bf a}-{\bf b}\rangle_{L^2(\Omega)}=0.
\end{equation}
\end{lemma}

\begin{lemma}[Boundary-layer estimate; rewriting of\ Lemma 2 in \cite{LiuLiuPego2007}]\label{lem:layer}
There exists $C_0>0$ (depending only on $\Omega$) such that for ${\bf a}:=\nabla p_S$,
writing ${\bf a}_\perp=({\bf n}{\bf n}^T){\bf a}$ and
${\bf a}_\parallel=(I-{\bf n}{\bf n}^T){\bf a}$ on $\Omega_s$,
\begin{equation}\label{eq:layer}
\int_{\Omega_s}\bigl(|{\bf a}_\perp|^2-|{\bf a}_\parallel|^2\bigr)
\ge - C_0 s\int_\Omega |{\bf a}|^2 .
\end{equation}
\end{lemma}

\begin{proof}[Proof of Theorem \ref{refinedTheorem1LiuLiuPego}]
We start from the identity (cf.\ equation (59) in \cite{LiuLiuPego2007})
\begin{equation}\label{eq:59}
\|\Delta{\bf u}\|^2
=
\|\Delta{\bf u}_\perp\|^2
+2(\Delta{\bf u}_\perp,\Delta{\bf u}_\parallel)
+\|\Delta{\bf u}_\parallel-\nabla p_S\|^2
+\|\nabla p_S\|^2 .
\end{equation}
We estimate the two terms
$(\Delta{\bf u}_\perp,\Delta{\bf u}_\parallel)$ and
$\|\Delta{\bf u}_\parallel-\nabla p_S\|^2$.

\medskip
\noindent
\textbf{Claim 1 (cross term).}
For every $\varepsilon_1>0$,
\begin{equation}\label{eq:claim1}
(\Delta{\bf u}_\perp,\Delta{\bf u}_\parallel)
\ge
-\varepsilon_1\|\Delta{\bf u}\|^2
- C_1(\varepsilon_1)\|\nabla{\bf u}\|^2,
\qquad
C_1(\varepsilon_1):=\frac{C_R^2}{2}\left(1+\frac{1}{\varepsilon_1}\right).
\end{equation}

\noindent
\emph{Proof of Claim 1.}
Let
\[
A:=\xi{\bf n}{\bf n}^T\Delta{\bf u}+(1-\xi)\Delta{\bf u},
\qquad
B:=\xi(I-{\bf n}{\bf n}^T)\Delta{\bf u}.
\]
Then by \eqref{eq:Delta-u-par} and \eqref{eq:Delta-u-perp}, $\Delta{\bf u}_\perp=A+R_1$ and $\Delta{\bf u}_\parallel=B+R_2$.
Hence
\begin{equation}\label{eq:expand1}
(\Delta{\bf u}_\perp,\Delta{\bf u}_\parallel)=(A,B)+(A,R_2)+(R_1,B)+(R_1,R_2).
\end{equation}
Using ${\bf n}{\bf n}^T(I-{\bf n}{\bf n}^T)=0$ and $0\le\xi\le 1$,
\[
A\cdot B
=
\xi(1-\xi)\,\Delta{\bf u}\cdot (I-{\bf n}{\bf n}^T)\Delta{\bf u}
=
\xi(1-\xi)\bigl|(I-{\bf n}{\bf n}^T)\Delta{\bf u}\bigr|^2\ge 0,
\]
so $(A,B)\ge 0$. Also $\|A\|\le\|\Delta{\bf u}\|$ and $\|B\|\le\|\Delta{\bf u}\|$.
By Cauchy-Schwarz and Young,
\[
|(A,R_2)|
\le \|A\|\,\|R_2\|
\le \|\Delta{\bf u}\|\,\|R_2\|
\le \frac{\varepsilon_1}{2}\|\Delta{\bf u}\|^2+\frac{1}{2\varepsilon_1}\|R_2\|^2,
\]
and similarly
\[
|(R_1,B)|\le \frac{\varepsilon_1}{2}\|\Delta{\bf u}\|^2+\frac{1}{2\varepsilon_1}\|R_1\|^2.
\]
Using \eqref{eq:Rbound-s}, we obtain
\[
|(R_1,R_2)|\le \|R_1\|\,\|R_2\|
\le \frac{(\|R_1\|+\|R_2\|)^2}{2}
\le \frac{C_R^2}{2} \|\nabla{\bf u}\|^2,
\]
and
\[
\frac{1}{2\varepsilon_1}(\|R_1\|^2+\|R_2\|^2)
\le
\frac{1}{2\varepsilon_1}(\|R_1\|+\|R_2\|)^2
\le 
 \frac{C_R^2}{2\varepsilon_1}\|\nabla{\bf u}\|^2.
\]
Inserting these bounds into \eqref{eq:expand1} yields \eqref{eq:claim1}.
\hfill$\square$

\medskip
\noindent
\textbf{Claim 2 (parallel term).}
For every $\varepsilon_2>0$,
\begin{equation}\label{eq:claim2}
\|\Delta{\bf u}_\parallel-\nabla p_S\|^2
\ge
(1-\varepsilon_2-2C_0 s)\|\nabla p_S\|^2
- C_2(\varepsilon_2)\|\nabla{\bf u}\|^2,
\qquad
C_2(\varepsilon_2):=\frac{2C_R^2}{\varepsilon_2}.
\end{equation}

\noindent
\emph{Proof of Claim 2.}
Set ${\bf a}:=\nabla p_S$ and ${\bf b}:=\Delta{\bf u}_\parallel$.
On $\Omega_s$ decompose
\[
{\bf a}_\perp=({\bf n}{\bf n}^T){\bf a},\quad {\bf a}_\parallel=(I-{\bf n}{\bf n}^T){\bf a},\qquad
{\bf b}_\perp=({\bf n}{\bf n}^T){\bf b},\quad {\bf b}_\parallel=(I-{\bf n}{\bf n}^T){\bf b}.
\]
Since ${\bf u}_\parallel$ (hence ${\bf b}$) is supported in $\Omega_s$, we have
${\bf b}=0$ on $\Omega_s^c$. Therefore,
\begin{align}
\|{\bf a}-{\bf b}\|^2
=\int_{\Omega_s^c}|{\bf a}-{\bf b}|^2+\int_{\Omega_s}|{\bf a}-{\bf b}|^2 
=\int_{\Omega_s^c}|{\bf a}|^2+\int_{\Omega_s}|{\bf a}_\perp-{\bf b}_\perp|^2
+\int_{\Omega_s}|{\bf a}_\parallel-{\bf b}_\parallel|^2 .
\label{eq:split}
\end{align}

\emph{(i) Lower bound for the $\perp$-term.}
Expand and use Young:
\begin{align}
\int_{\Omega_s}|{\bf a}_\perp-{\bf b}_\perp|^2
&=\int_{\Omega_s}\bigl(|{\bf a}_\perp|^2-2{\bf a}_\perp\cdot{\bf b}_\perp+|{\bf b}_\perp|^2\bigr) \notag\\
&\ge \int_{\Omega_s}\bigl(|{\bf a}_\perp|^2-2{\bf a}_\perp\cdot{\bf b}_\perp\bigr) 
\ge (1-\varepsilon_2)\int_{\Omega_s}|{\bf a}_\perp|^2
-\frac{1}{\varepsilon_2}\int_{\Omega_s}|{\bf b}_\perp|^2 .
\label{eq:perp}
\end{align}

\emph{(ii) Lower bound for the $\parallel$-term using orthogonality.}
By Lemma \ref{lem:orth},
\[
0=\int_\Omega {\bf a}\cdot({\bf a}-{\bf b})
=\int_{\Omega_s^c}|{\bf a}|^2+\int_{\Omega_s}{\bf a}_\perp\cdot({\bf a}_\perp-{\bf b}_\perp)
+\int_{\Omega_s}{\bf a}_\parallel\cdot({\bf a}_\parallel-{\bf b}_\parallel).
\]
Hence
\begin{align}
-2\int_{\Omega_s}{\bf a}_\parallel\cdot({\bf a}_\parallel-{\bf b}_\parallel)
&=2\int_{\Omega_s^c}|{\bf a}|^2+2\int_{\Omega_s}{\bf a}_\perp\cdot({\bf a}_\perp-{\bf b}_\perp).
\label{eq:key-ortho}
\end{align}
Now use the identity
\[
|{\bf a}_\parallel-{\bf b}_\parallel|^2+|{\bf a}_\parallel|^2
=2|{\bf a}_\parallel|^2-2{\bf a}_\parallel\cdot{\bf b}_\parallel+|{\bf b}_\parallel|^2
\ge -2{\bf a}_\parallel\cdot({\bf a}_\parallel-{\bf b}_\parallel),
\]
and integrate over $\Omega_s$; combining with \eqref{eq:key-ortho} gives
\begin{align}
\int_{\Omega_s}\bigl(|{\bf a}_\parallel-{\bf b}_\parallel|^2+|{\bf a}_\parallel|^2\bigr)
&\ge 2\int_{\Omega_s^c}|{\bf a}|^2+2\int_{\Omega_s}{\bf a}_\perp\cdot({\bf a}_\perp-{\bf b}_\perp).
\label{eq:par-start}
\end{align}
Next, estimate the last term by Young:
\begin{align}
2{\bf a}_\perp\cdot({\bf a}_\perp-{\bf b}_\perp)
=2|{\bf a}_\perp|^2-2{\bf a}_\perp\cdot{\bf b}_\perp 
\ge (2-\varepsilon_2)|{\bf a}_\perp|^2-\frac{1}{\varepsilon_2}|{\bf b}_\perp|^2 .
\label{eq:young2}
\end{align}
Insert \eqref{eq:young2} into \eqref{eq:par-start}, 
 subtract $\int_{\Omega_s}|{\bf a}_\parallel|^2$ from both sides,  
 split $\int_{\Omega_s}|{\bf a}_\parallel|^2$ to $(1-\varepsilon_2) \int_{\Omega_s}|{\bf a}_\parallel|^2  + (2-\varepsilon_2)\int_{\Omega_s}-|{\bf a}_{\parallel}|^2$,
 and finally drop $2\int_{\Omega_s^c}|{\bf a}|^2$ 
  to obtain
\begin{align}
\int_{\Omega_s}|{\bf a}_\parallel-{\bf b}_\parallel|^2
&\ge
2\int_{\Omega_s^c}|{\bf a}|^2
+(2-\varepsilon_2)\int_{\Omega_s}|{\bf a}_\perp|^2
-\frac{1}{\varepsilon_2}\int_{\Omega_s}|{\bf b}_\perp|^2
-\int_{\Omega_s}|{\bf a}_\parallel|^2 \notag\\
&\ge 
(1-\varepsilon_2) \int_{\Omega_s}|{\bf a}_\parallel|^2 
+(2-\varepsilon_2)\int_{\Omega_s}|{\bf a}_\perp|^2
    -|{\bf a}_{\parallel}|^2
-\frac{1}{\varepsilon_2}\int_{\Omega_s}|{\bf b}_\perp|^2.
\label{eq:par}
\end{align}

\emph{(iii) Combine.}
Substitute \eqref{eq:perp} and \eqref{eq:par} into \eqref{eq:split}:
\begin{align}
\|{\bf a}-{\bf b}\|^2
&\ge
\int_{\Omega_s^c}|{\bf a}|^2
+(1-\varepsilon_2)\int_{\Omega_s}|{\bf a}_\perp|^2
-\frac{1}{\varepsilon_2}\int_{\Omega_s}|{\bf b}_\perp|^2 \notag\\
& 
\quad + (1-\varepsilon_2) \int_{\Omega_s}|{\bf a}_\parallel|^2 
+(2-\varepsilon_2)\int_{\Omega_s}|{\bf a}_\perp|^2
    -|{\bf a}_{\parallel}|^2
-\frac{1}{\varepsilon_2}\int_{\Omega_s}|{\bf b}_\perp|^2.\notag\\
&\ge 
(1-\varepsilon_2) \int_{\Omega}|{\bf a}|^2
+(2-\varepsilon_2)\int_{\Omega_s}|{\bf a}_\perp|^2
      -|{\bf a}_{\parallel}|^2
-\frac{2}{\varepsilon_2}\int_{\Omega_s}|{\bf b}_\perp|^2 \notag\\
&\ge 
(1-\varepsilon_2) \int_{\Omega}|{\bf a}|^2
-2C_0 s \int_{\Omega_s} |{\bf a}|^2
-\frac{2}{\varepsilon_2}\int_{\Omega}|{\bf b}_\perp|^2,
\label{eq:combine-raw}
\end{align}
where Lemma\,\ref{lem:layer} and $2-\varepsilon_2<2$ are used in the last step.
This gives the
clean estimate (their (71))
\[
\|{\bf a}-{\bf b}\|^2
\ge (1-\varepsilon_2-2C_0 s)\|{\bf a}\|^2
-\frac{2}{\varepsilon_2}\int_{\Omega_s}|{\bf b}_\perp|^2.
\]

Finally, we bound $\int_{\Omega_s}|{\bf b}_\perp|^2$ using the structure of ${\bf b}$.
From \eqref{eq:Delta-u-par}, ${\bf b}=\xi(I-{\bf n}{\bf n}^T)\Delta{\bf u}+R_2$, so
${\bf b}_\perp=({\bf n}{\bf n}^T){\bf b}=({\bf n}{\bf n}^T)R_2$ because
${\bf n}{\bf n}^T(I-{\bf n}{\bf n}^T)=0$. Hence,
\[
\int_{\Omega_s}|{\bf b}_\perp|^2 \le \|R_2\|^2 \le C_R^2\|\nabla{\bf u}\|^2.
\]
Therefore,
\[
\|\Delta{\bf u}_\parallel-\nabla p_S\|^2
\ge (1-\varepsilon_2-2C_0 s)\|\nabla p_S\|^2
-\frac{2C_R^2}{\varepsilon_2}\|\nabla{\bf u}\|^2,
\]
which is \eqref{eq:claim2}. \hfill$\square$

\medskip
\noindent
\textbf{Finish the theorem.}
Insert \eqref{eq:claim1} and \eqref{eq:claim2} into \eqref{eq:59}, and discard the nonnegative term
$\|\Delta{\bf u}_\perp\|^2$:
\[
\|\Delta{\bf u}\|^2
\ge
-2\varepsilon_1\|\Delta{\bf u}\|^2
+(1-\varepsilon_2-2C_0 s)\|\nabla p_S\|^2
+\|\nabla p_S\|^2
-\bigl(2C_1(\varepsilon_1)+C_2(\varepsilon_2)\bigr)\|\nabla{\bf u}\|^2.
\]
Thus
\begin{equation}\label{eq:62}
(1+2\varepsilon_1)\|\Delta{\bf u}\|^2
\ge
(1+\beta_1)\|\nabla p_S\|^2
-\bigl(2C_1(\varepsilon_1)+C_2(\varepsilon_2)\bigr)\|\nabla{\bf u}\|^2,
\qquad
\beta_1:=1-\varepsilon_2-2C_0 s.
\end{equation}
Divide by $1+\beta_1$:
\begin{align}
\|\nabla p_S\|^2
\le
\frac{1+2\varepsilon_1}{1+\beta_1}\|\Delta{\bf u}\|^2
+\frac{2C_1(\varepsilon_1)+C_2(\varepsilon_2)}{1+\beta_1}\|\nabla{\bf u}\|^2.
\label{Liu001}
\end{align}

We choose $\varepsilon_2>0$ and $s>0$ sufficiently small such that $0<\beta_1<1$.  Since $\varepsilon_1>0$,
$\frac{1+2\varepsilon_1}{1+\beta_1} > \frac{1+0}{1+1}=\frac{1}{2}$. But when $\beta_1\to 1$ and $\varepsilon_1\to 0+$, the ratio approaches $\frac{1}{2}$. Therefore, the ratio $\frac{1+2\varepsilon_1}{1+\beta_1}$ can be expressed as $\frac{1}{2} + \varepsilon$ for some $\varepsilon>0$. Then the final question is how the parameters $\varepsilon_1$, $\varepsilon_2$, and $s$ depend on $\varepsilon$ when our goal is 
\begin{align}
\frac{1+2\varepsilon_1}{1+\beta_1}\le \frac12+\varepsilon.
\label{Liu_goal}
\end{align}
First, observe that $\frac{1+2\varepsilon_1}{2} < \frac{1+2\varepsilon}{1+\beta_1}$. Combined with \eqref{Liu_goal}, this yields 
\begin{align}
\varepsilon_1<\varepsilon.
\label{eps_1}
\end{align}
Second, note that $\frac{1}{1+\beta_1}<\frac{1+2\varepsilon}{1+\beta_1}$. Combining this with \eqref{Liu_goal} yields $\beta_1>\frac{1}{1/2+\varepsilon} -1$. Substituting $\beta_1=1-\varepsilon_2 - 2C_0s$ gives $\frac{2\varepsilon}{1/2+\varepsilon}>\varepsilon_2 + 2C_0s$. Since $\frac{2\varepsilon}{1/2} > \frac{2\varepsilon}{1/2+\varepsilon}$, it follows that $4\varepsilon>\varepsilon_2 + 2C_0s$. Because all these values are positive,  
\begin{align}
\varepsilon_2 < 4\varepsilon
\text{ and }
s<\frac{2\varepsilon}{C_0}.
\label{eps_2}
\end{align} 
From \eqref{eps_1} and \eqref{eps_2} (note $C_0$ only relies on $\Omega$ in Lemma\,\ref{lem:layer}), it follows that $\varepsilon_1$, $\varepsilon_2$, and $s$ must all tend to $0$ as $\varepsilon \to 0^+$.
For example, one can choose 
\begin{align}
\tilde{\varepsilon}=\min\{\varepsilon, 1\}, 
\quad
\varepsilon_1= \frac{\tilde\varepsilon}{4}, 
\quad
\varepsilon_2=\frac{\tilde\varepsilon}{2},
\quad
s=\frac{\tilde\varepsilon}{4C_0},
\label{Liu_choice}
\end{align}
then
it is easy to show that  
\begin{align}
\frac{1+2\varepsilon_1}{1+\beta_1}\le \frac12+\tilde\varepsilon  \le \frac{1}{2}+ \varepsilon.
\label{Liu002}
\end{align}
In this case, 
\begin{align}
C_1(\varepsilon_1)&=\frac{C_R^2}{2} (1+\frac{1}{\varepsilon_1})=\frac{C_R^2}{2} ( 1+\frac{4}{\tilde\varepsilon}) 
<\frac{2.5 C_R^2}{\tilde\varepsilon}
=\frac{2.5 C_R^2}{\min\{1,\varepsilon\}},
\label{C1_eps_upperbd}
\\
C_2(\varepsilon_2)&=\frac{2C_R^2}{\varepsilon_2}=\frac{4C_R^2}{\tilde\varepsilon}
=\frac{4C_R^2}{\min\{1,\varepsilon \}}.
\label{C2_eps_upperbd}
\end{align}
Finally, we obtain \eqref{eq:thm} from \eqref{Liu001}, \eqref{C1_eps_upperbd}, \eqref{C2_eps_upperbd} 
with  
\[
C_S(\varepsilon) := 2\frac{2.5 C_R^2}{\min\{1,\varepsilon \}}
 +C_2(\varepsilon_2) = \frac{9C_R^2}{\min\{1,\varepsilon \}}.
\]
Because $C_R= \frac{2C}{\min\{1,s \}}$ from \eqref{eq:Rbound-s} and $s=\frac{\min\{\varepsilon,1\}}{4C_0}$ from \eqref{Liu_choice}, 
\begin{align}
C_S(\varepsilon) = \frac{36 C^2}{\min\{1,\varepsilon \}  \cdot \left(\min\{1, \frac{\min\{1,\varepsilon \}}{4C_0} \} \right)^2 }.
\label{C_eps_arbi_eps}
\end{align}
When $\varepsilon<\min\{1, 4C_0\}$, it reduces to 
\begin{align}
C_S(\varepsilon)=\frac{576 C^2 C_0^2}{\varepsilon^3}.
\label{C_eps_small_eps}
\end{align}
Recall that both $C$ and $C_0$ only depends on $\Omega$. 

\end{proof}

\bibliographystyle{plain}
\bibliography{ref}

\end{document}